\documentclass[leqno]{amsart}
\usepackage{amsfonts,amssymb,amsmath,amsgen,amsthm}
\usepackage{hyperref}
\usepackage[dvips]{graphicx}

\theoremstyle{plain}
\newtheorem{theorem}{Theorem}[section]

\newtheorem{proposition}[theorem]{Proposition}
\theoremstyle{remark}
\newtheorem{remark}[theorem]{Remark}

\def\C{{\mathbf C}}
\def\R{{\mathbf R}}
\def\N{{\mathbf N}}
\def\T{{\mathbf T}}
\def\F{\mathcal F}
\def\O{\mathcal O}

\def\virgp{\raise 2pt\hbox{,}}

\def\bu{{\bf u}}

\def\({\left(}
\def\){\right)}
\def\<{\left\langle}
\def\>{\right\rangle}
\def\le{\leqslant}
\def\ge{\geqslant}

\def\Eq#1#2{\mathop{\sim}\limits_{#1\rightarrow#2}}
\def\Tend#1#2{\mathop{\longrightarrow}\limits_{#1\rightarrow#2}}

\def\d{{\partial}}
\def\l{\lambda}

\def\si{{\sigma}}
\def\eps{\varepsilon}

\DeclareMathOperator{\IM}{Im}
\DeclareMathOperator{\DIV}{\mathrm{div}}

\numberwithin{equation}{section}

\begin{document}

\title[NLS and semiclassical limit]{An asymptotic preserving
  approach for nonlinear Schr\"odinger equation
  in the semiclassical limit}
\author[R. Carles]{R\'emi Carles}
\email{Remi.Carles@math.cnrs.fr}
\author[B. Mohammadi]{Bijan Mohammadi}
\email{Bijan.Mohammadi@univ-montp2.fr}
\address{Univ. Montpellier~2\\Math\'ematiques
\\CC~051\\F-34095 Montpellier}
\address{CNRS, UMR 5149\\  F-34095 Montpellier\\ France}
\thanks{This work was supported by the French ANR project
  R.A.S. (ANR-08-JCJC-0124-01)}
\begin{abstract}
We study numerically the semiclassical limit for the nonlinear
Schr\"odinger equation thanks to a modification of the Madelung
transform due to E.~Grenier. This approach is naturally asymptotic
preserving, and allows for the presence of vacuum. Even if the mesh
size and the time step do not depend on the 
Planck constant, we recover the position and current densities in the
semiclassical limit, with a numerical rate of convergence in
accordance with the theoretical
results, before shocks appear in the limiting Euler
equation. By using simple projections, the mass and the momentum of
the solution are well preserved 
by the numerical scheme,
while the variation of the energy is not negligible
numerically. Experiments suggest that beyond the critical time for the
Euler equation, Grenier's approach yields smooth but highly
oscillatory terms. 
\end{abstract}
\maketitle

\section{Introduction}
\label{sec:intro}

We consider the cubic nonlinear equation
\begin{equation}
  \label{eq:NLS}
  i\eps\d_t u^\eps +\frac{\eps^2}{2}\Delta u^\eps = |u^\eps|^2
  u^\eps,\quad (t,x)\in \R_+\times \R^d.
\end{equation}
The goal is to compute the solution $u^\eps$ in such a way that for
$\eps=1$, we solve the nonlinear Schr\"odinger equation, and in the
semiclassical limit $\eps\to 0$, we retrieve the limit in terms of
compressible Euler equation, as recalled below. This equation appears
in several contexts in Physics. For instance, in the case $\eps=1$,
\eqref{eq:NLS} corresponds to an envelope equation in the propagation
of lasers, a case where $t$ does not correspond to time, but to the
direction of propagation; see e.g. \cite{Sulem} and references
therein. The semiclassical regime is present in the modeling of
Bose--Einstein condensation, where $\eps$ corresponds to the
(rescaled) Planck constant; see e.g. \cite{PS03} and references
therein. A remarkable property in the semiclassical regime is that the
limit is expressed in terms of a compressible, isentropic Euler
equation.
\smallbreak

A popular way to relate the semiclassical limit to fluid dynamics is
the use of the Madelung transform \cite{Madelung}, which is
essentially the polar decomposition: seek the solution to
\eqref{eq:NLS} of the form
\begin{equation*}
  u^\eps(t,x) = \sqrt{\rho(t,x)}e^{iS(t,x)/\eps},\quad
  \rho\ge 0,\quad S\in \R.
\end{equation*}
Plugging this expression into \eqref{eq:NLS}, and separating real and
imaginary parts yields
\begin{equation}
  \label{eq:separ}
  \left\{
    \begin{aligned}
      &\sqrt \rho \( \d_t S +\frac{1}{2}\lvert \nabla S\rvert^2 +
      \rho\)= \frac{\eps^2}{2}\Delta\(\sqrt \rho\),\\
& \d_t \sqrt \rho + \nabla S\cdot \nabla \sqrt \rho
      +\frac{1}{2}\sqrt \rho \Delta S =0.
    \end{aligned}
\right.
\end{equation}
Two comments are in order at this stage: the first equation shows that
$S$ depends on $\eps$ and the second equation shows that so does
$\rho$ in general. We shall underscore this fact by using the notation
$(S^\eps,\rho^\eps)$. Second, the equation for $S^\eps$ can be
simplified, \emph{provided that $\rho^\eps$ has no zero}. Introducing
the velocity $v^\eps = \nabla S^\eps$, \eqref{eq:separ} yields the
system of \emph{quantum hydrodynamics} (QHD), see also \cite{Ga94}:
\begin{equation}
  \label{eq:qhd}
  \left\{
    \begin{aligned}
      &\d_t v^\eps +v^\eps\cdot \nabla v^\eps +
      \nabla \rho^\eps= \frac{\eps^2}{2}\nabla\( \frac{\Delta\(\sqrt
      \rho^\eps\)}{\sqrt \rho^\eps}\),\\
& \d_t \rho^\eps  + \DIV \(\rho^\eps v^\eps\) =0.
    \end{aligned}
\right.
\end{equation}
The term on the right hand side of the equation for $v^\eps$ is
classically referred to as \emph{quantum pressure}.
In the limit $\eps\to 0$, this term disappears, and we find the
compressible Euler equation:
\begin{equation}
  \label{eq:euler}
  \left\{
    \begin{aligned}
      &\d_t v +v\cdot \nabla v +
      \nabla \rho= 0,\\
& \d_t \rho  + \DIV \(\rho v\) =0.
    \end{aligned}
\right.
\end{equation}
This approach was used recently to develop an asymptotic preserving
scheme for the linear Schr\"odinger equation ($|u^\eps|^2u^\eps$ is
replaced with $V(x)u^\eps$), see \cite{DGM07}. The goal of an
asymptotic preserving scheme is to have a unified way to compute the
solution as $\eps=1$, and to retrieve the limit as $\eps\to 0$, in
such a way that the discretization does not depend on $\eps$; see
e.g. \cite{Ji99,DJT08}. As pointed out in \cite{DGM07}, the drawback of
Madelung transform is that it does not support the 
presence of vacuum ($\rho=0$). The point of view that we shall study
numerically is due to E.~Grenier \cite{Grenier98}, and consists in
seeking $u^\eps$ as
\begin{equation}\label{eq:phaseampl}
  u^\eps(t,x) = a^\eps(t,x)e^{i\phi^\eps(t,x)/\eps},\quad
  a^\eps\in \C,\quad \phi^\eps\in \R.
\end{equation}
Allowing the amplitude $a^\eps$ to be complex-valued introduces an
extra degree of freedom, compared to the Madelung transform. The
choice of Grenier consists in imposing
\begin{equation}
  \label{eq:grenier}
  \left\{
    \begin{aligned}
      &\d_t \phi^\eps +\frac{1}{2}|\nabla \phi^\eps|^2 +|a^\eps|^2
      =0,\\
&\d_t a^\eps +\nabla \phi^\eps\cdot \nabla a^\eps +
\frac{1}{2}a^\eps\Delta \phi^\eps = i\frac{\eps}{2}\Delta a^\eps.
    \end{aligned}
\right.
\end{equation}
In terms of $v^\eps=\nabla \phi^\eps$, this becomes
\begin{equation}
  \label{eq:grenierv}
  \left\{
    \begin{aligned}
      &\d_t v^\eps +v^\eps\cdot \nabla v^\eps +\nabla |a^\eps|^2
      =0,\\
&\d_t a^\eps +v^\eps\cdot \nabla a^\eps +
\frac{1}{2}a^\eps\DIV v^\eps = i\frac{\eps}{2}\Delta a^\eps.
    \end{aligned}
\right.
\end{equation}
In this model, \emph{the
presence of vacuum ($a^\eps=0$) is not a problem}. We will see that
this is so both on a theoretical level and in computational tests. In the limit
$\eps\to 0$, we find formally
\begin{equation}
  \label{eq:eulersym}
  \left\{
    \begin{aligned}
      &\d_t v +v\cdot \nabla v +\nabla |a|^2
      =0,\\
&\d_t a +v \cdot \nabla a +
\frac{1}{2}a\DIV v =0.
    \end{aligned}
\right.
\end{equation}
We check that $(\rho,v)=(|a|^2,v)$ then solves \eqref{eq:euler}:
\eqref{eq:eulersym} corresponds to the nonlinear symmetrization of
\eqref{eq:euler} (\cite{MUK86,JYC90}).
\smallbreak

In this paper, we have chosen to focus on the defocusing cubic nonlinearity, for
which the relevance of \eqref{eq:grenier} to study the semiclassical
limit is proved (see \S\ref{sec:known}). It seems very likely that
equivalent numerical results should be available for other
nonlinearities, as discussed in  \S\ref{sec:other}, even though in
several cases, no theoretical result is available concerning the
natural generalization \eqref{eq:greniervgen} of
\eqref{eq:grenierv}. Similarly, in the
linear setting considered in \cite{DGM07}, this modified Madelung
transformation should overcome the problem of vacuum pointed out in
\cite{DGM07}.
\smallbreak

We also stress the fact that the convergence of \eqref{eq:grenierv}
towards \eqref{eq:eulersym} holds so long as no singularity has
appeared in the solution of \eqref{eq:eulersym} (or, equivalently, in
\eqref{eq:euler}). Note that except in the very specific case $d=1$
(where the cubic Schr\"odinger equation is completely integrable), no
analytical result seems to be available concerning the asymptotic
behavior of $u^\eps$ as $\eps\to 0$ for large time (that is, after a
singularity has formed in the solution to the Euler equation). As
pointed out in \cite{CaBook}, the notion of caustic seems to be
different in the case of \eqref{eq:NLS}, compared to the linear case
\begin{equation*}
  i\eps\d_t \psi^\eps +\frac{\eps^2}{2}\Delta \psi^\eps =
  V(x)\psi^\eps, 
\end{equation*}
where several computational results are available past caustics (see
e.g. \cite{Go02,Go05} and references therein, and \S\ref{sec:open}). 

\subsection{Conserved quantities}
\label{sec:conserve}

Equation~\eqref{eq:NLS} enjoys a bi-Hamiltonian structure, and
therefore has two quantities which are independent of time:
\begin{align}
 \label{eq:mass}
\text{Mass: }& \frac{d}{dt}\|u^\eps(t)\|_{L^2(\R^d)}^2=0.\\
\label{eq:energy}
\text{Energy: }& \frac{d}{dt}\(\|\eps\nabla u^\eps(t)\|_{L^2(\R^d)}^2
+ \|u^\eps(t)\|_{L^4(\R^d)}^4\)=0.
\end{align}
A third important quantity is conserved, which plays a crucial role,
e.g. in the study of finite time blow-up in the case of focusing
nonlinearities:
\begin{equation}\label{eq:momentum}
  \text{Momentum: }\frac{d}{dt}\IM\int_{\R^d}\overline u^\eps(t,x)\eps
  \nabla u^\eps(t,x)dx =0.
\end{equation}
Plugging the phase/amplitude representation \eqref{eq:phaseampl} into
these conservation laws, and passing formally to the limit $\eps\to
0$, we recover conservation laws associated to the Euler equation
\eqref{eq:euler} (\cite{CaBook}):
\begin{align*}
  \frac{d}{dt}\int_{\R^d}\rho(t,x)dx =\frac{d}{dt}\int_{\R^d} \(\rho |v|^2 +
  \rho^2\)(t,x)dx= \frac{d}{dt}\int_{\R^d} \(\rho v\)(t,x)dx=0.
\end{align*}
Setting $J^\eps(t)=x+i\eps t\nabla$, two other evolution laws are
available:
\begin{align*}
  \text{Pseudo-conformal: }&  \frac{d}{dt}\left(\|J^\varepsilon(t)
u^\varepsilon\|_{L^2(\R^d)}^2
+ t^2\|u^\varepsilon\|_{L^{4}(\R^d)}^{4}
\right)=t (2-d)\|u^\varepsilon\|_{L^{4}(\R^d)}^{4}.\\
& \frac{d}{dt}\operatorname{Re}\int_{\R^d} \overline u^\eps(t,x)
 J^\varepsilon(t)u^\varepsilon(t,x) dx=0 .
\end{align*}
Passing formally to the limit $\eps\to 0$, we infer:
\begin{align*}
  & \frac{d}{dt}\int_{\R^d} \left( \left| x -t v
        (t,x)\right|^2 \rho(t,x) +t^2 \rho^2(t,x)
\right)dx
=(2-d)t \int_{\R^d} \rho^2(t,x)dx.\\
& \frac{d}{dt}\int_{\R^d} \left(x -t v(t,x)\right)\rho(t,x) dx =0.
\end{align*}
We discuss this aspect further into details in \S\ref{sec:conserve2}.
\subsection{Semiclassical limit for NLS: numerical approach}
\label{sec:seminum}
The most reliable approach so far to study numerically the
semiclassical limit for Schr\"odinger equations seems to be the
time-splitting spectral discretization (Lie or Strang splitting, see
\cite{BBD02}): one solves alternatively two linear equations,
\begin{equation*}
  i\eps\d_t v^\eps+\frac{\eps^2}{2}\Delta v^\eps =0 ,\quad
\text{and} \quad i\eps\d_t v^\eps = |v^\eps|^2 v^\eps.
\end{equation*}
Despite the appearance, the second equation is linear, since in view
of the gauge invariance, $\d_t \(|v^\eps|^2\)=0$, so the second equation
boils down to $i\d_t v^\eps = |v_I^\eps|^2 v^\eps$, where $v_I^\eps$
denotes the initial value for $v^\eps$.
\smallbreak

Note that from \cite{MPP99}, usual
finite-difference schemes for the linear Schr\"odinger equation may
lead to very
wrong approximations. Instead, schemes based on the fast Fourier
transform (FFT) have been preferred. In \cite{BJM02},  it was shown
that the time-splitting method, coupled with a
trigonometric spectral approximation of the spatial
derivative,   conserves the total mass, and  is gauge-invariant,
time-reversible. Moreover, with this approach, the convergence of the
scheme in $L^2$ is proved, when the nonlinearity in \eqref{eq:NLS} is
replaced by an external potential. This regime turns out to be far
less singular in the limit $\eps\to 0$ than the nonlinear case of
\eqref{eq:NLS}, as discussed below.
\smallbreak
We briefly point out that the numerical
study in \cite{BJM02,BJM03} shows that, contrary to the case of the
linear Schr\"odinger equation,  to study the semiclassical
limit for \eqref{eq:NLS} with time-splitting, it is necessary to
consider mesh sizes and time steps which are $\O(\eps)$. This is due
to the fact that the semiclassical regime is strongly nonlinear
(supercritical, in the terminology of \cite{CaBook}): we consider
initial data which are $\O(1)$ in $L^2\cap L^\infty$, and there is no
power of $\eps$ in front of the nonlinearity. As a consequence, the
semiclassical limit is a ``strongly nonlinear'' process, since
starting with a semilinear Schr\"odinger equation (for fixed
$\eps>0$), we come up in the limit $\eps\to 0$ with a quasilinear
equation (the compressible Euler equation). 

\smallbreak
In \cite{BJM03}, it is shown that mesh sizes and time steps must be
taken of order $\O(\eps)$, even to
recover the behavior of two
physically important quantities:
\begin{align*}
  \text{Position density: }& \rho^\eps(t,x) =
  |u^\eps(t,x)|^2=|a^\eps(t,x)|^2.\\
\text{Current density: }&J^\eps(t,x) = \eps\IM \(\overline
u^\eps(t,x)\nabla u^\eps(t,x)\).
\end{align*}
We refer to the numerical results in
\cite[Example~4.3]{BJM03}, which show some important instability in the
numerical approximation for \eqref{eq:NLS}, at least if the time step
is large compared to $\eps$: evidently, the position and current
densities cannot be computed correctly if mesh size and time step are
independent of $\eps$.
\smallbreak

On the contrary, we obtain a good description of $\rho^\eps$ and
$J^\eps$ as $\eps\to 0$ when studying numerically the system
\eqref{eq:grenierv}, even if the time step is independent of
$\eps$. Things would probably be similar in the case of the QHD system
\eqref{eq:qhd}, up to the important aspect that the presence of vacuum
($\rho^\eps=0$) is not allowed in \eqref{eq:qhd}. The idea to explain
this difference is the following. To construct directly the wave
function $u^\eps$ solving \eqref{eq:NLS}, errors which are large
compared to $\eps$ (say of order $\eps^{\alpha}$, $0<\alpha<1$) lead
to instability of order $\O(1)$ on $u^\eps$ after a short time (of
order $\eps^{1-\alpha}$). Among possible sources of errors, we can
mention a simple space shift, which is rather likely to occur in
numerical studies. This can actually be proved thanks to the
approach of Grenier, see \cite{CaARMA}. This is due to the strong
coupling phase/amplitude in \eqref{eq:grenier}: a small modification
of the amplitude $a^\eps$ leads to a modification of the same order
for $\phi^\eps$. To recover $u^\eps$, one has to divide $\phi^\eps$ by
$\eps$, which is small, so the actual error for $\phi^\eps$ may be
dramatically increased.
\smallbreak

One can rephrase the above analysis as follows. The semiclassical
limit for \eqref{eq:NLS} is
``strongly nonlinear'': as $\eps\to 0$, we pass from a semilinear
equation (for fixed $\eps$, the Cauchy problem for \eqref{eq:NLS} is
handled by perturbative methods relying on properties of the linear
equation, see e.g. \cite{CazCourant}), to a quasilinear one, the Euler
equation \eqref{eq:euler}
(in which the nonlinear terms cannot be treated by perturbative
methods, see e.g. \cite{Taylor3}). As a consequence, the asymptotic
behavior of $u^\eps$ is very sensitive to small errors \cite{CaARMA}.
In time splitting methods, one considers the nonlinearity as a
perturbation, while this is not sensible in the framework of
\eqref{eq:NLS}, unless a high precision in the space and time steps is
demanded. It would be quite different with some positive
power --- at least $1$ --- of $\eps$ in front of the nonlinearity; see
\cite{CaBook} for theoretical explanations, and \cite{BJM02,BJM03} for
numerical illustrations.
\smallbreak

If one is interested only in the position and current densities, small
errors in \eqref{eq:grenierv} are not so important, since one never
has to divide the phase by $\eps$ (see Section~\ref{sec:background} for more
details). This explains why we can obtain satisfactory results by
considering a mesh size $h=\Delta x$ independent of $\eps$, and a time
step given by the parabolic scaling, that is, proportional to $h^2$.

\smallbreak

An extra step in the numerical analysis of nonlinear Schr\"odinger
equations was achieved in \cite{Be04}, where a semi-discrete scheme
was introduced, which turns NLS into an almost linear system, in the
case $\eps=1$.
It is based on a central-difference approximation shifted by a half
time-step. For $t_n= n \delta t$ and $t_{n+1/2} = (n + \frac 1 2) \delta
t$, let $u^n$ be the approximation at $t=t_n$. The scheme is
given by
\begin{equation*}
\left \{
  \begin{aligned}
i \frac {u^{n+1}-u^n}{\delta t}  +\frac{1}{2}\Delta
\left( \frac {u^{n+1}+u^n}{2} \right)
& = \psi^{n+1/2} \( \frac {u^{n+1}+u^n}{2} \), \\
 \frac {\psi^{n+1/2}+\psi^{n-1/2}}{2} & =  |u^n|^{2}.
  \end{aligned}
\right .
\end{equation*}
This approach has the advantage of preserving the mass \eqref{eq:mass}
and an analogue of the energy \eqref{eq:energy} of the solution
\cite{Be04}:
\begin{align*}
  \int_{\R^d}|u^n|^2  & = \int_{\R^d}|u^0|^2 ,\quad \text{and}\quad E^n=
  E^0,\text{ where}\\
E^n&=\int_{\R^d}\(|\nabla u^n|^2 +2|u^n|^2
\psi^{n-1/2}-\(\psi^{n-1/2}\)^2\).
\end{align*}
It does not seem that there is also an analogue of the momentum which
is conserved, in the same fashion as \eqref{eq:momentum}. 
Note that to adapt this approach numerically in the semiclassical
regime, one would also have to consider mesh sizes and time steps which
are $\O(\eps)$. Therefore, the approaches in \cite{BJM02,BJM03,Be04} do
not seem well suited for asymptotic preserving schemes.
\smallbreak

In this paper, we present numerical experiments only, and do not claim
to justify the approach by numerical analysis arguments. In view of
the little knowledge that we have on the behavior of the solution to
\eqref{eq:grenierv} past the critical time for the Euler equation,
such a study could reasonably be expected only so long as the solution of
the Euler equation remains smooth. Yet, such a study would be an
interesting challenge, which we do not address here.

\subsection{Outline of the paper}

In Section~\ref{sec:background}, we recall the main theoretical results
established for the semiclassical analysis of \eqref{eq:NLS}. The
main goal is to state some results which can thereafter be tested
numerically to validate the scheme. The numerical implementation is
presented in Section~\ref{sec:scheme}. Numerical experiments (based on
three examples) are discussed in Section~\ref{sec:exp}. We conclude
the paper in Section~\ref{sec:concl}.

\section{The theoretical point of view}
\label{sec:background}

We will always consider initial data of the form
\begin{equation}\label{eq:ci}
  u^\eps(0,x) = a_0(x)e^{i\phi_0(x)/\eps},\quad a_0\in \C,\quad
  \phi_0\in \R,
\end{equation}
where $a_0$ and $\phi_0$ are smooth, say in $H^s(\R^d)$ for all
$s$. In that case, \eqref{eq:grenier} is supplemented with the Cauchy
data
\begin{equation*}
  a^\eps(0,x)=a_0(x)\quad ;\quad \phi^\eps(0,x)=\phi_0(x).
\end{equation*}
This implies that the Cauchy data for \eqref{eq:grenierv} are
\begin{equation}\label{eq:CIGrenier}
  a^\eps(0,x)=a_0(x)\quad ;\quad v^\eps(0,x)=\nabla\phi_0(x).
\end{equation}

\subsection{Known results}
\label{sec:known}
A second advantage of the system \eqref{eq:grenierv} over
\eqref{eq:qhd}, besides the role of vacuum, is that it already has the
form of an hyperbolic symmetric system. Separate real and
imaginary parts of $a^\eps$, $a^\eps =
a_1^\eps + ia_2^\eps$, \eqref{eq:grenierv} takes the form
\begin{equation*}
  \partial_t \bu^\eps +\sum_{j=1}^n
  A_j(\bu^\eps)\partial_j \bu^\eps
  = \frac{\eps}{2} L
  \bu^\eps\, ,
\end{equation*}
\begin{equation*}
  \text{with}\quad \bu^\eps = \left(
    \begin{array}[l]{c}
       a_1^\eps \\
       a_2^\eps \\
       v^\eps_1 \\
      \vdots \\
       v^\eps_d
    \end{array}
\right) , \quad L = \left(
  \begin{array}[l]{ccccc}
   0  &-\Delta &0& \dots & 0   \\
   \Delta  & 0 &0& \dots & 0  \\
   0& 0 &&0_{d\times d}& \\
   \end{array}
\right),
\end{equation*}
\begin{equation*}
  \text{and}\quad A(\bu,\xi)=\sum_{j=1}^d A_j(\bu)\xi_j
= \left(
    \begin{array}[l]{ccc}
      v\cdot \xi & 0& \frac{a_1 }{2}\,^{t}\xi \\
     0 &  v\cdot \xi & \frac{a_2}{2}\,^{t}\xi \\
     2  a_1 \, \xi
     &2  a_2\, \xi &  v\cdot \xi I_d
    \end{array}
\right).
\end{equation*}
The matrix $A$ is symmetrized by a constant diagonal matrix $S$ such
that $SL=L$. We note that $L$ is skew-symmetric, so it plays no role
in energy estimates in Sobolev spaces $H^s(\R^d)$. The main results in
\cite{Grenier98} can be summarized as follows:
\begin{theorem}[From \cite{Grenier98}]\label{theo:grenier}
  Let $d\ge 1$, $s>4+d/2$, and $a_0,\nabla\phi_0\in H^s(\R^d)$. \\
$1.$ There exist $T>0$ and a unique solution $(\rho,v)\in
C([0,T];H^s(\R^d))^2$ to \eqref{eq:euler} such that
$\rho(0,x)=|a_0(x)|^2$ and $v(0,x)=\nabla \phi_0(x)$. \\
$2.$ For the same $T$, \eqref{eq:grenierv} has a unique
solution $(a^\eps,v^\eps)\in C([0,T];H^{s-2}(\R^d))^2$ such that
$a^\eps(0,x)=a_0(x)$ and
$v^\eps(0,x)=\nabla \phi_0(x)$. \\
$2'.$ For the same $T$, \eqref{eq:eulersym} has a unique
solution $(a,v)\in C([0,T];H^{s}(\R^d))^2$ such that
$a(0,x)=a_0(x)$ and
$v(0,x)=\nabla \phi_0(x)$. \\
$3.$ As $\eps\to 0$, we have:
\begin{equation*}
  \|a^\eps-a\|_{L^\infty([0,T];H^{s-2})} +
  \|v^\eps-v\|_{L^\infty([0,T];H^{s-2})} =\O\(\eps\).
\end{equation*}
\begin{equation*}
\end{equation*}
\end{theorem}
\begin{remark}[Periodic case]
  The same result holds in the periodic setting ($x\in \T^d$
  instead of $x\in \R^d$), with exactly the same proof.
\end{remark}
Once $v^\eps$ is constructed, there are at least two ways to get back
to $\phi^\eps$. Either argue that $v^\eps$ remains irrotational, or
simply define $\phi^\eps$ as
\begin{equation}\label{eq:backtophi}
  \phi^\eps(t)= \phi_0 -\int_0^t \(\frac{1}{2}|v^\eps(s)|^2 +\nabla
  |a^\eps(s)|^2\)ds,
\end{equation}
and check that $\d_t\( v^\eps-\nabla \phi^\eps\)= \d_t v^\eps -\nabla
\d_t \phi^\eps=0$.
So for $t\in [0,T]$, that is so long as the solution to the Euler
equation \eqref{eq:euler} remains smooth, the solution to
\eqref{eq:NLS} with initial data $u^\eps_{\mid t=0}= a_0
e^{i\phi_0/\eps}$ is given by $u^\eps =a^\eps e^{i\phi^\eps/\eps}$.
\smallbreak

Note that even if $\phi_0=0$, $\phi^\eps$ (as well as $v^\eps$ and
$v$) must not be
expected to be zero (nor even small), because of the strong coupling
in \eqref{eq:grenier}. Typically, if $\phi_0=0$, \eqref{eq:grenier}
yields $\d_t \phi^\eps_{\mid t=0} = -|a_0|^2\not =0$.

\smallbreak
In addition, in the limit $\eps\to 0$, we recover the main two
quadratic observables:
\begin{align*}
  \rho^\eps&=|a^\eps|^2\Tend \eps 0 |a|^2=\rho\text{ in
  }L^\infty([0,T];L^1(\R^d)),\\
 J^\eps &= \IM\(\eps \overline
  u^\eps \nabla u^\eps \) = |a^\eps|^2 v^\eps +\eps \IM \(\overline
  a^\eps \nabla a^\eps\)\Tend \eps 0 |a|^2 v=J\text{ in
  }L^\infty([0,T];L^1(\R^d)).
\end{align*}
We have more precisely:
\begin{equation}\label{eq:obsquad}
  \|\rho^\eps -\rho\|_{L^\infty([0,T];L^1\cap L^\infty)} +
  \|J^\eps-J\|_{L^\infty([0,T];L^1\cap L^\infty)}=\O\(\eps\).
\end{equation}
We can also prove the convergence of the wave function
(\cite{Grenier98}). In the particular case which we consider
where the initial amplitude $a^\eps(0,x)$ does not depend on $\eps$,
we have (with an obvious definition for $\phi$):
\begin{equation*}
  \|u^\eps - a e^{i\phi/\eps}\|_{L^\infty([0,T];L^2\cap
    L^\infty)}=\O(\eps).
\end{equation*}
In general, a modulation of $a$ must be taken into account to have
such an approximation of the wave function (\cite{CaBook}): $u^\eps
\approx ae^{i\phi^{(1)}}  e^{i\phi/\eps}$. In the framework of this
paper, we have $\phi^{(1)}=0$ (see \cite[Section~4.2]{CaBook}).

\subsection{An open question}
\label{sec:open}

As pointed out in the introduction, no analytical result seems to be
available concerning the semiclassical limit of \eqref{eq:NLS} when
the solution of the Euler equation \eqref{eq:euler} has become
singular.  
Theorem~\ref{theo:grenier} gives a rather complete picture for the
asymptotic behavior of $u^\eps$ for $t\in [0,T]$, that it before the
solution to \eqref{eq:euler} becomes singular. Note that if for
instance $a_0$ and $\phi_0$ are compactly supported, then no matter
how small they are, $(a,v)$ develops a singularity in finite time
(\cite{MUK86,JYC90,Xin98}). On the other hand, for fixed $\eps>0$, we
know that the solution to \eqref{eq:NLS} with initial data
$u^\eps_{\mid t=0}= a_0
e^{i\phi_0/\eps}\in H^s(\R^d)$, $s\ge 1$, is global in time with the
same regularity, at least if $d\le 4$: $u^\eps\in
C([0,\infty[;H^s(\R^d))$. See \cite{GV79Cauchy} (or \cite{CazCourant})
for the case $d\le 3$, and
\cite{RV07} for the case $d=4$ (which is energy-critical).
\smallbreak

A natural question is then: what happens to $u^\eps$ as the solution
to the Euler equation \eqref{eq:euler} becomes singular? In the linear
setting,
\begin{equation}\label{eq:schrodlin}
  i\eps\d_t u^\eps_{\rm lin} +\frac{\eps^2}{2}\Delta u^\eps_{\rm lin}
  = 0\quad ;\quad
  u^\eps_{{\rm lin}\mid t=0}=a_0 e^{i\phi_0/\eps},
\end{equation}
the question is rather well understood: when the solution to the
corresponding Burger's equation (for the phase) becomes singular, a
caustic is formed, which is a set in $(t,x)$-space (see
e.g. \cite{D74,MF81}). Near the caustic, the amplitude of $u^\eps_{\rm
  lin}$ is amplified, like a negative power of $\eps$. For instance,
if $\phi_0(x) = -|x|^2/2$, then
\begin{equation*}
  u^\eps_{\rm lin}(t,x)\Eq \eps 0 \left\{
    \begin{aligned}
      &\frac{1}{(1-t)^{d/2}}a_0\(\frac{x}{1-t}\)
      e^{i|x|^2/(2\eps(t-1))}\text{ if }t<1,\\
&\frac{1}{\eps^{d/2}}\widehat a_0\(\frac{x}{\eps}\)\text{ if }t=1,
    \end{aligned}
\right.
\end{equation*}
where $\widehat a_0$ denotes the Fourier transform of $a_0$; see
\cite{CaBook} for several developments around this example, and
\cite{CGM3AS} for corresponding numerical experiments. Such a
concentration is ruled out in the case of
\eqref{eq:NLS}, since the conservation of the energy \eqref{eq:energy}
yields the uniform bound
\begin{equation*}
  \|u^\eps(t)\|_{L^4(\R^d)} \le C \text{ independent of }\eps\in
  ]0,1]\text{ and } t\in \R.
\end{equation*}
We remark that multiplying each equation in \eqref{eq:grenierv},
derivatives become exactly $\eps$-derivatives: every time a term is
differentiated, it is multiplied by $\eps$. This is consistent with
the possibility that $a^\eps$ and $v^\eps$ become oscillatory past the
critical time for the Euler equation (with wavelength of order $\eps$
or more). The numerical experiments we
present below suggest that this is indeed the case. We insist on the
fact that no result is
available, though, on global existence aspects for
\eqref{eq:grenierv}: the solution may be globally smooth (and
$\eps$-oscillatory, in the sense of \cite{GMMP}), but it may blow up
in finite time. 
\smallbreak

Note however that
the approach we present here is no longer expected to be asymptotic
preserving beyond the breakup time for the Euler equation. The
presence of rapid oscillations is a possible explanation, and we then
recover the problem pointed out in \cite{BJM02} for pre-breakup times:
 rapid oscillations can be resolved only if time step and mesh
sizes are comparable to the (small) wavelength of the wave. Finally,
we point out that even in the linear case (see e.g. the above
example), one cannot expect an asymptotic  preserving approach to
solve  \eqref{eq:schrodlin} after a caustic has formed: near the
caustic, small spatial scales must be taken into account. In the above
example, the wave function is concentrated at scale $\eps$. A
possibility to get an aymptotically preserving 
approach in the linear case would be the use of Lagrangian integrals
\cite{D74}; see \cite{CaBook} for an extension in a very specific
nonlinear setting. Note however that the definition of the Lagrangian
integral depends on the initial phase, so this approach is more
delicate to implement numerically. The $K$-branch approach would lead
to similar requirements; see \cite{BC98,Go02,Go05}. 
\subsection{Conserved quantities}
\label{sec:conserve2}
 In the one-dimensional case $d=1$, the cubic nonlinear Schr\"odinger
 equation \eqref{eq:NLS} has infinitely many conserved quantities
 \cite{ZS73} (it is completely integrable, see \cite{ZM75}). We shall
 not emphasize this 
 particular case in this paper, and rather consider the case of a
 cubic nonlinearity in arbitrary dimension. Numerical experiments are
 presented in the two-dimensional case $d=2$.
\smallbreak

In this general case, we retain the three standard conservations:
mass \eqref{eq:mass}, momentum \eqref{eq:momentum}, and energy
\eqref{eq:energy}. Writing the solution to \eqref{eq:NLS} as $u^\eps =
a^\eps e^{i\phi^\eps/\eps}$, we infer three corresponding conversation
laws for the solution to \eqref{eq:grenierv}:

\begin{proposition}\label{prop:conservations}
  Let  $d\ge 1$ and $(a^\eps,v^\eps)\in
  C\([0,T];H^1\cap L^\infty(\R^d)\)^2$ solve \eqref{eq:grenierv}.  The following
  three quantities do not depend on time:\\
$(1)$ The $L^2$-norm of $a^\eps$:
$\displaystyle  \frac{d}{dt}\|a^\eps(t)\|_{L^2(\R^d)}^2 =0.$\\
$(2)$ The momentum: 
  $\displaystyle  \frac{d}{dt}\int_{\R^d} \(|a^\eps(t,x)|^2 v^\eps(t,x) +\eps \IM
    \(\overline a^\eps(t,x) \nabla a^\eps(t,x)\)\)dx=0.$\\
$(3)$ The energy: if $v^\eps_{\mid t=0}$ is irrotational, $\nabla \wedge
v^\eps_{\mid t=0}=0$, then 
\begin{equation*}
  \frac{d}{dt}\int_{\R^d} \(|\eps \nabla a^\eps(t,x) +ia^\eps(t,x)
v^\eps(t,x)|^2  +|a^\eps(t,x)|^4\)dx=0.
\end{equation*}
\end{proposition}
\begin{proof}[Sketch of proof]
 This result can be proved  by using the standard regularizing
procedure and suitable multipliers. We shall just indicate the formal
procedure.

The conservation of mass is proved by multiplying the second equation
in \eqref{eq:grenierv} by $\overline a^\eps$, integrating in space,
and taking the real value. 

The conservation of the momentum is obtained as follows. Multiply the
equation for $v^\eps$ by $\frac{1}{2}|a^\eps|^2$, and integrate in
space. Multiply the equation for $a^\eps$ by $i\eps\nabla \overline
a^\eps + \overline a^\eps v^\eps$, integrate in space and consider the
real value. Summing these two relations yields the conservation of the
momentum. 

For the energy, the procedure is similar. Note that 
 \begin{equation*}
  \d_t v^\eps = -\nabla\(\frac{|v^\eps|^2}{2}+|a^\eps|^2\),\quad
  \text{hence } \d_t \(\nabla \wedge v^\eps\)=0.
\end{equation*}
Therefore, if $\nabla \wedge
v^\eps_{\mid t=0}=0$, then we can find $\phi^\eps$ such that
$(\phi^\eps,a^\eps)$ solves \eqref{eq:grenier}. Multiply the equation
in $\phi^\eps$ by $-\frac{1}{2}\d_t|a^\eps|^2$, the equation for
$a^\eps$ by $i\eps\d_t 
\overline a^\eps +\overline a^\eps \d_t \phi^\eps$. Sum up the two
equations, integrate in space, and take the real part.  
\end{proof}

\subsection{About other nonlinearities}
\label{sec:other}

Equation~\eqref{eq:NLS} is the defocusing cubic nonlinear
Schr\"odinger equation. Other nonlinearities are physically relevant
too: focusing or defocusing nonlinearities are considered, as well as
other powers, in the context
of laser Physics (see e.g. \cite{Sulem}) or in the context of
Bose--Einstein Condensation (see e.g. \cite{DGPS,JosserandPomeau}),
for instance.
\smallbreak

The (short time) semiclassical limit for nonlinear Schr\"odinger
equations has been studied rigorously for other
nonlinearities. Typically, for defocusing nonlinearities
\begin{equation*}
  i\eps \d_t u^\eps +\frac{\eps^2}{2}\Delta u^\eps =
  |u^{\eps}|^{2\si}u^\eps, \quad \si\in \N,
\end{equation*}
a result similar to Theorem~\ref{theo:grenier} is available; see
\cite{ACARMA,ChironRousset}. However, the analysis does not rely on an
extension of \eqref{eq:grenier} where $|a^\eps|^2$ would be replaced
with $|a^\eps|^{2\si}$: for $\eps=0$ (corresponding to the limiting
Euler equation in the case $\si\in \N$ too), one uses a nonlinear
symmetrizer (the ``good'' unknown is $(\nabla \phi, a^\si)$), and for
$\eps>0$, this change of variable affects the skew-symmetric term
$i\eps \Delta a^\eps$ in such a way that apparently the analysis of
\cite{Grenier98} cannot be directly adapted.
\smallbreak

For focusing
nonlinearities, typically
\begin{equation*}
  i\eps \d_t u^\eps +\frac{\eps^2}{2}\Delta u^\eps =
 - |u^{\eps}|^{2\si}u^\eps, \quad \si\in \N,
\end{equation*}
the limiting equation in the system analogous to \eqref{eq:euler} is
\emph{elliptic} (as opposed the hyperbolic system
\eqref{eq:euler}). It turns out that in this case, the ``elliptic
Euler system'' is ill-posed in Sobolev spaces (\cite{GuyCauchy}):
working with analytic regularity becomes necessary \cite{GuyCauchy},
and sufficient
\cite{PGX93,ThomannAnalytic} in order to justify the semiclassical
analysis.
\smallbreak

An hybrid nonlinearity (neither focusing, nor defocusing) also plays a
role in physical models: the cubic--quintic nonlinearity,
\begin{equation*}
  i\eps \d_t u^\eps +\frac{\eps^2}{2}\Delta u^\eps =
  |u^{\eps}|^{4}u^\eps+\lambda |u^{\eps}|^{2}u^\eps,
\end{equation*}
with $\l\in \R$ possibly negative. This model is mostly used as an envelope
equation in optics, is also considered in BEC for alkalimetal gases (see
e.g. \cite{JPhysB,PhysRevA63,PhysRevE}), in which case $\lambda
<0$. The cubic term corresponds to
a negative scattering length, and the quintic term to a repulsive
three-body elastic interaction. Justifying the semiclassical analysis
 was achieved in \cite{ACIHP} by a slight modification of the approach of
 \cite{Grenier98} (in a different
functional framework).
\smallbreak

To rephrase the above discussion, the approach in \cite{Grenier98} to
study the semiclassical limit for
\begin{equation*}
 i\eps \d_t u^\eps +\frac{\eps^2}{2}\Delta u^\eps =
  f\(|u^{\eps}|^{2}\)u^\eps
\end{equation*}
relies on the assumption $f'>0$. However, the analysis has been
carried out in several other situations, without considering the
natural generalization of \eqref{eq:grenierv},
\begin{equation}
  \label{eq:greniervgen}
  \left\{
    \begin{aligned}
      &\d_t v^\eps +v^\eps\cdot \nabla v^\eps +\nabla f\(|a^\eps|^2\)
      =0,\\
&\d_t a^\eps +v^\eps\cdot \nabla a^\eps +
\frac{1}{2}a^\eps\DIV v^\eps = i\frac{\eps}{2}\Delta a^\eps.
    \end{aligned}
\right.
\end{equation}
It seems reasonable to believe that even though no rigorous study for
this system is available in general (for $\eps>0$), this system can be
used for numerical simulations.
\smallbreak

Finally, the approach of \cite{Grenier98} was generalized to the case
where an external potential is introduced (which may model a confining
trap in the framework of Bose--Einstein Condensation), see
\cite{CaBook}, and  to the case of Schr\"odinger--Poisson system
\cite{AC-SP,LiLinEJDE,LiuTadmor,Ma-p}.

\section{Numerical implementation}
\label{sec:scheme}

One expects the oscillatory nature of the solutions to be
difficult to capture numerically. We would like to use a stable
numerical scheme with the time step independent of $\eps$ but function
of $h$. The scheme solves  system \eqref{eq:grenier} on coarser
meshes than what necessary to capture all wavelengths. Therefore,
the solution has inevitably error in it. Still we give a great
deal of effort on conservation issues for the density, energy and
momentum. The time step and mesh size being both independent of
$\eps$, one can tackle very small $\eps$ values and the scheme
also works for $\eps=\O(1)$. Obviously, if all scales are aimed
at being captured, then the space grid size need be of order of $\eps$
or less and so the time step.

Present results show that the scheme being conservative and stable
macroscopic quantities remain observable even when
 the spatial-temporal oscillations are not fully resolved numerically
 because the mesh is not enough fine to 
 capture wavelengths below $h$. Typically, one uses $h=0.01$ for all
 $\eps$ in a square domain of side one. 

 In our approach, conservation is ensured by projection steps to
 guarantee a correct behavior for 
 total position, energy and current (momentum) densities.
 The aim  is also to show that basic numerical methods
 \cite{strang,pratic}  can be used which permits the adaptation of
 generic PDE solvers. We also point out that we have privileged
 projections which are rather cheap computationally, since they are
 obtained by a simple rescaling. 

 The implementation has been done in two dimensions in space but
 extension to third dimension does not appear being a
 difficulty. Periodic boundary conditions and initial data with
 compact support have been considered. 

 Let us start with system \eqref{eq:grenierv} which we rewrite as:
\begin{equation}
  \label{systeps}
  \d_t U^\eps + F(U^\eps) = 0,
\end{equation}
$$ U^\eps(x\in \Omega,t=0)=U_0(x), ~~U^\eps(\partial \Omega,t)=\hbox{periodic}.$$
 where $U^\eps=(a^\eps,v^\eps)$. $\Omega_h$ is a discrete two dimensional square
  domain of side $L$. $U_0(x)$  is a regular  initial condition. For
  all the  simulations  presented in this paper we consider
 $$ U_0(x)=(a_0(x),\alpha f(x), \alpha g(x))^t, $$
 with $a_0(x)$ a complex function  independent of $\eps$ with compact support.
 $\alpha$ is real and $f$ and $g$ are real functions with compact support.
 The initial pattern is therefore periodic of period $L$ in both space
 directions. 
 Together with the periodicity, oscillations in space can be
 introduced through $a_0$, $f$ and $g$. 
Below we show numerical results with two values of $\alpha$.

We consider second order finite difference discretizations of
partial differential operators. But, the periodic boundary
conditions permit the implementation of high order spatial
discretizations as well as spectral methods. One notices that
despite the presence of first order space derivatives, no
numerical viscosity is necessary to stabilize the system both in
the hydrodynamic limit and for $\eps \neq 0$. We therefore keep
the numerical viscosity to zero for all simulations which means no
upwinding has been used. This leads to a consistent scheme with
truncation error in $h^2$:
$$ F(U^\eps)=F_h(U^\eps) + \O(h^2).$$

We consider a simple first order explicit time integration scheme:
\begin{equation}
  \label{systepsh}
{1\over k}(U^\eps_{h,n+1/2}-U^\eps_{h,n}) + F_h(U^\eps_{h,n}) = 0,
\end{equation}
$$ U^\eps_{h,0}=U_0(x_h), ~~U^\eps_{h,n+1/2}(\partial
\Omega_h)=\hbox{periodic}.$$ 
$n+1/2$ denotes an intermediate state, before projection, where
conservation  is not guaranteed for mass and momentum. It is
interesting that the approach appears stable even for explicit
time integration. With a first order scheme in time, and  a time
step in $h^2$, the time integration error will be comparable to
the truncation error in space.

Once $U^\eps_{h,n+1/2}$ is computed, one needs to project it over the
admissible space to get $U^\eps_{h,n+1}$
based on enforcing mass, energy and momentum conservation constraints
(see Proposition~\ref{prop:conservations}):
$$J_{1,2,3}(U^\eps_{h,n+1/2},U^\eps_{h,0})={I_{1,2,3}(U^\eps_{h,n+1/2})\over
  I_{1,2,3}(U^\eps_{h,0})} = 1,$$
where
\begin{align*}
  I_1(U^\eps_{h,n+1/2}) &=\int_{\Omega_h} |a^\eps_{h,n+1/2}|^2 dx,\\
I_2(U^\eps_{h,n+1/2})& = \int_{\Omega_h} \left(|a^\eps_{h,n+1/2}|^4  +
\left\lvert \eps \nabla_h a^\eps_{h,n+1/2}  + i a^\eps_{h,n+1/2}
v^\eps_{h,n+1/2}\right\rvert^2 \right) dx,\\
I_3(U^\eps_{h,n+1/2})&=\int_{\Omega_h} \left( |a^\eps_{h,n+1/2}|^2
  v^\eps_{h,n+1/2} +\eps \IM \(\overline 
  a^\eps_{h,n+1/2} \nabla a^\eps_{h,n+1/2}\) \right) dx.
\end{align*}
$J_3$ is a vector of the size $d$ of the space dimension. This problem
is overdetermined  with essentially two variables ($a^\eps$ and
$v^\eps$). This overdetermination is maybe one reason why no
numerical scheme is available for these equations verifying all
conservation constraints. With $a^\eps$ complex, there are as many
variables as constraints. Still we did not manage to enforce at
the same time the mass $J_1$ and energy $J_2$ constraints.
We have chosen here to enforce $J_1$ and $J_3$.

$J_1$ can be easily enforced in $a^\eps_{h,n+1}$ by simply defining:
$$ a^\eps_{h,n+1} = a^\eps_{h,n+1/2} \left({I_{1}(U^\eps_{h,0})\over
    I_{1}(U^\eps_{h,n+1/2})}\right)^{1/2}. $$ 

The  projection  aims at looking for a particular equilibrium for
the constraints after a splitting of the variables. The above scaling
suggests an \emph{a priori} but natural
splitting of the variables to be modified by each constraint. More
precisely, $J_1$ defines the
corrections for $a_{h,n+1/2}^\eps$ and the vector $J_3$ the ones for
the components $(v_{h,n+1/2}^\eps)_j$ of the
velocity through:
\begin{equation}
(\tilde{I}_3)^j=\int_{\Omega_h} \left( |a^\eps_{h,n+1}|^2
  (v^\eps_{h,n+1/2})_j +\eps \IM \(\overline
  a^\eps_{h,n+1} \d_j a^\eps_{h,n+1}\) \right) dx, ~~ j=1,\ldots,d.
  \label{i3semi}
  \end{equation}
  Because we are looking for a cheap projection based on scaling, we
  adopt the following corrections for each 
  component of $v_{h,n+1/2}^\eps$:
$$ (v^\eps_{h,n+1})_j = (v^\eps_{h,n+1/2})_j
\left({I^j_{3}(U^\eps_{h,0})\over \tilde{I}^j_{3}}\right), ~~
j=1,\ldots,d. $$
Through the numerical examples below we see that these scalings are
efficient in conserving mass and current densities.

\section{Numerical experiments}
\label{sec:exp}
We show the application of our projection schemes for several initial
conditions. In the first
 case the current density is nearly zero and not in the second. A
 third case shows the robustness of the approach
 with initial vanishing $a^\eps$.  We show the impact of the
 projection on the conservation of
mass, energy and momentum through $J_1$, $J_2$ and $J_3$. We will
see that  mass and energy cannot be both conserved at the same time.

\subsection{Nearly zero initial current}

We consider $L=0.5$, $a_0(x)=\exp(-80((x_1-L/2)^2+(x_2-L/2)^2)) (1+i)$
and $\alpha=10^{-10}$ (hence $v^\eps_{\mid t=0}\approx 0$).

Figures~\ref{iso_rj01} shows the initial position and current
densities.  Figures~\ref{iso_rj1_ssproj} and \ref{iso_rj1} show
the solutions at $T=0.1 \sec$ for $\eps=0, 0.001, 0.01$ and 0.1
without and with the projection steps.


\begin{figure}[!htp!]
\begin{center}
\begin{tabular}{ll}
\includegraphics[width=5cm,angle=-90]{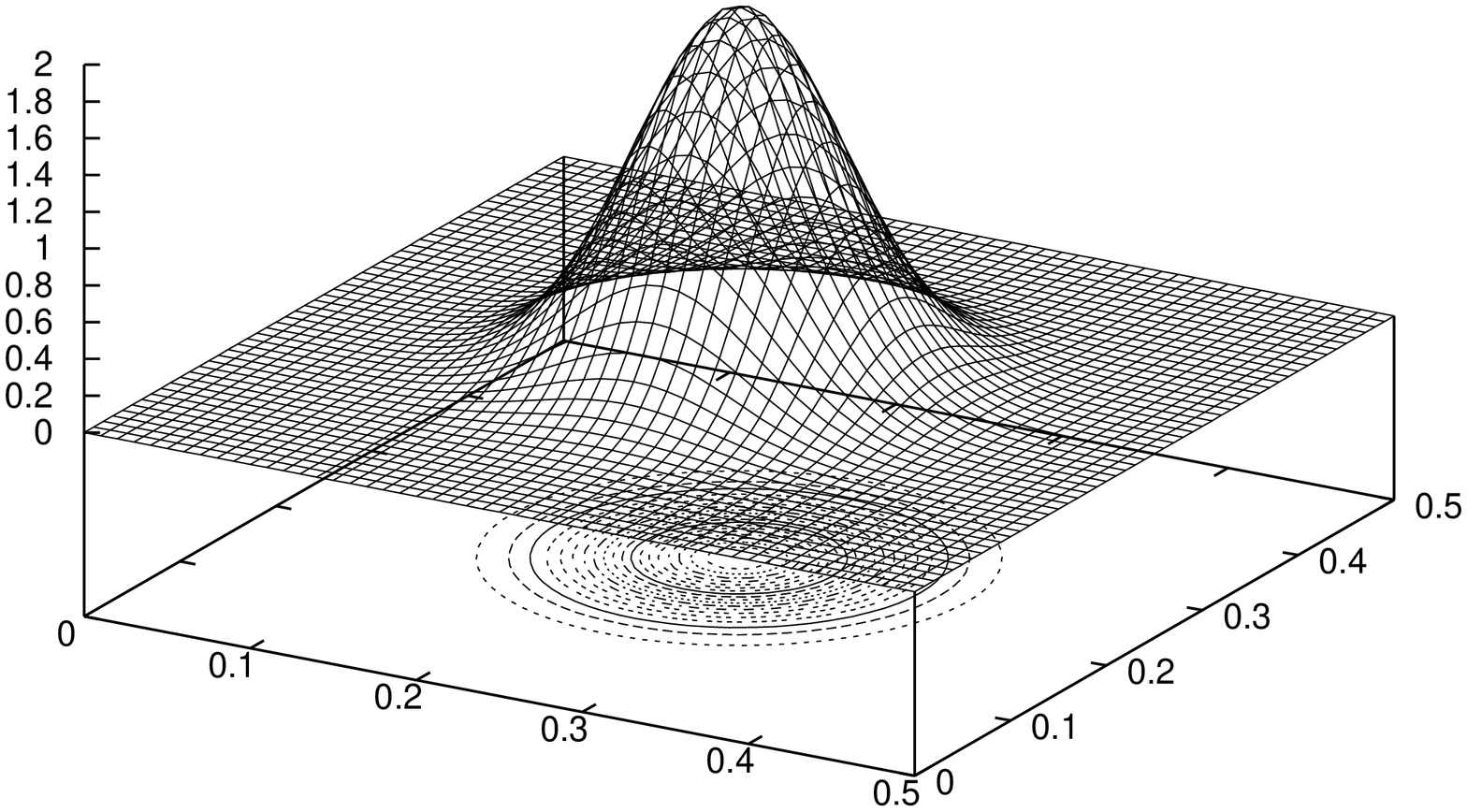}&
\hspace{-1.5cm}
\includegraphics[width=5cm,angle=-90]{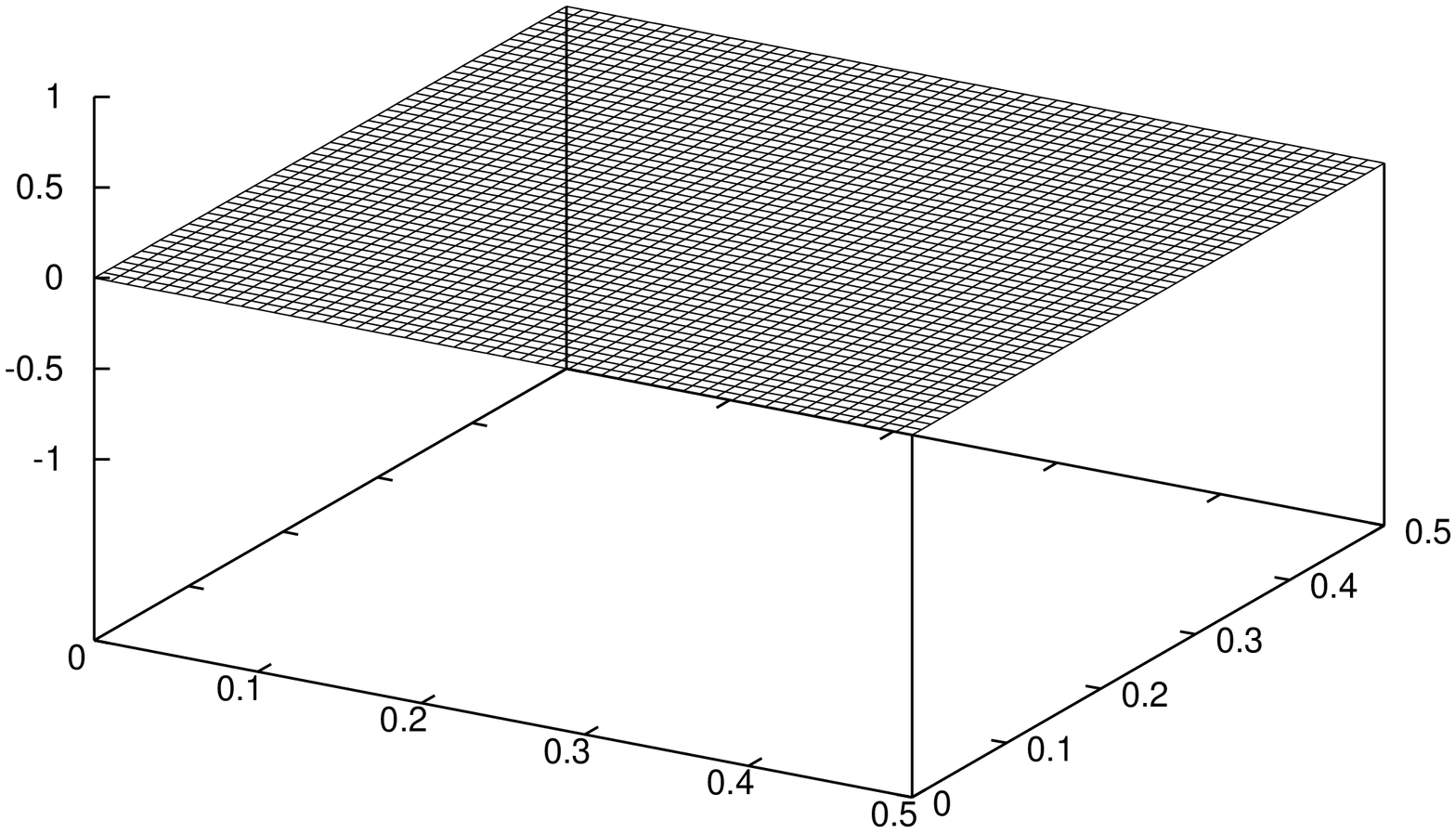}
\end{tabular}
\caption{\label{iso_rj01} Initial position (left) and norm of the
 current density vector (right) with $\alpha=10^{-10}$.}
\end{center}

\end{figure}


\begin{figure}[!htp!]
\begin{center}
\begin{tabular}{ll}
\vspace{-1.5cm}
\includegraphics[width=5cm,angle=-90]{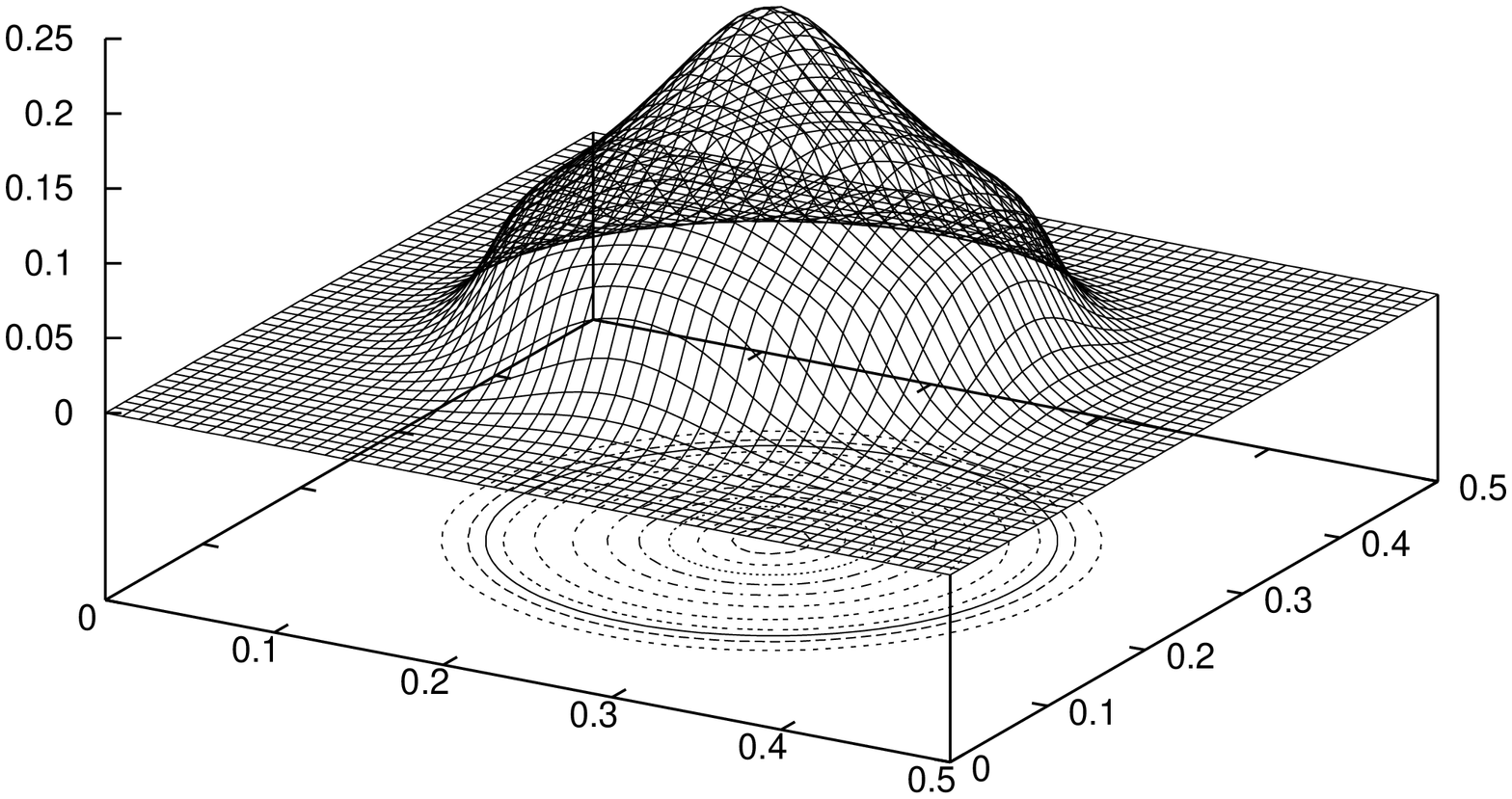}&
\hspace{-1.5cm}
\includegraphics[width=5cm,angle=-90]{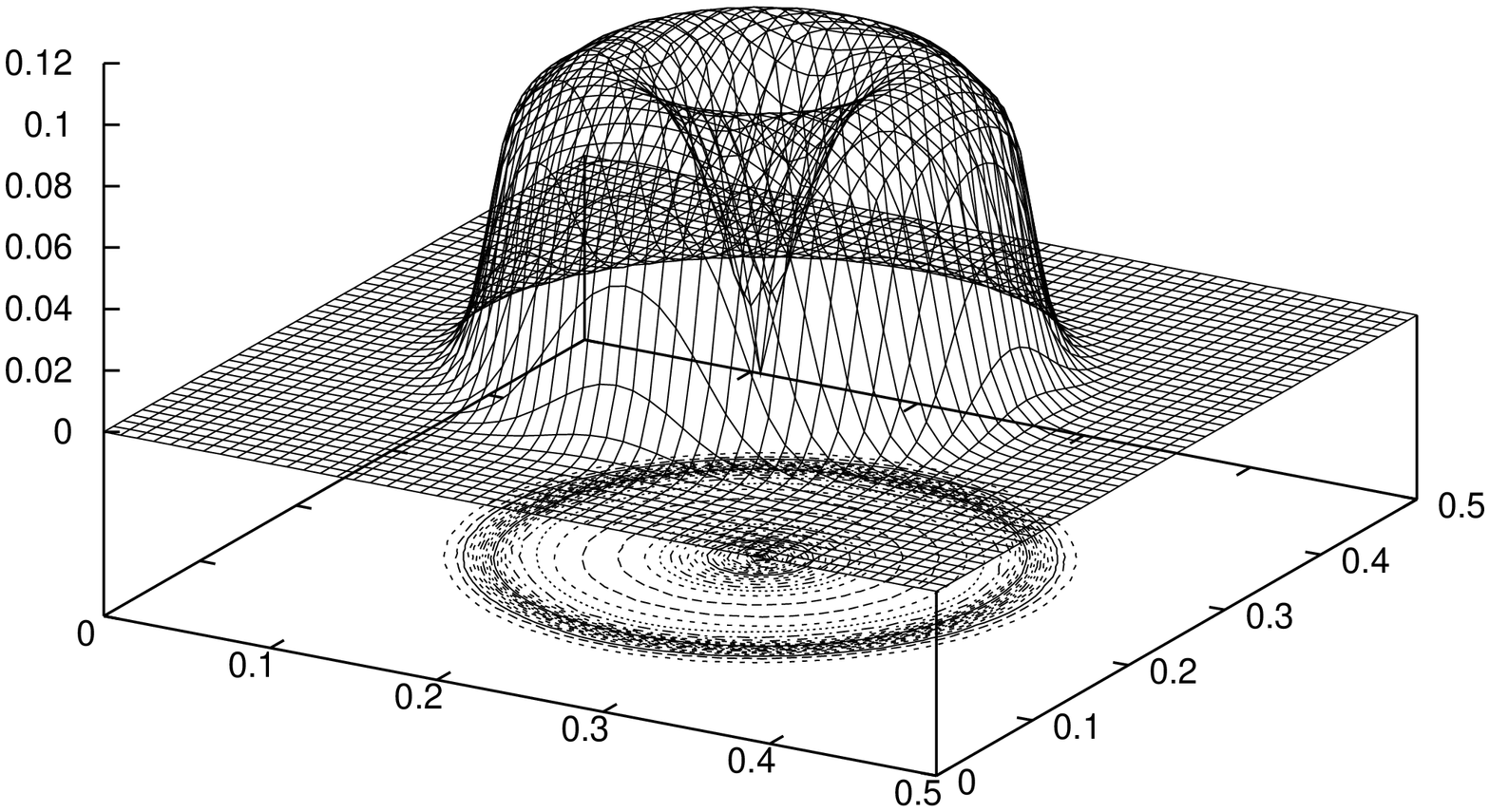}\\
\vspace{-1.5cm}
\includegraphics[width=5cm,angle=-90]{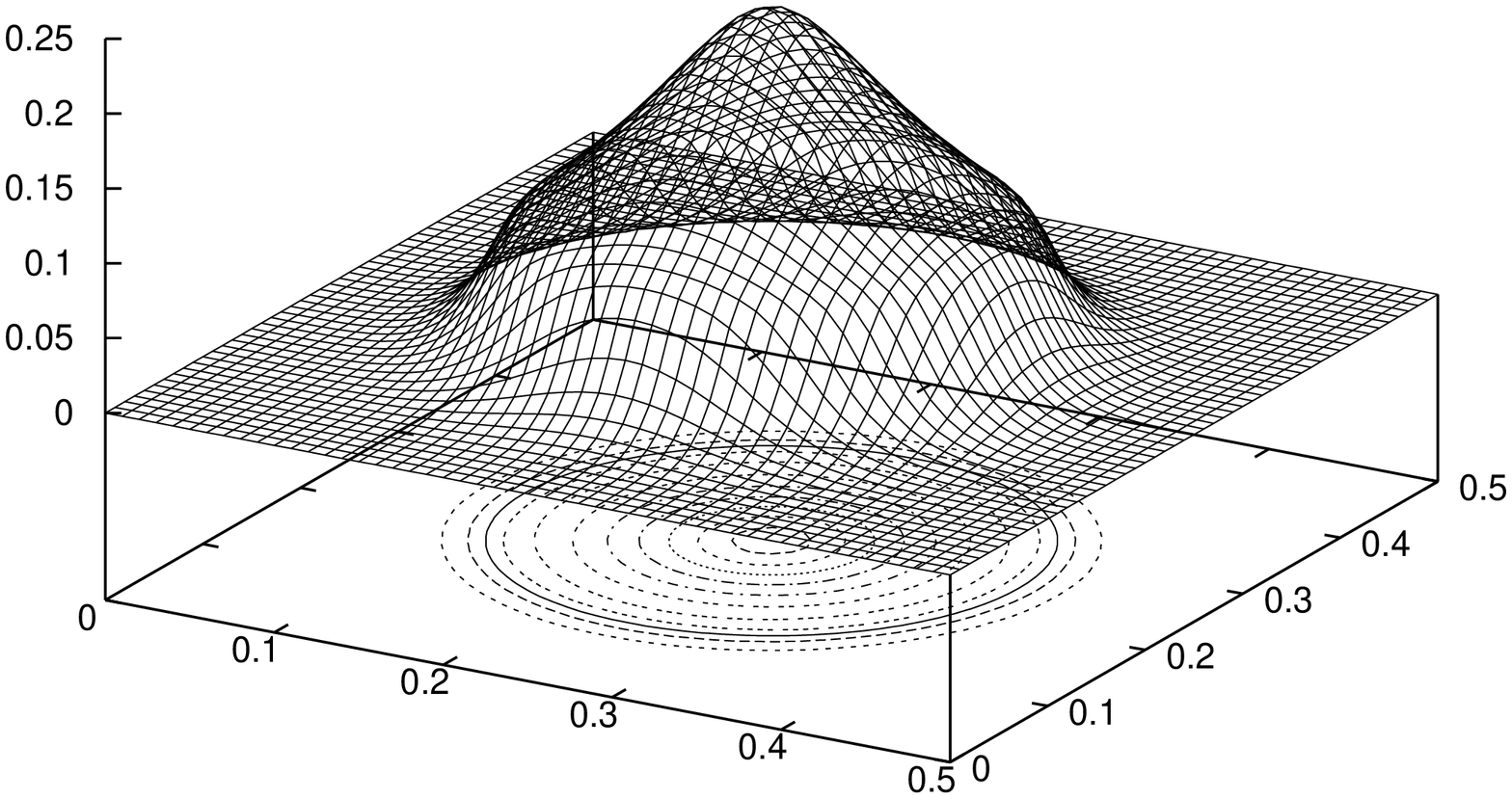}&
\hspace{-1.5cm}
\includegraphics[width=5cm,angle=-90]{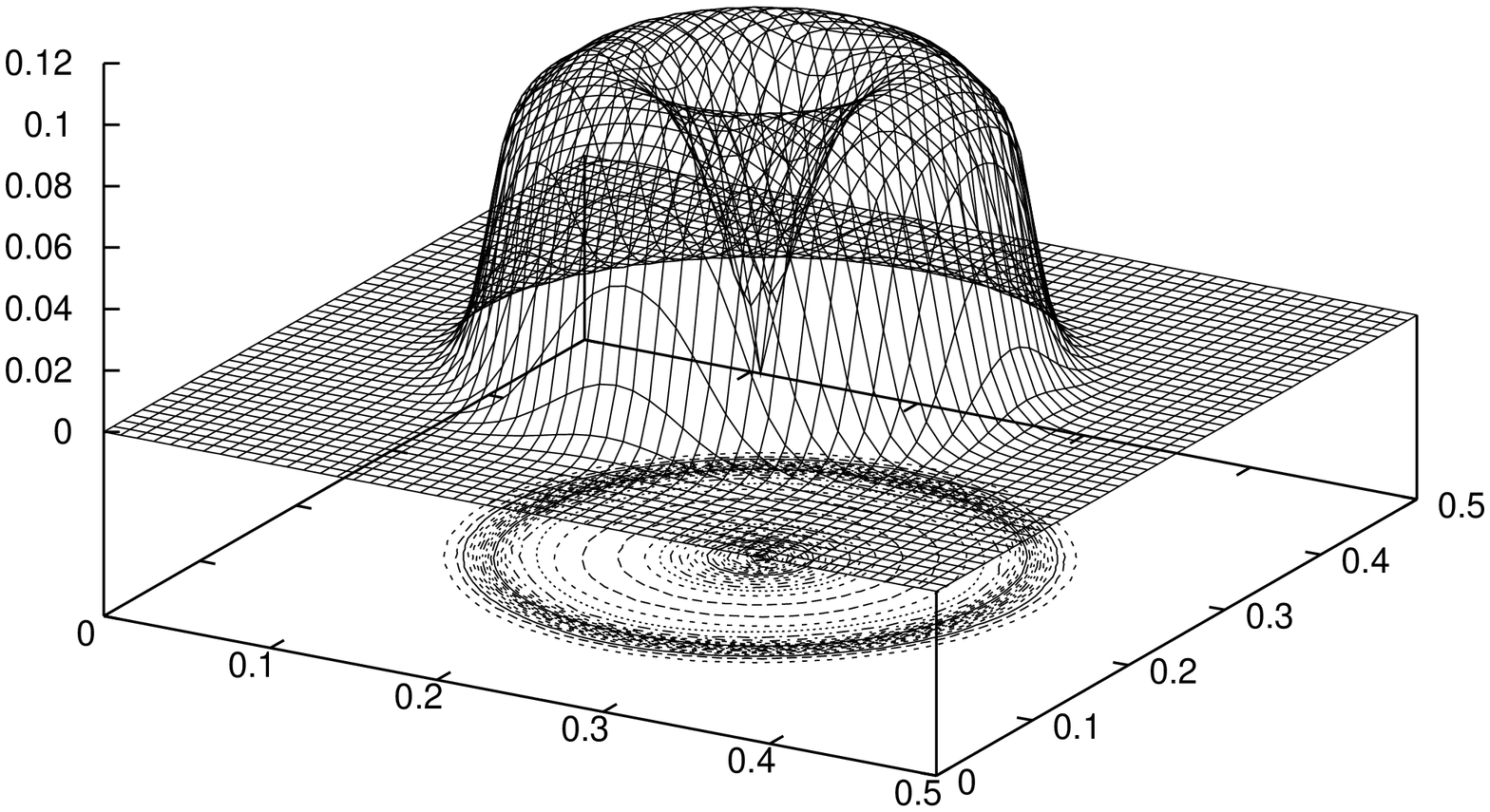}\\
\vspace{-1.5cm}
\includegraphics[width=5cm,angle=-90]{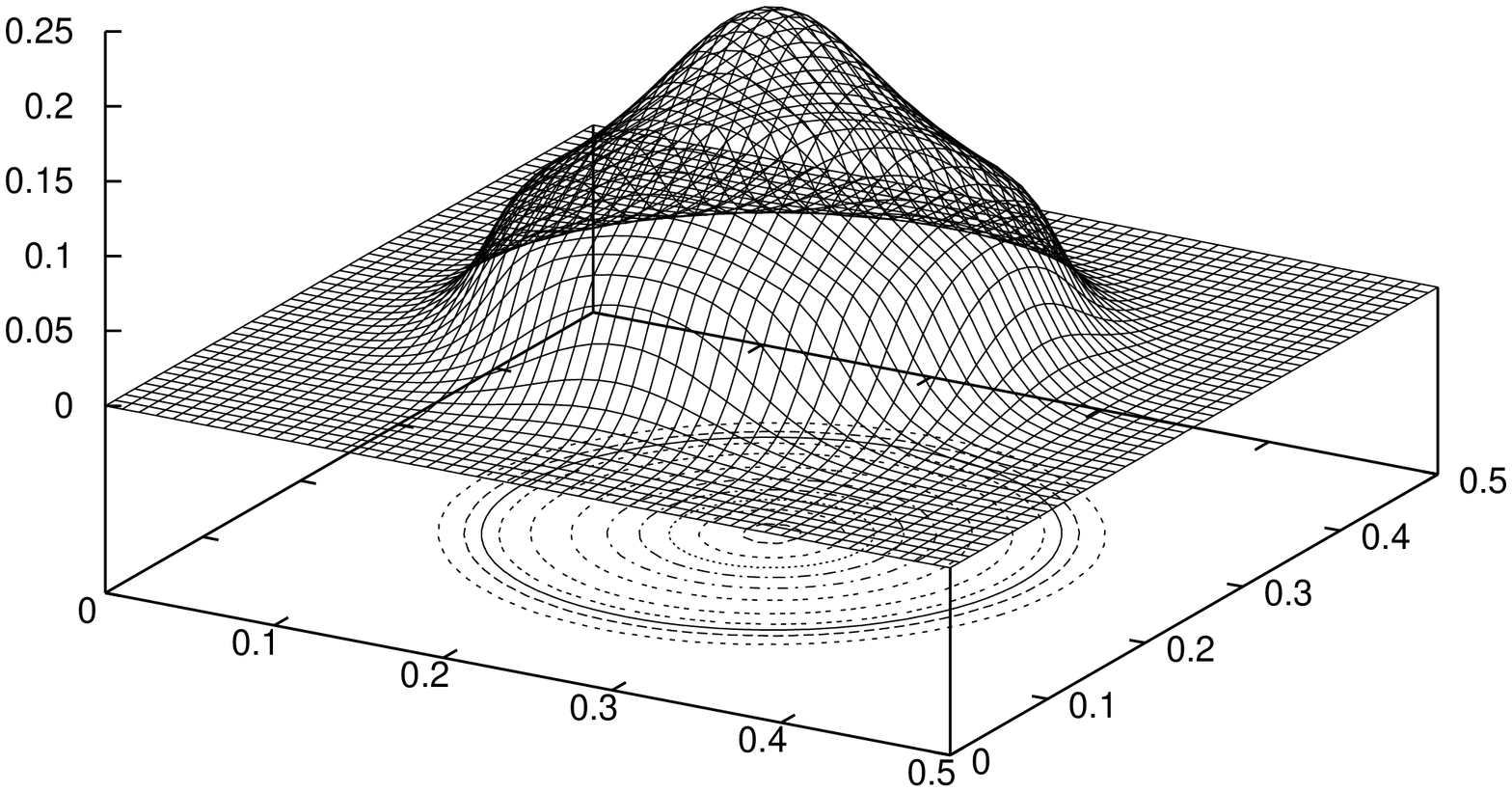}&
\hspace{-1.5cm}
\includegraphics[width=5cm,angle=-90]{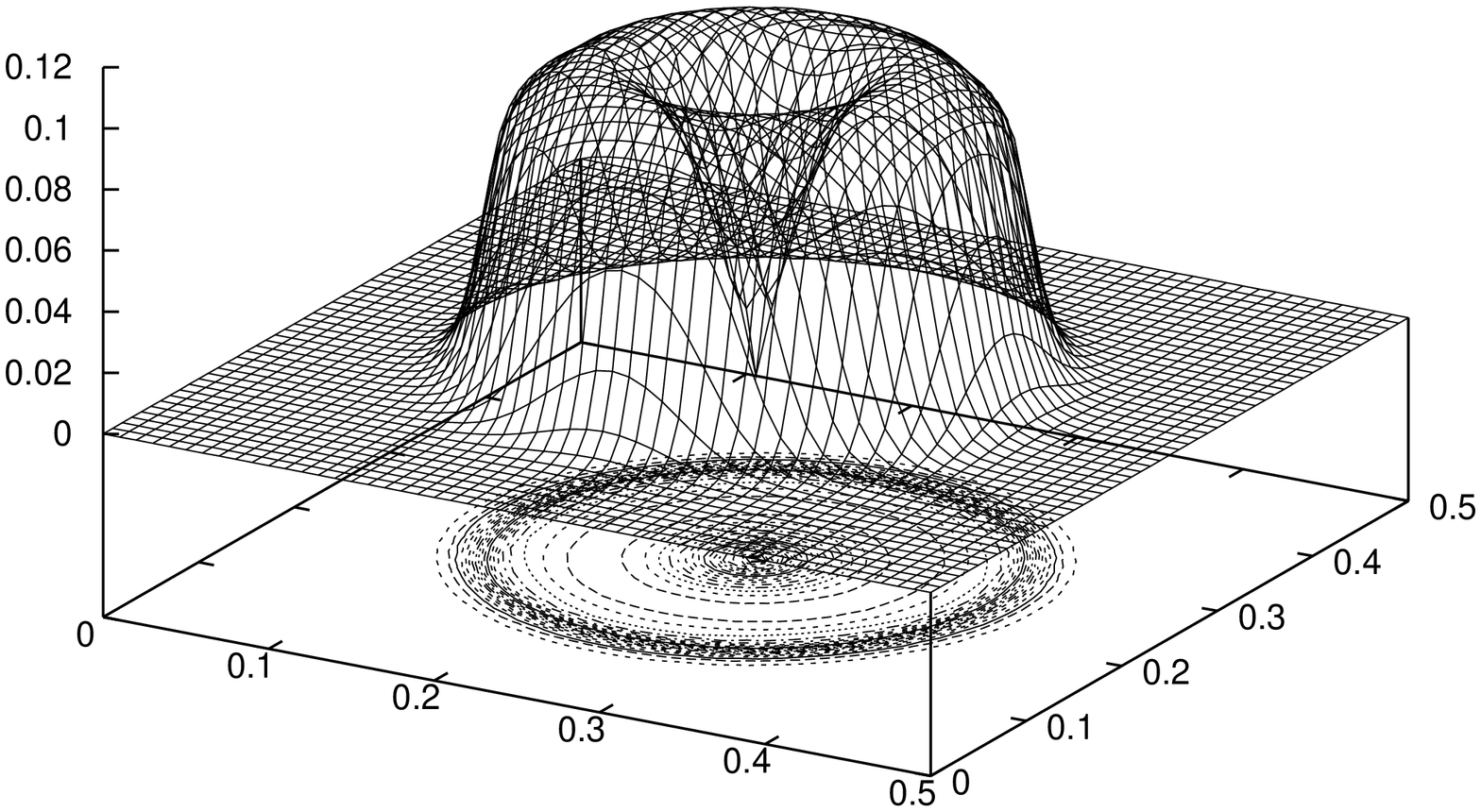}\\
\includegraphics[width=5cm,angle=-90]{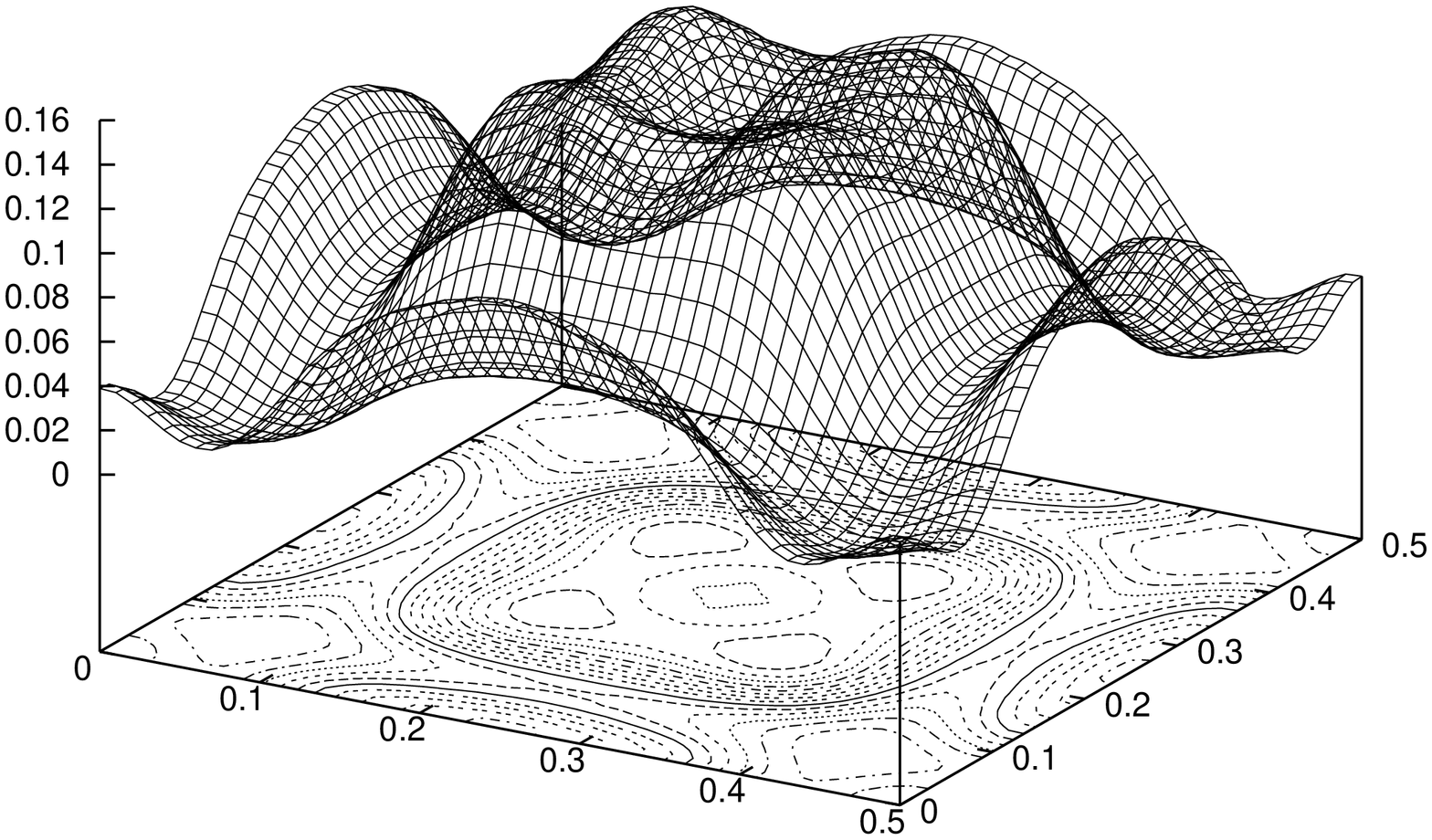}&
\hspace{-1.5cm}
\includegraphics[width=5cm,angle=-90]{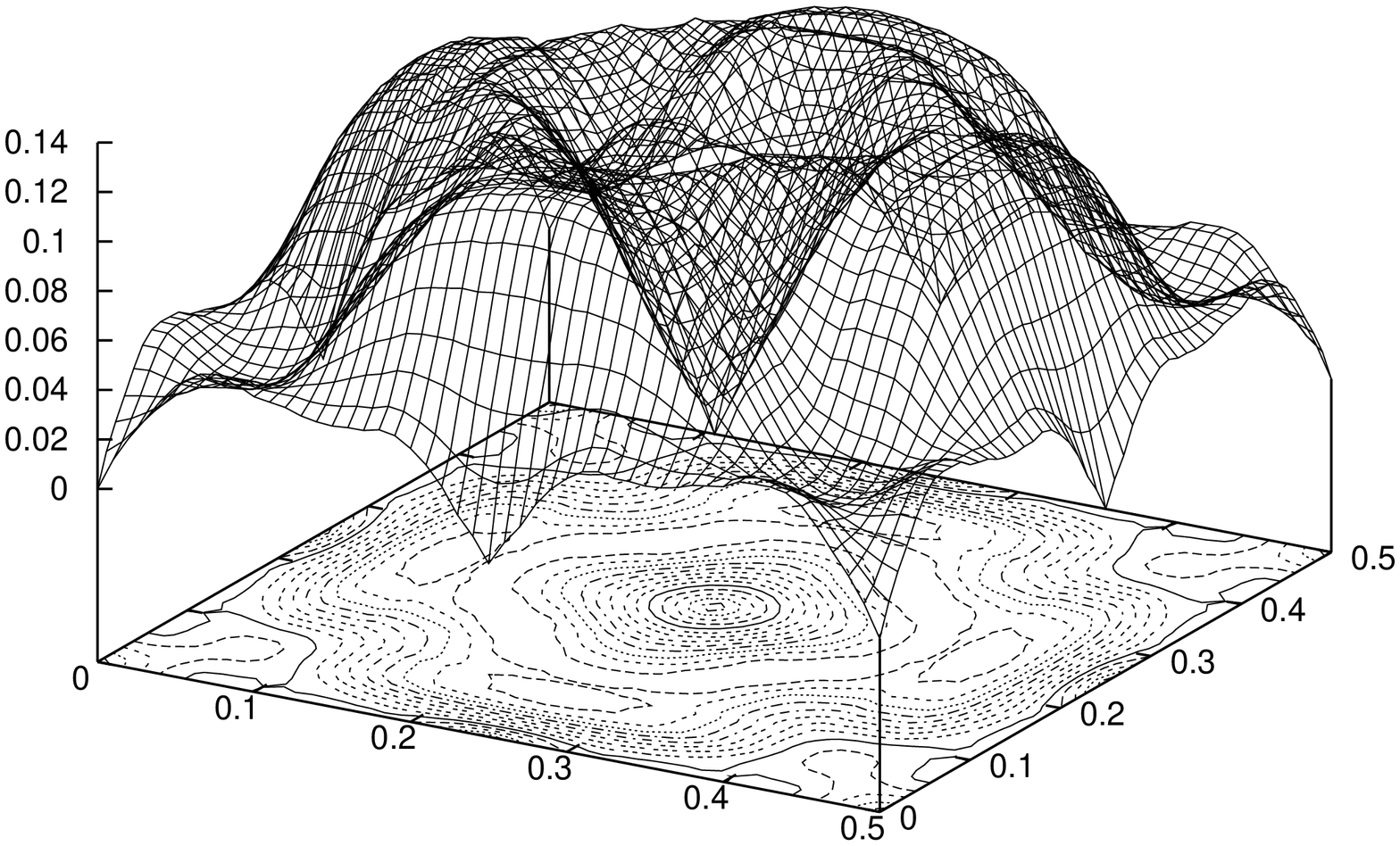}
\end{tabular}
\caption{ \label{iso_rj1_ssproj} Without projection. Position (left column) and norm of the current
density vector (right column) at $T=0.1 \sec$ for (resp. from the top)
 $\eps=0, 0.001, 0.01$ and $0.1$ with $\alpha=10^{-10}$.}
\end{center}
\end{figure}

\begin{figure}[!htp!]
\begin{center}
\includegraphics[width=6cm,angle=-90]{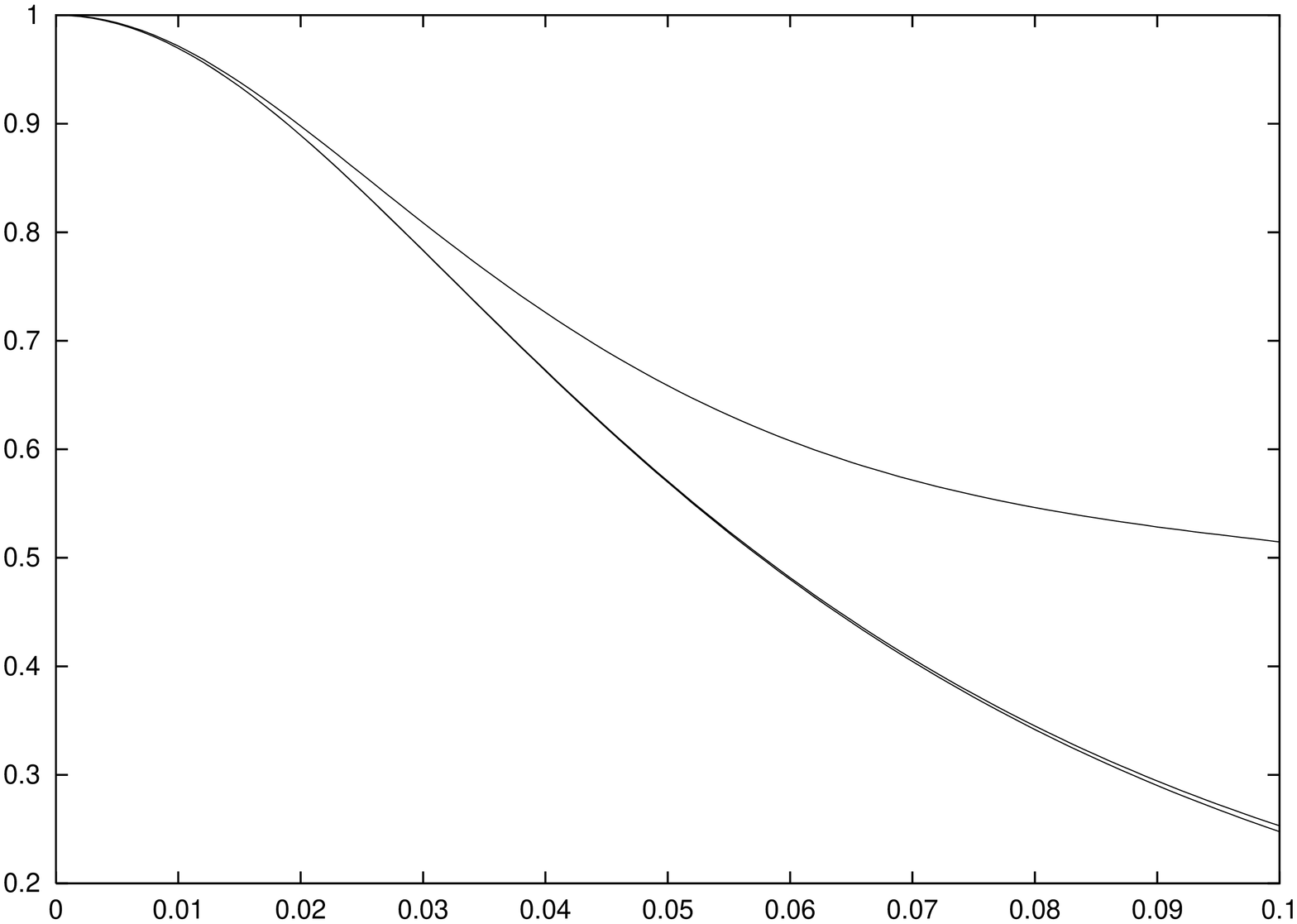}
\includegraphics[width=6cm,angle=-90]{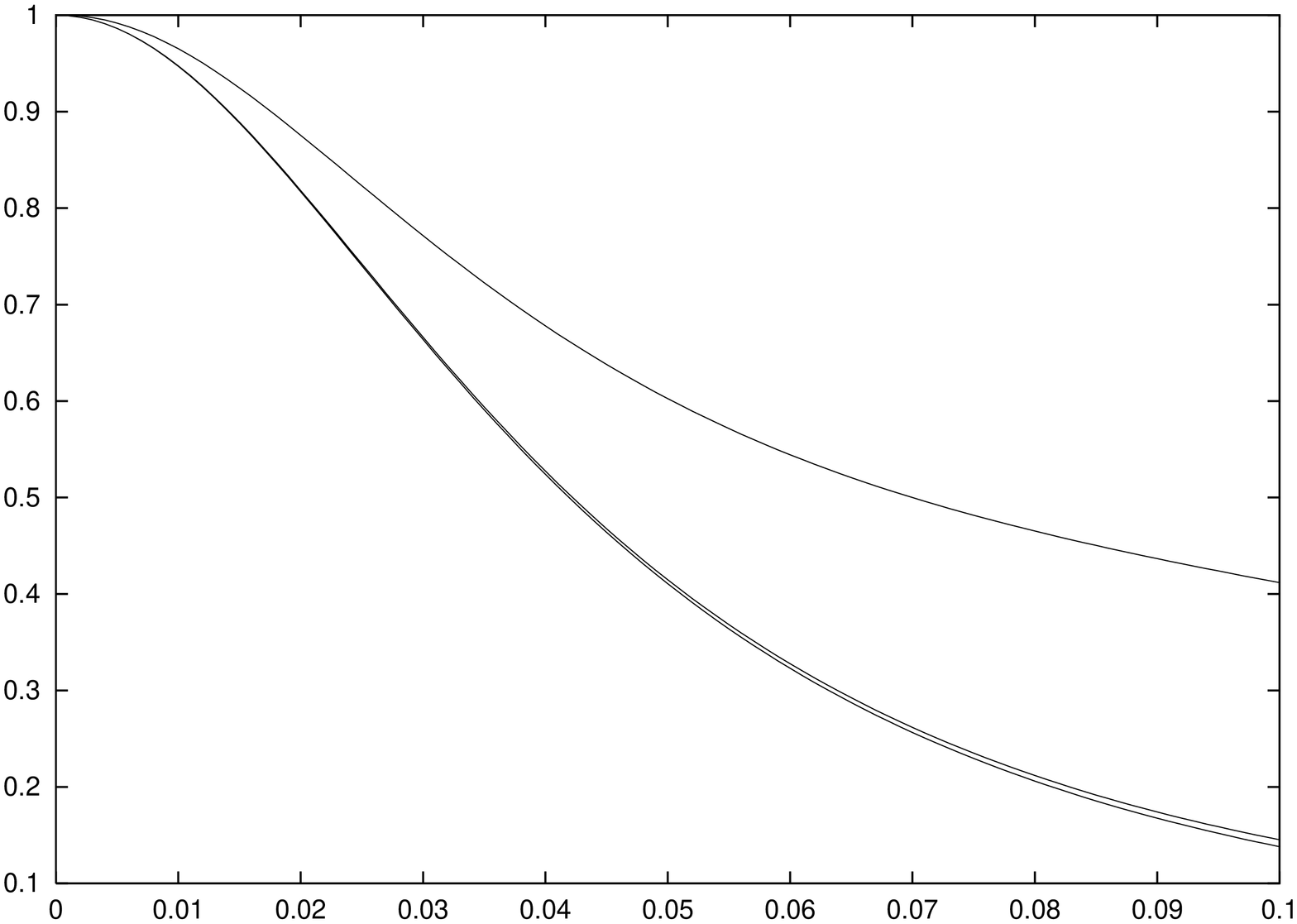}
\includegraphics[width=6cm,angle=-90]{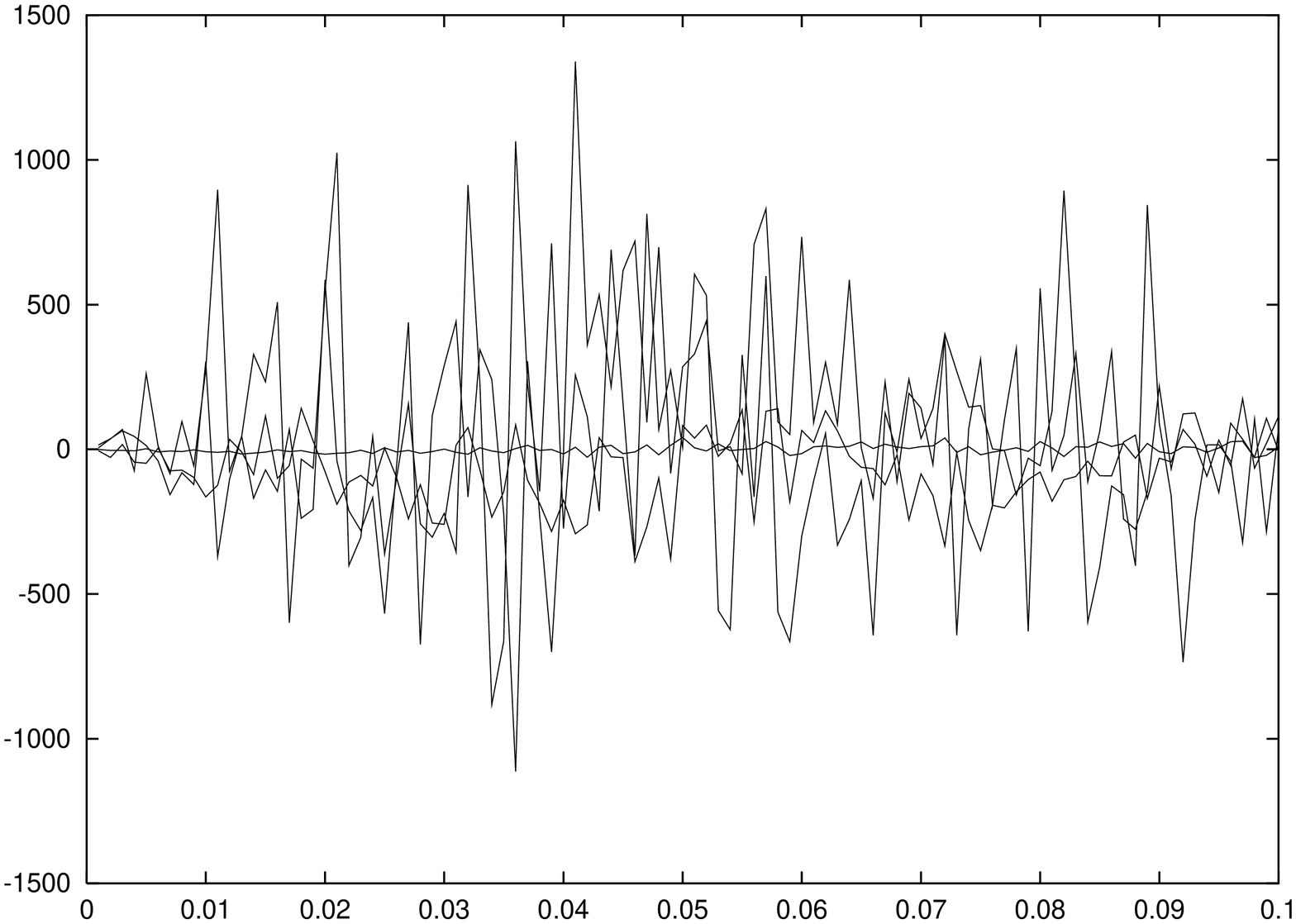}
\end{center}
\caption{ \label{res_roej1_ssproj} Without projection. Evolution  in
  time ($\sec$) of (resp. from the top) the constraints on the
  position density, energy and sum of both components of the current
density  for  $\eps=0, 0.001, 0.01$ and $0.1$ for an initial condition
with $\alpha=10^{-10}$.}
\end{figure}


\begin{figure}[!htp!]
\begin{center}
\begin{tabular}{ll}
\vspace{-1.5cm}
\includegraphics[width=5cm,angle=-90]{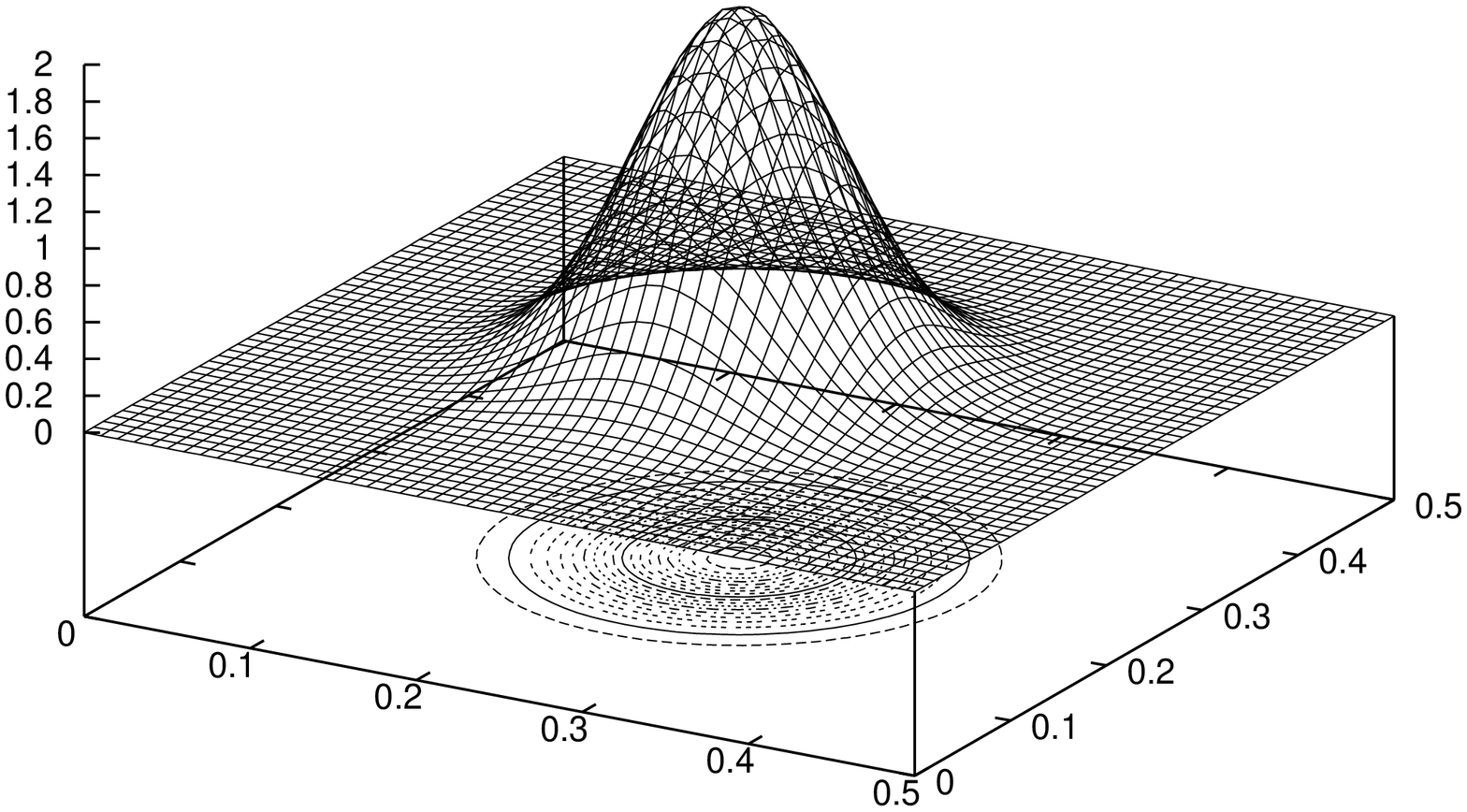}&
\hspace{-1.5cm}
\includegraphics[width=5cm,angle=-90]{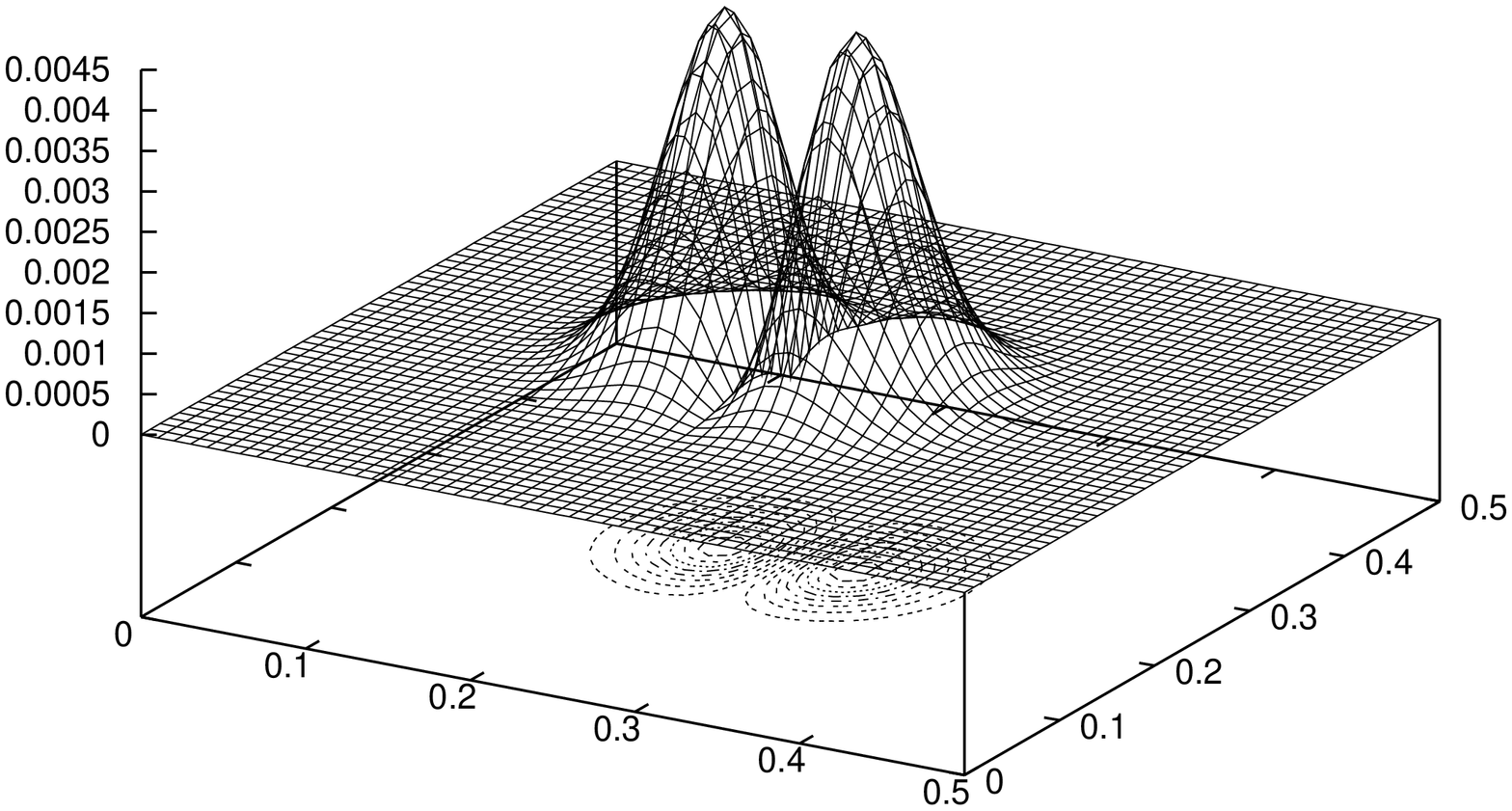}\\
\vspace{-1.5cm}
\includegraphics[width=5cm,angle=-90]{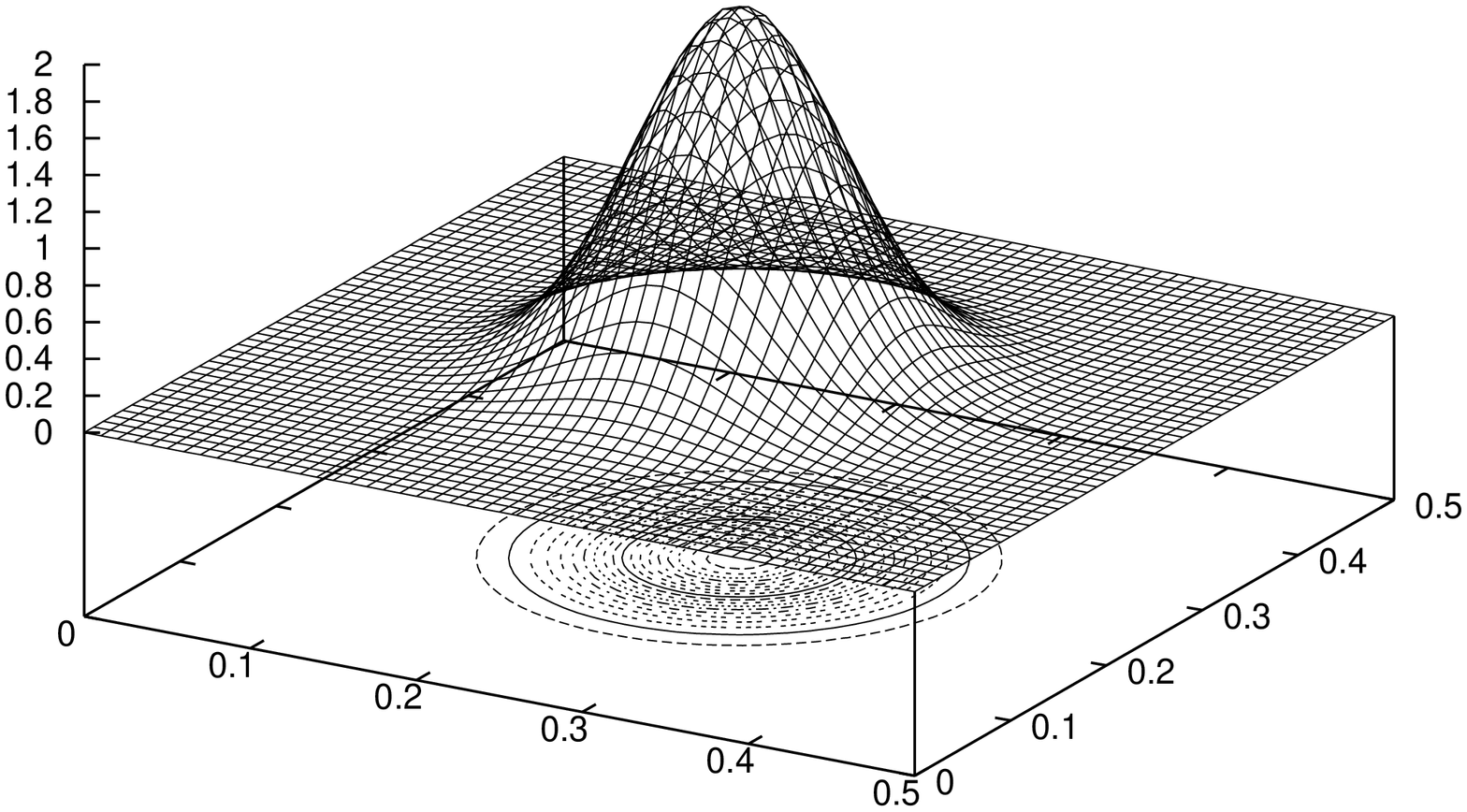}&
\hspace{-1.5cm}
\includegraphics[width=5cm,angle=-90]{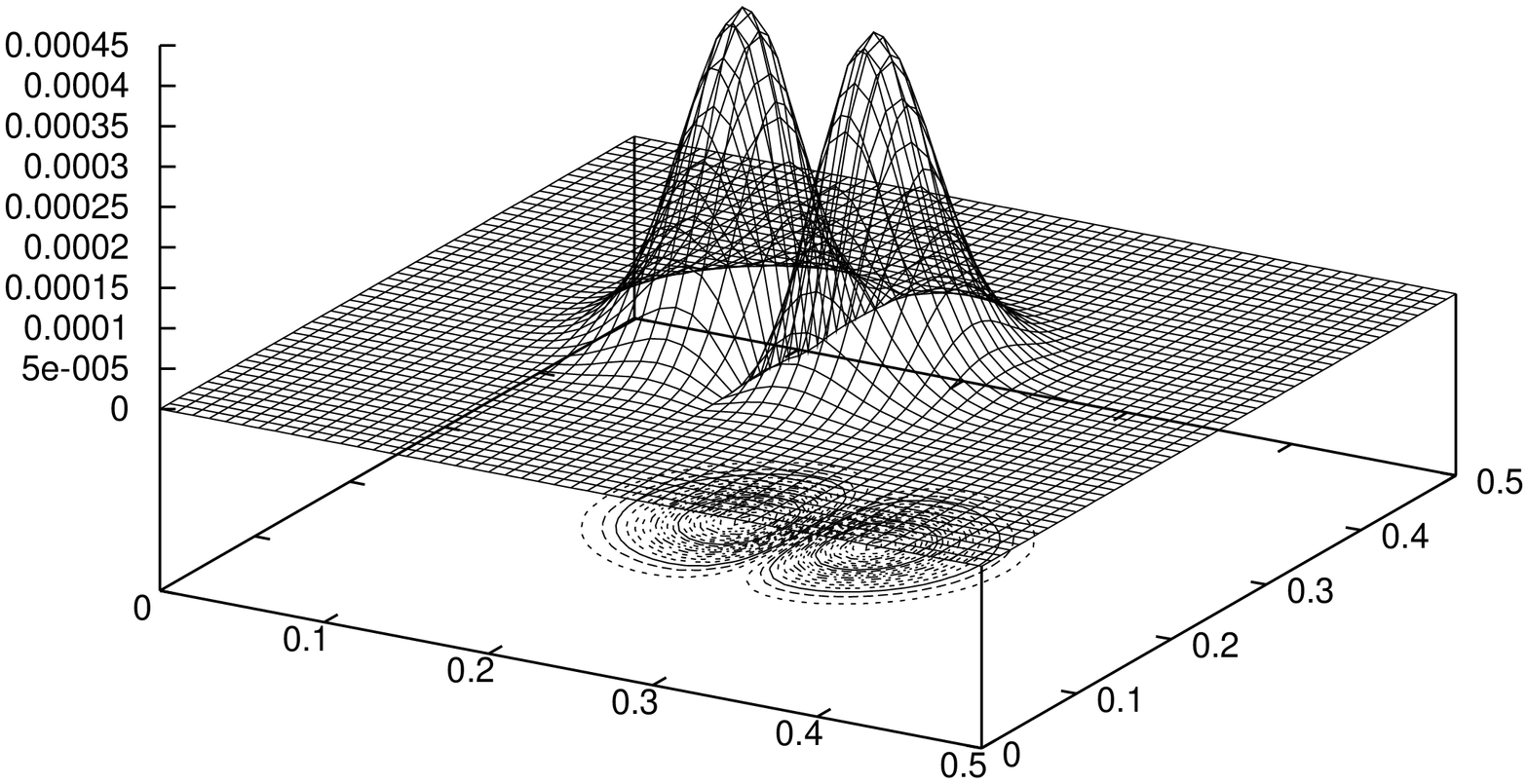}\\
\vspace{-1.5cm}
\includegraphics[width=5cm,angle=-90]{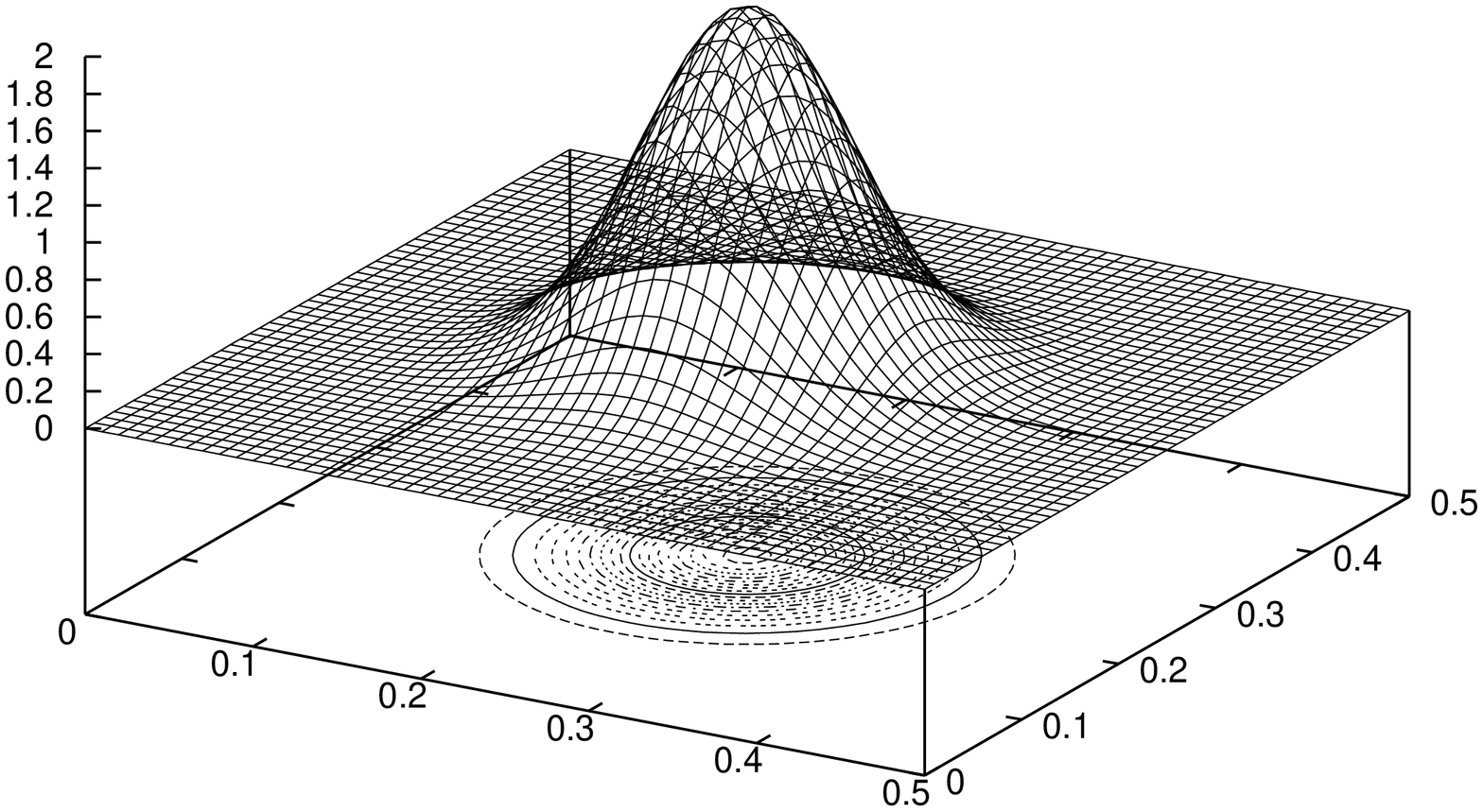}&
\hspace{-1.5cm}
\includegraphics[width=5cm,angle=-90]{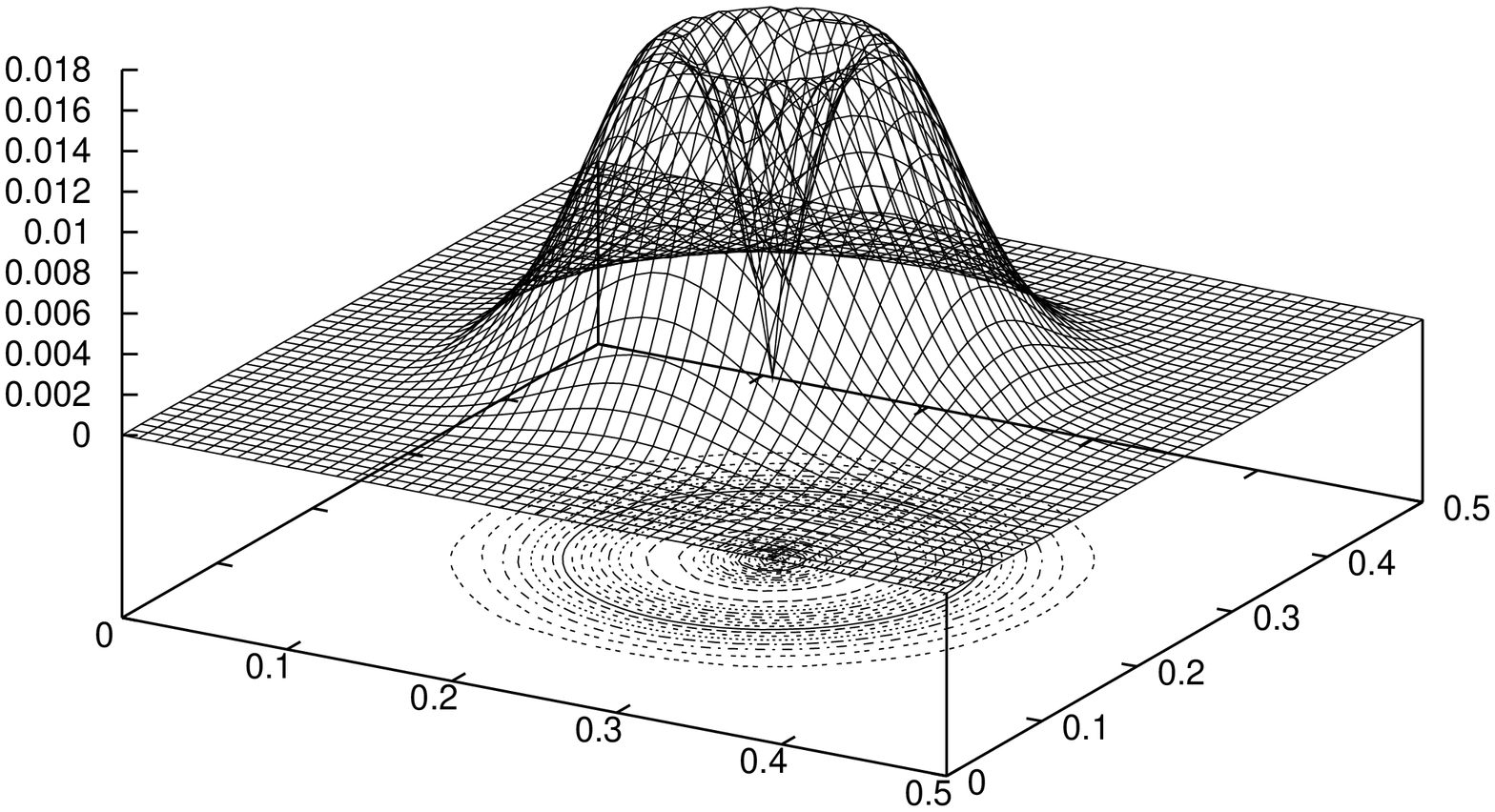}\\
\includegraphics[width=5cm,angle=-90]{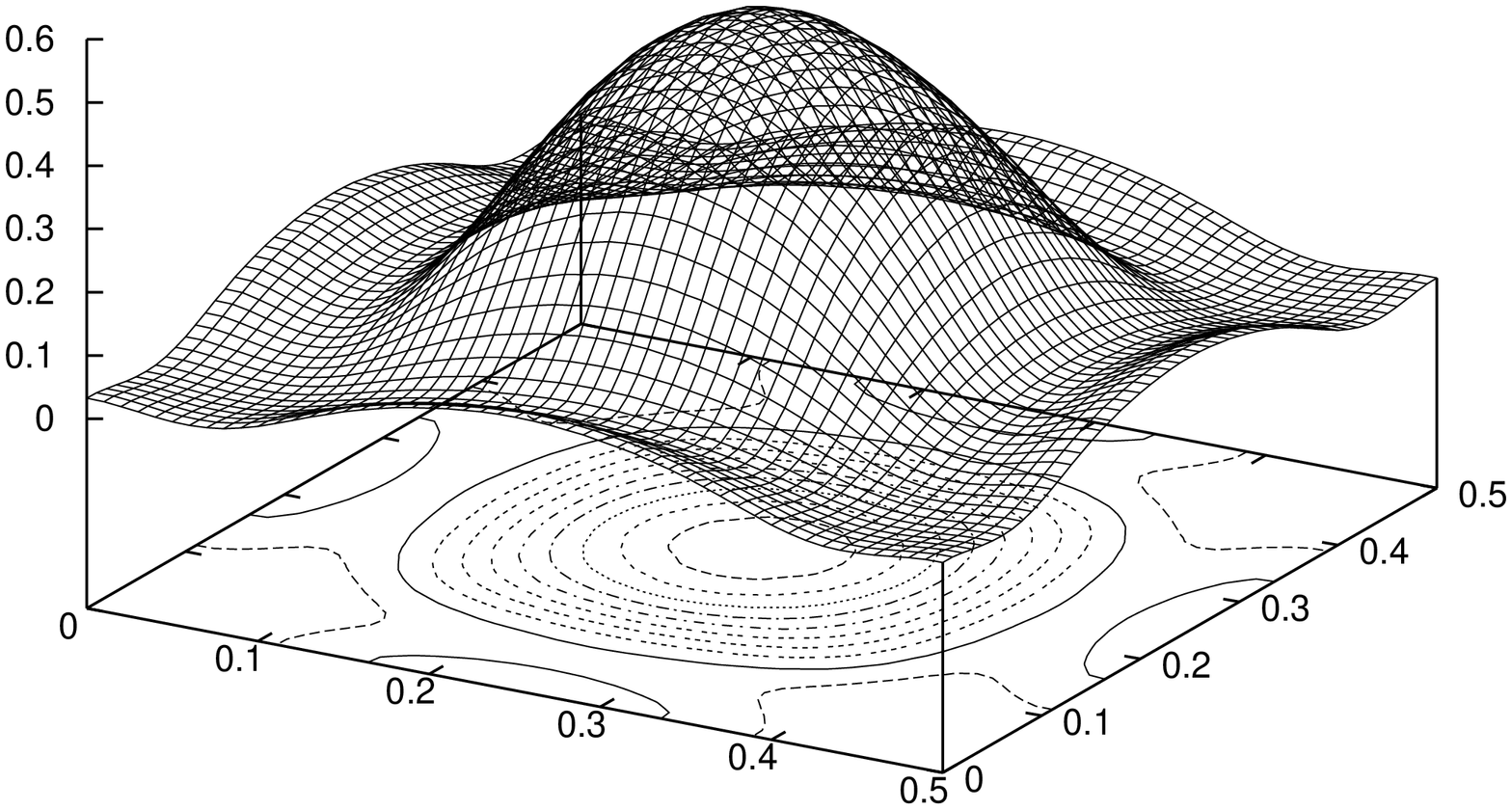}&
\hspace{-1.5cm}
\includegraphics[width=5cm,angle=-90]{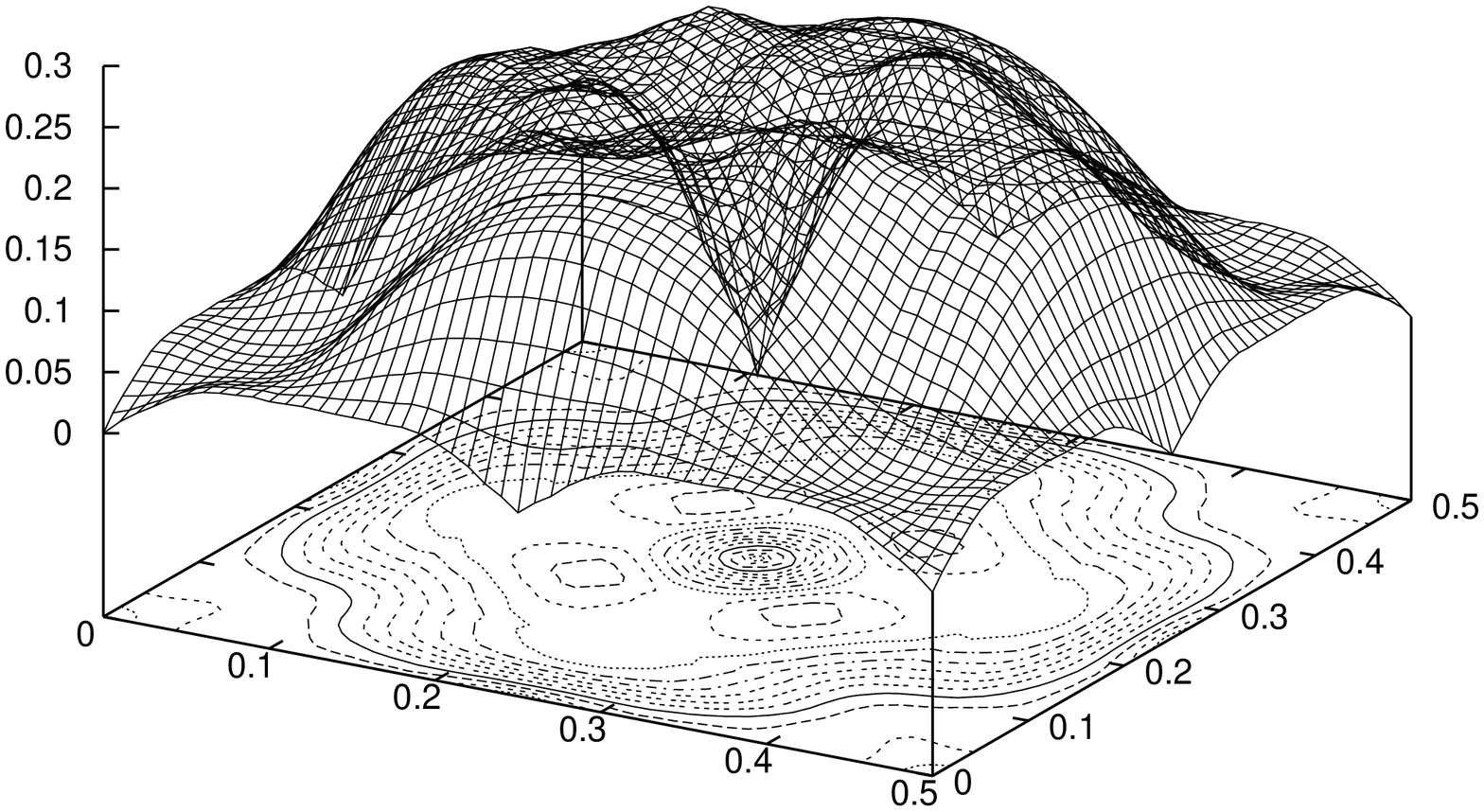}
\end{tabular}
\caption{ \label{iso_rj1} Position (left column) and norm of the current
density vector (right column) at $T=0.1 \sec$ for (resp. from the top)
 $\eps=0, 0.001, 0.01$ and $0.1$ with $\alpha=10^{-10}$.}
\end{center}
\end{figure}

\begin{figure}[!htp!]
\begin{center}
\includegraphics[width=6cm,angle=-90]{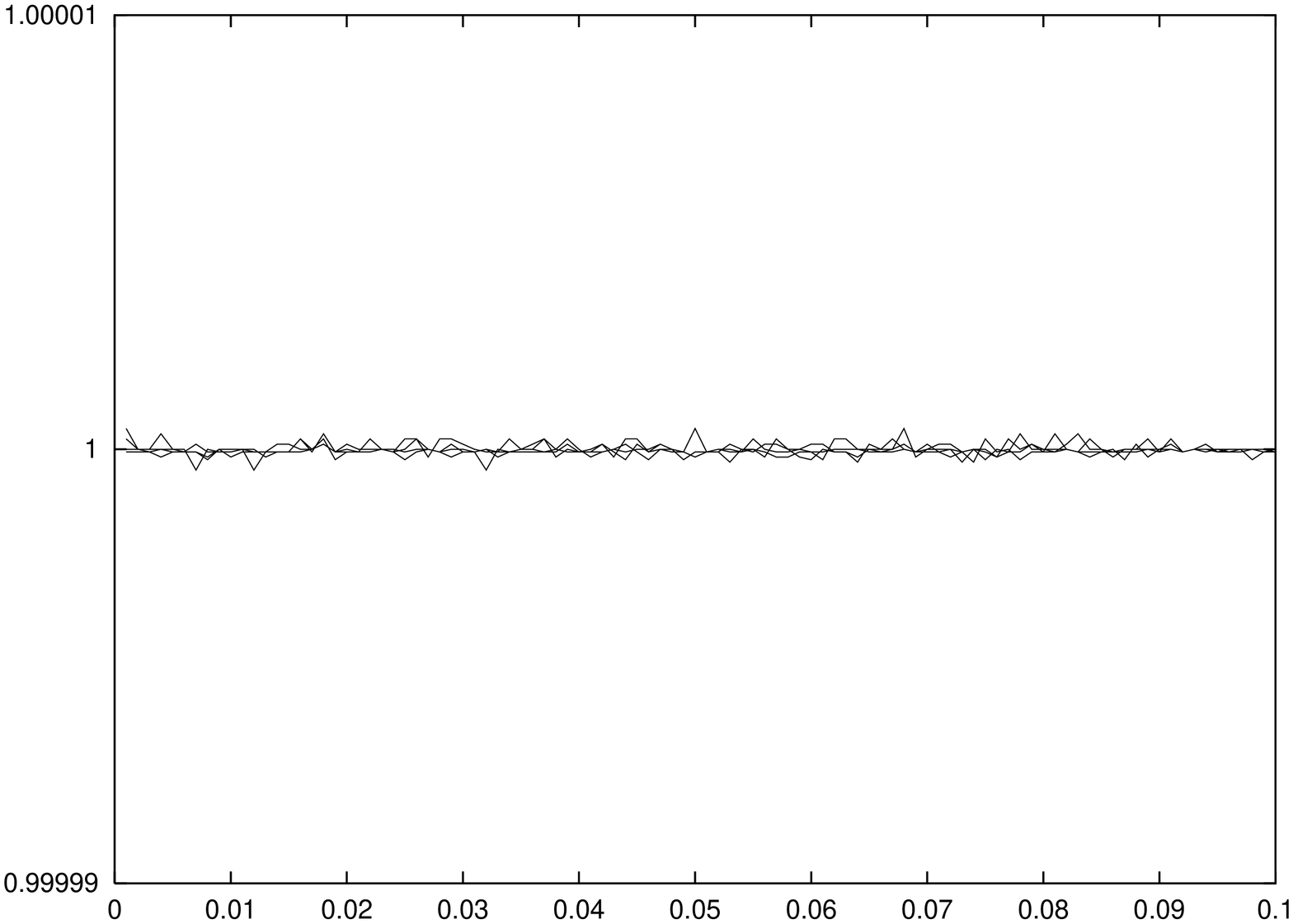}
\includegraphics[width=6cm,angle=-90]{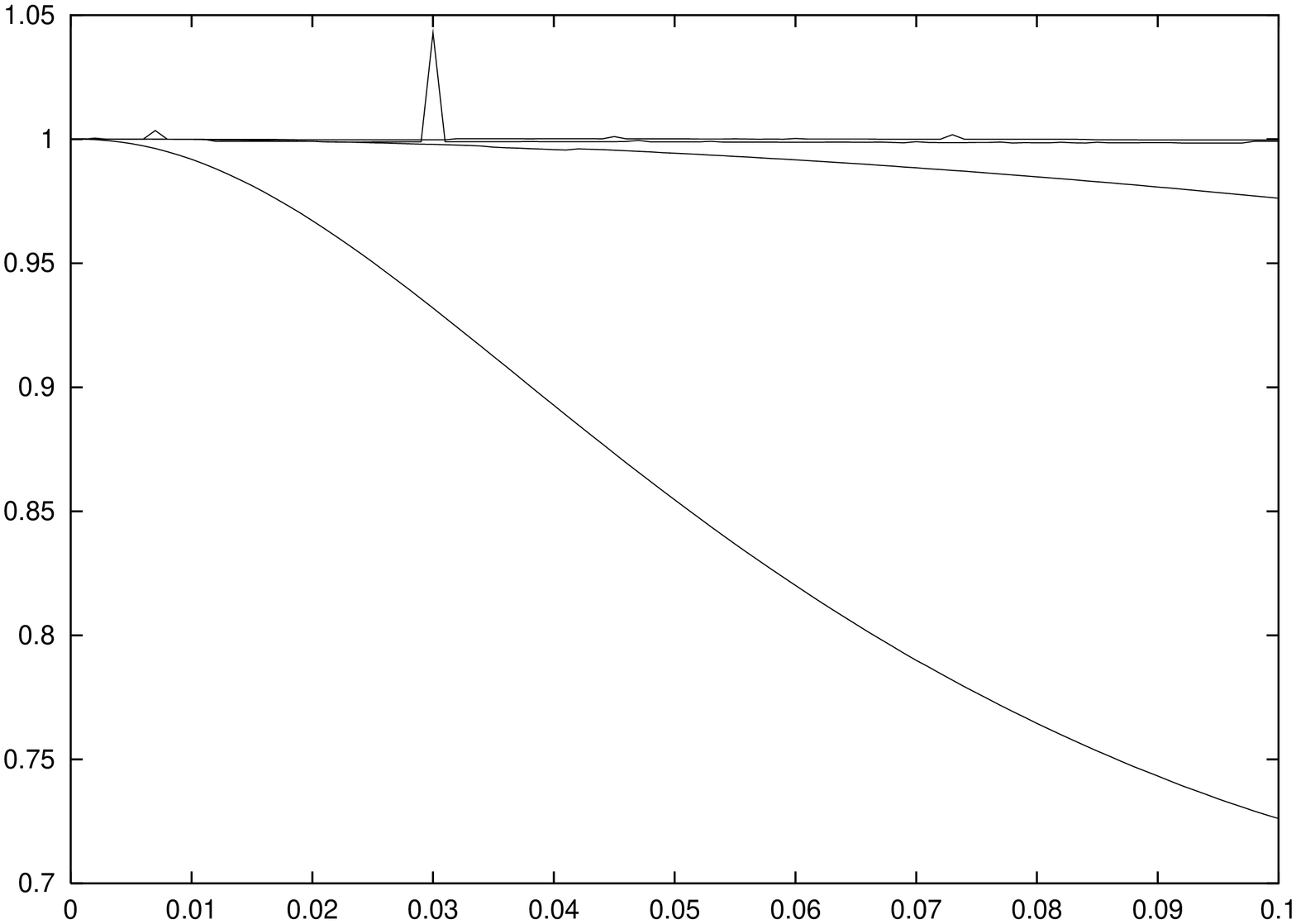}
\includegraphics[width=6cm,angle=-90]{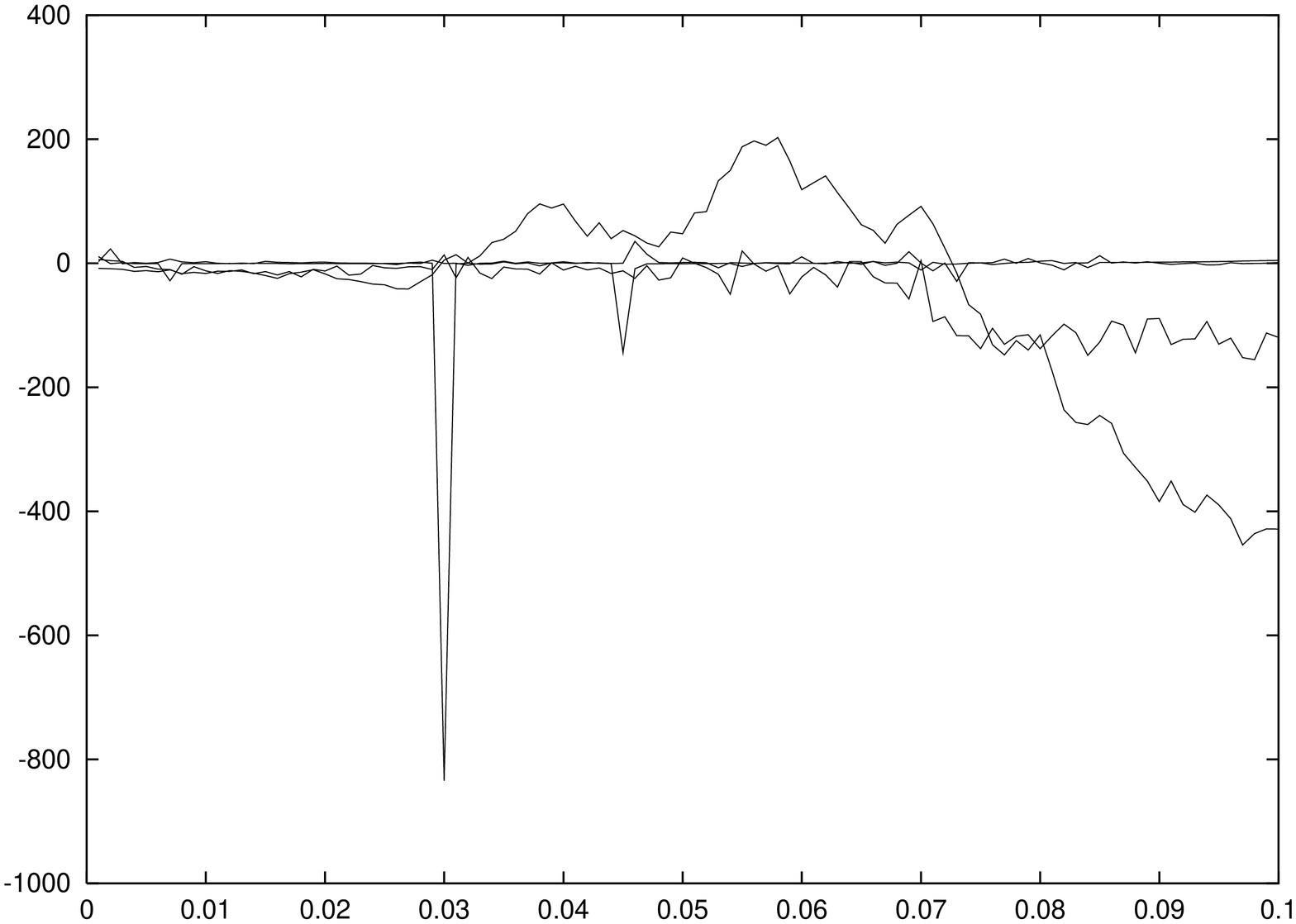}
\end{center}
\caption{ \label{res_roej1} Without projection. Evolution  in time
  ($\sec$) of (resp. from the top) the constraints on the position
  density, energy and sum of both components of the current density
  for  $\eps=0, 0.001, 0.01$ and $0.1$ for an initial condition with
  $\alpha=10^{-10}$. For the energy, deviation increases with
  $\eps$. Larger oscillations appear with higher $\eps$. Large values
  are due to the fact that initial value of the current is nearly zero.}
\end{figure}

Figures~\ref{res_roej1_ssproj} and  \ref{res_roej1}  show the
evolution of the constraints
with time for different values of $\eps$ without and with the
projection steps. The original scheme can
be seen being not conservative and dissipative. Of course, less
dissipative numerical schemes could be used, but this does not
remove the necessity for the projection step. Relative momentum
constraint values appear being large, but
one should keep in mind that these are in fact very close to zero.
What is most important is that mass and energy constraints cannot
be satisfied at the same time. This can also be seen in the
next case with initial current density.

An interesting indicator for the  behavior of the solver is by
checking if the following quantity
is linear in $\eps$ at a given time $T$ independent of $\eps$ (see
\eqref{eq:obsquad}):

\begin{equation}
\label{eqindic_L1}
  \|\rho^\eps_{h,T} -\rho_{h,T}\|_{L^1(\Omega_h)} +
  \|J^\eps_{h,T}-J_{h,T}\|_{L^1(\Omega_h)}.
\end{equation}
This is shown in Figure~\ref{indic_L1}   at $T=0.1 \sec$. The slope
grows with time.

\begin{figure}[!htp!]
\begin{center}
\includegraphics[width=7cm,angle=-90]{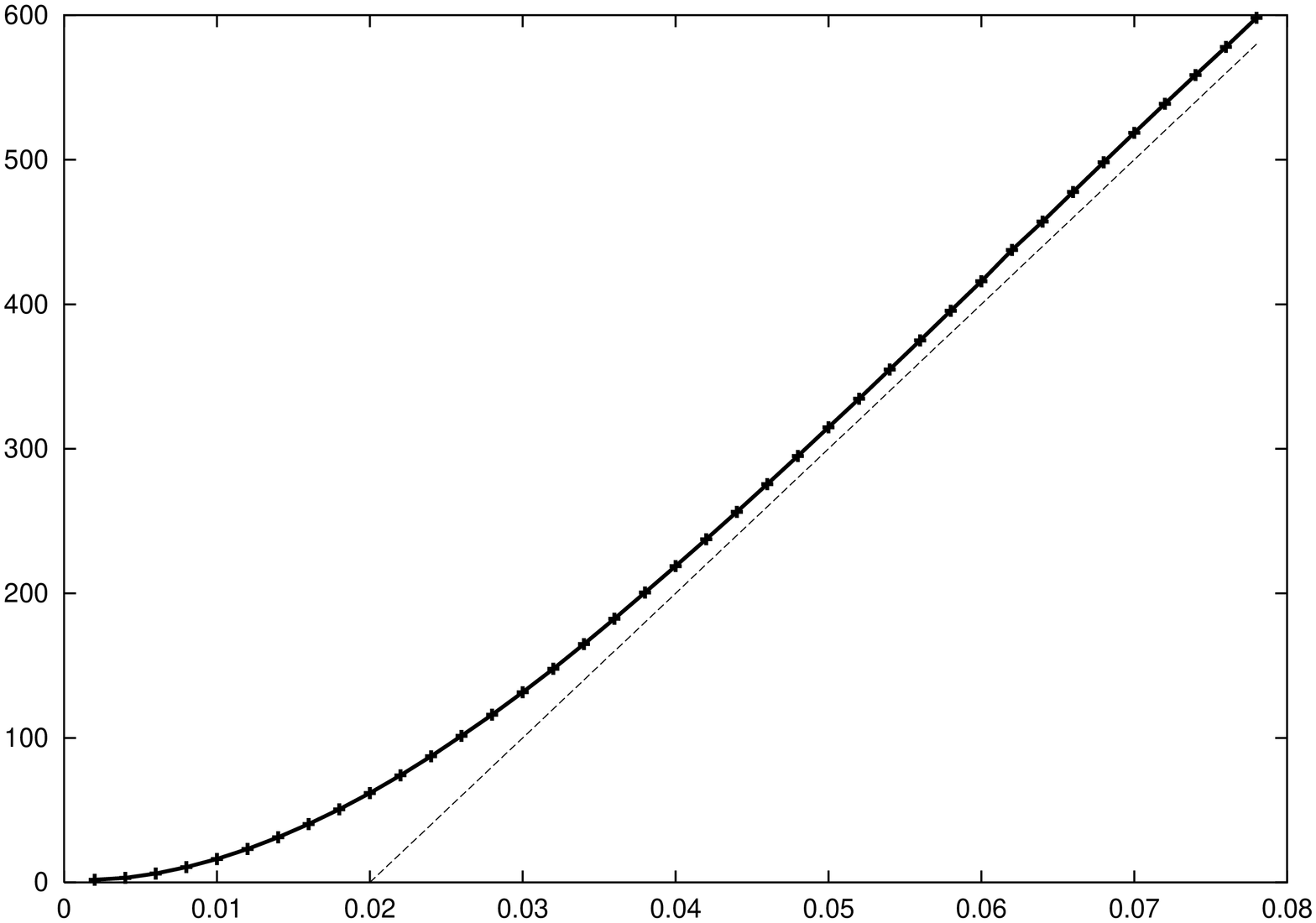}
\end{center}
\caption{ \label{indic_L1} Linear dependency of (\ref{eqindic_L1})  at $T=0.1 \sec$ with respect to $\eps$.}
\end{figure}

\begin{figure}[!htp!]
\begin{center}
\includegraphics[width=7cm,angle=-90]{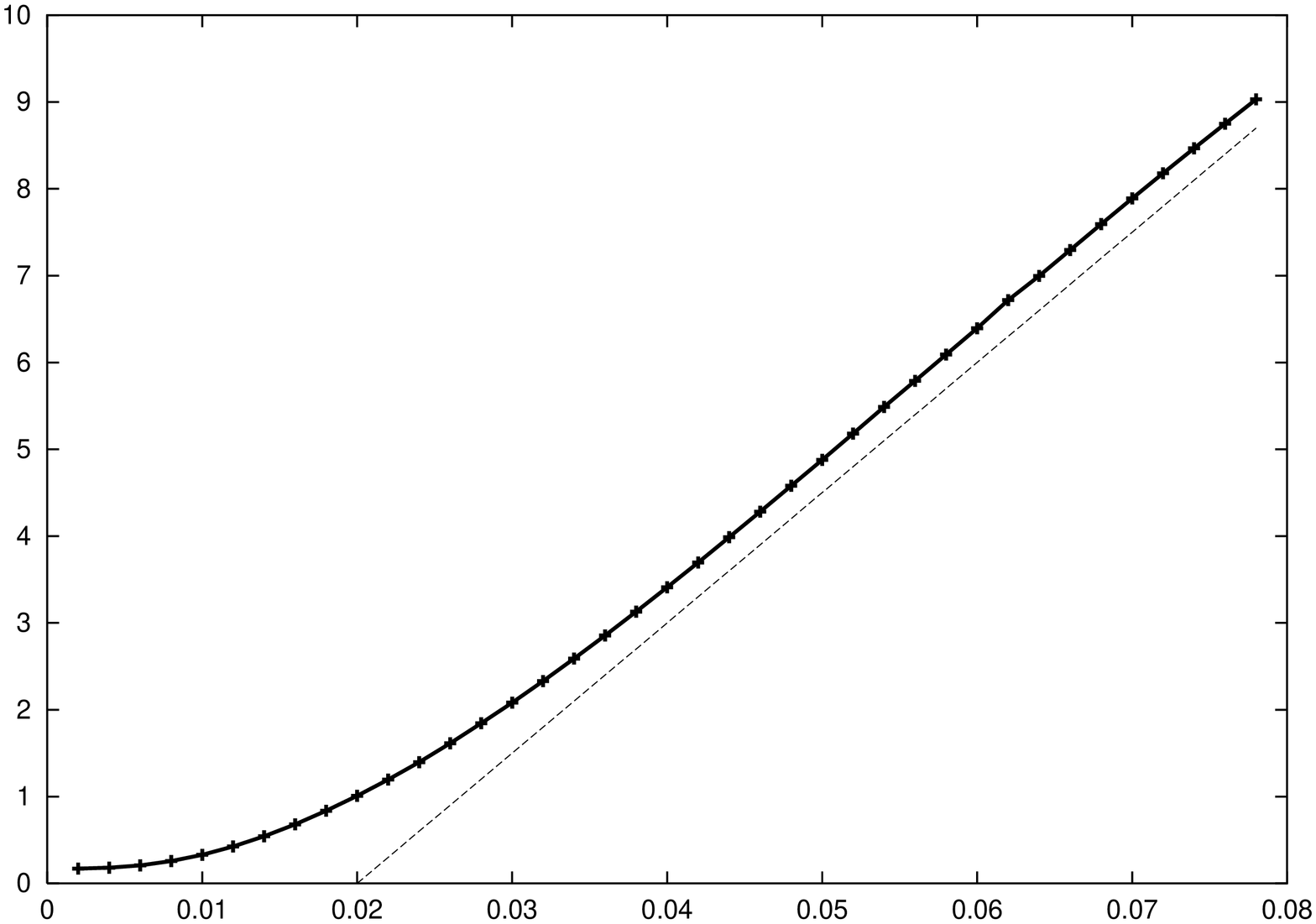}
\end{center}
\caption{ \label{indic_L2} Linear dependency of  (\ref{eqindic_L2}) at $T=0.1 \sec$ with respect to $\eps$.}
\end{figure}


In the same way, Figure~\ref{indic_L2}   shows
 the dependency with respect to $\eps$ for the following
quantity (see Theorem~\ref{theo:grenier}):

\begin{equation}
\label{eqindic_L2}
  \|a^\eps_{h,T}-a_{h,T}\|_{L^2(\Omega_h)} +  \|v^\eps_{h,T}-v_{h,T}\|_{L^2(\Omega_h)}.
\end{equation}
Again, the dependency is linear for small $\eps$ at $T=0.1 \sec$.

\subsection{Non zero initial current}

This is the same case as before but with $\alpha=10^{-2}$ and
\begin{equation}
\label{fg}
\left\{
  \begin{aligned}
    f(x)&=\exp(-80((x_1-L/2)^2+(x_2-L/2)^2)) \sin(10 x_1),\\
g(x)&=\exp(-80((x_1-L/2)^2+(x_2-L/2)^2)) \cos(10 x_1).
  \end{aligned}
     \right.
\end{equation}
Figure~\ref{iso_rj02} shows the initial position and current
densities. Figure~\ref{iso_rj2} shows the solution at $T=0.1 \sec$
for $\eps=0, 0.001, 0.01$ and $0.1$. Figure~\ref{res_roej2}  shows
the evolution of the position density, energy and current density
  constraints with time for different values of $\eps$ when
  only  mass through $I_1$ and the current density through vector
  $I_3$ have been maintained. 

\begin{figure}[!htp!]
\begin{center}
\begin{tabular}{ll}
\includegraphics[width=5cm,angle=-90]{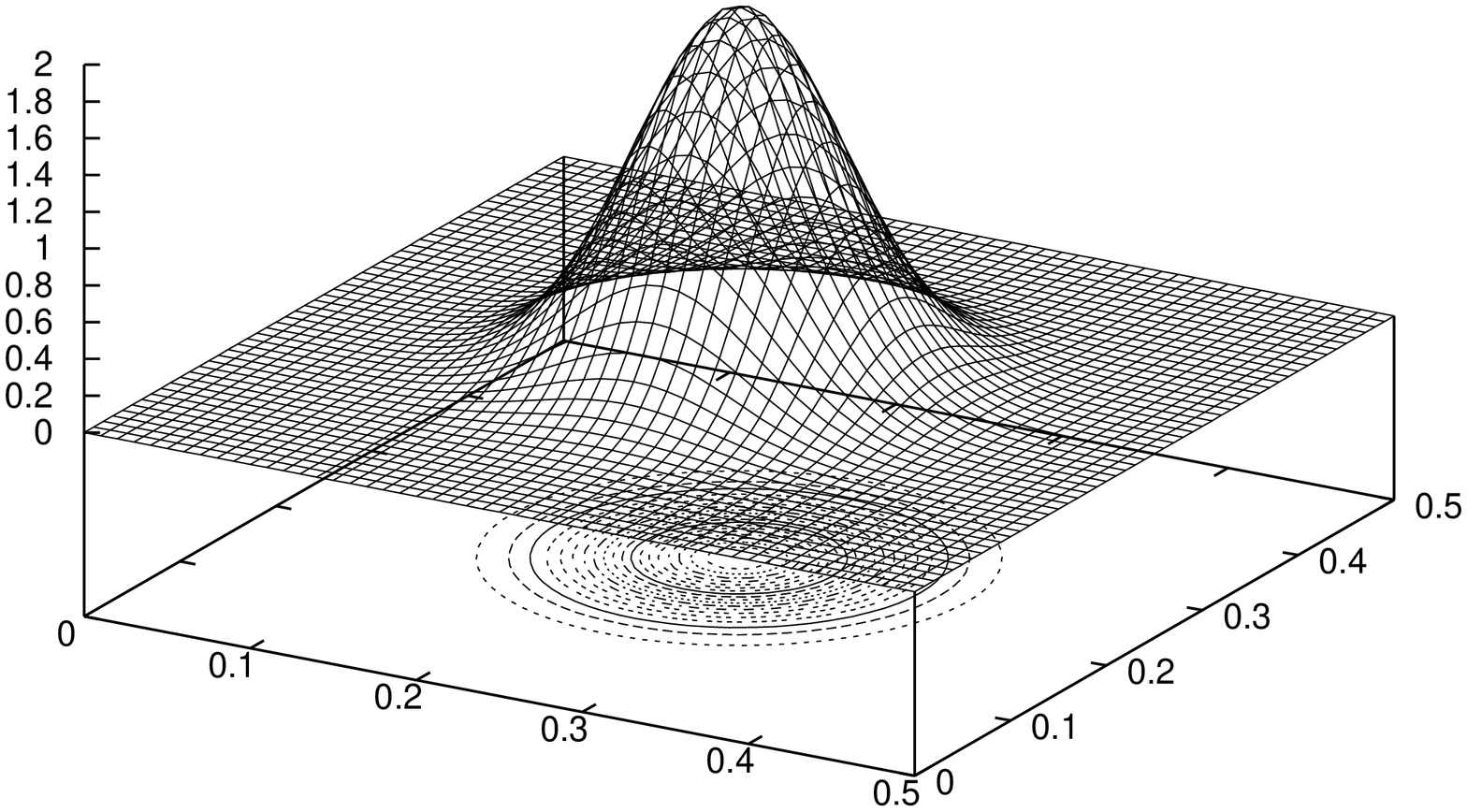}&
\hspace{-1.5cm}
\includegraphics[width=5cm,angle=-90]{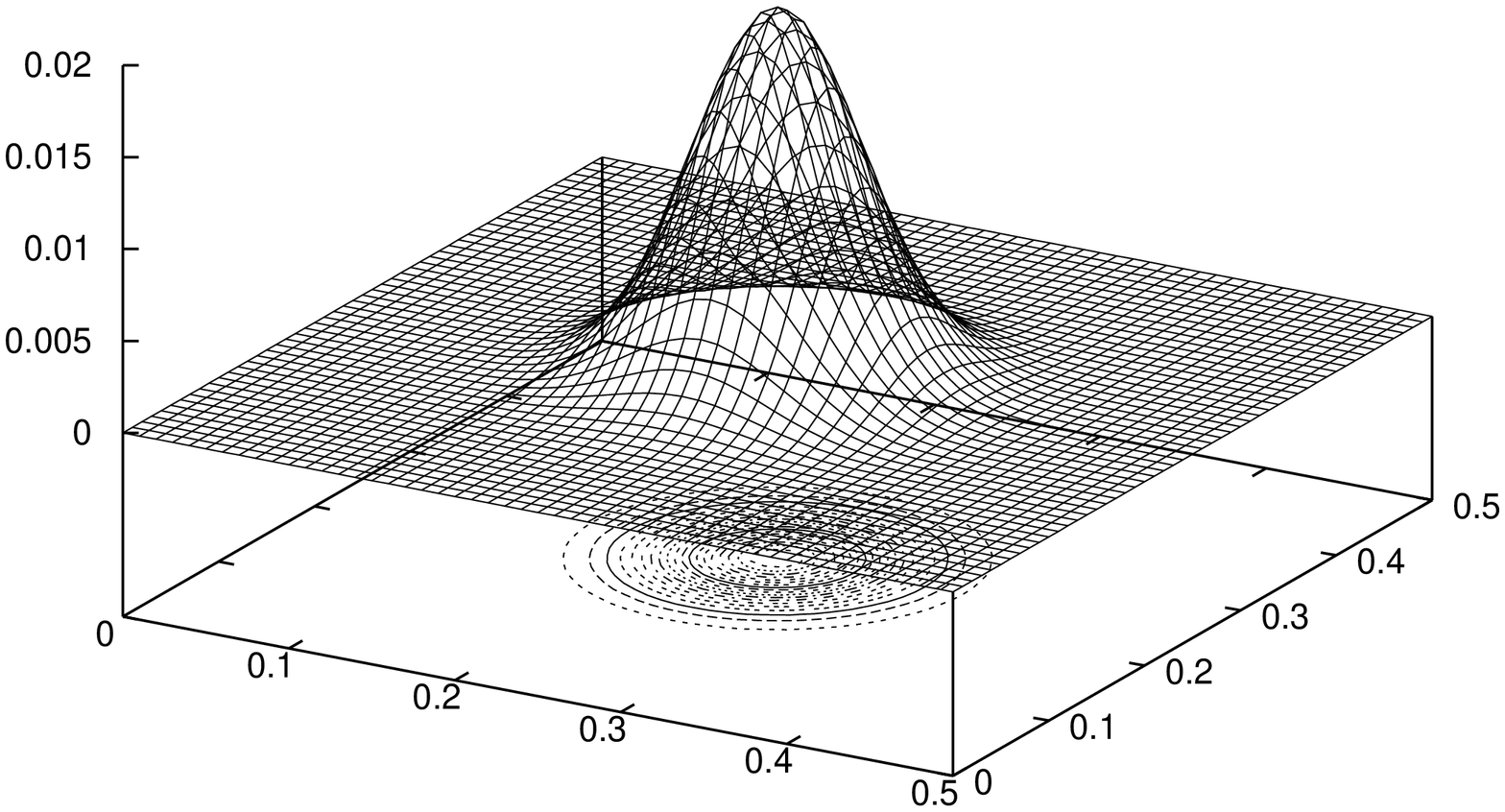}
\end{tabular}
\caption{ \label{iso_rj02} Initial position (left) and norm of the current density
vector (right)  with $\alpha=0.01$.}
\end{center}
\end{figure}

\begin{figure}[!htp!]
\begin{center}
\begin{tabular}{ll}
\vspace{-1.5cm}
\includegraphics[width=5cm,angle=-90]{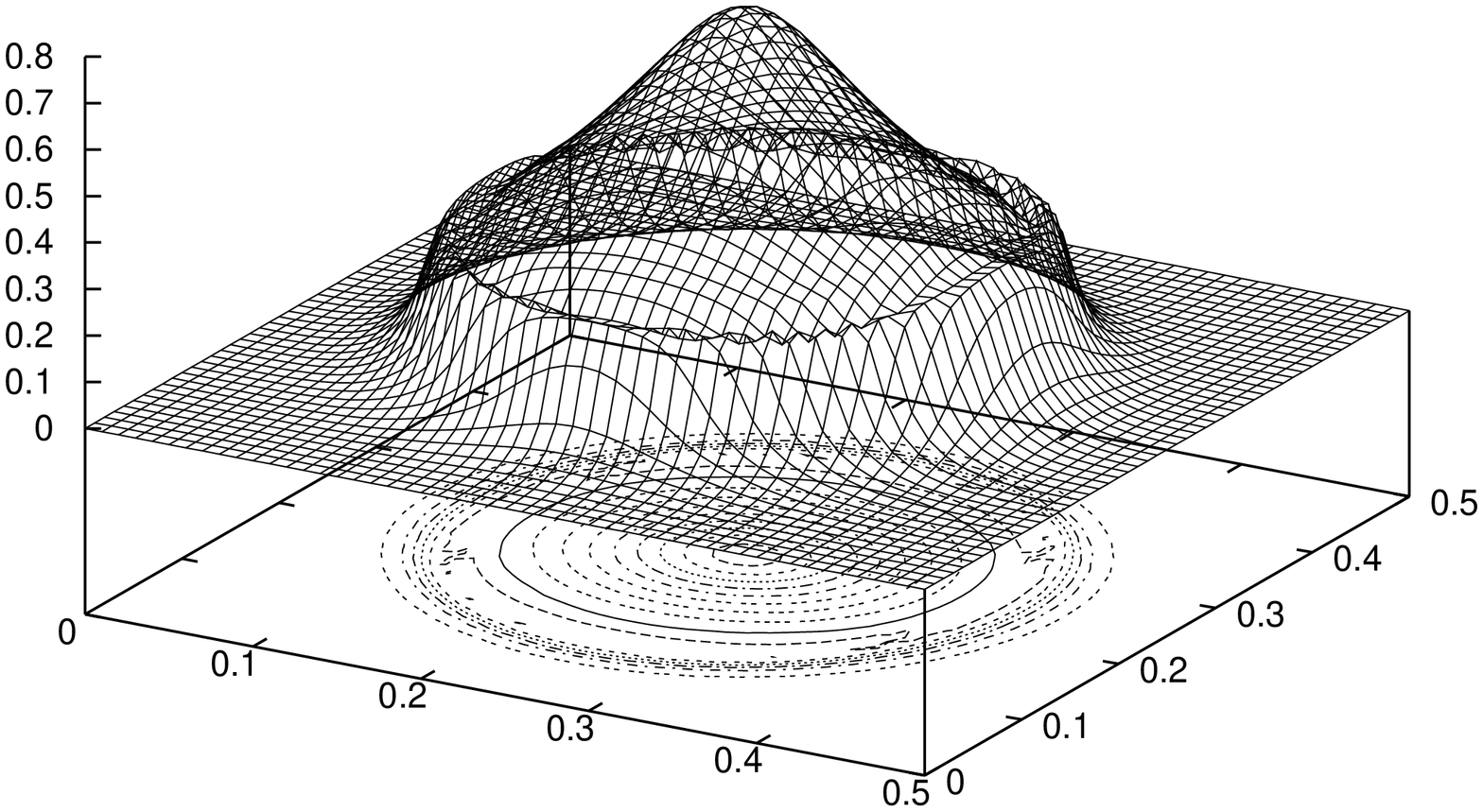}&
\hspace{-1.5cm}
\includegraphics[width=5cm,angle=-90]{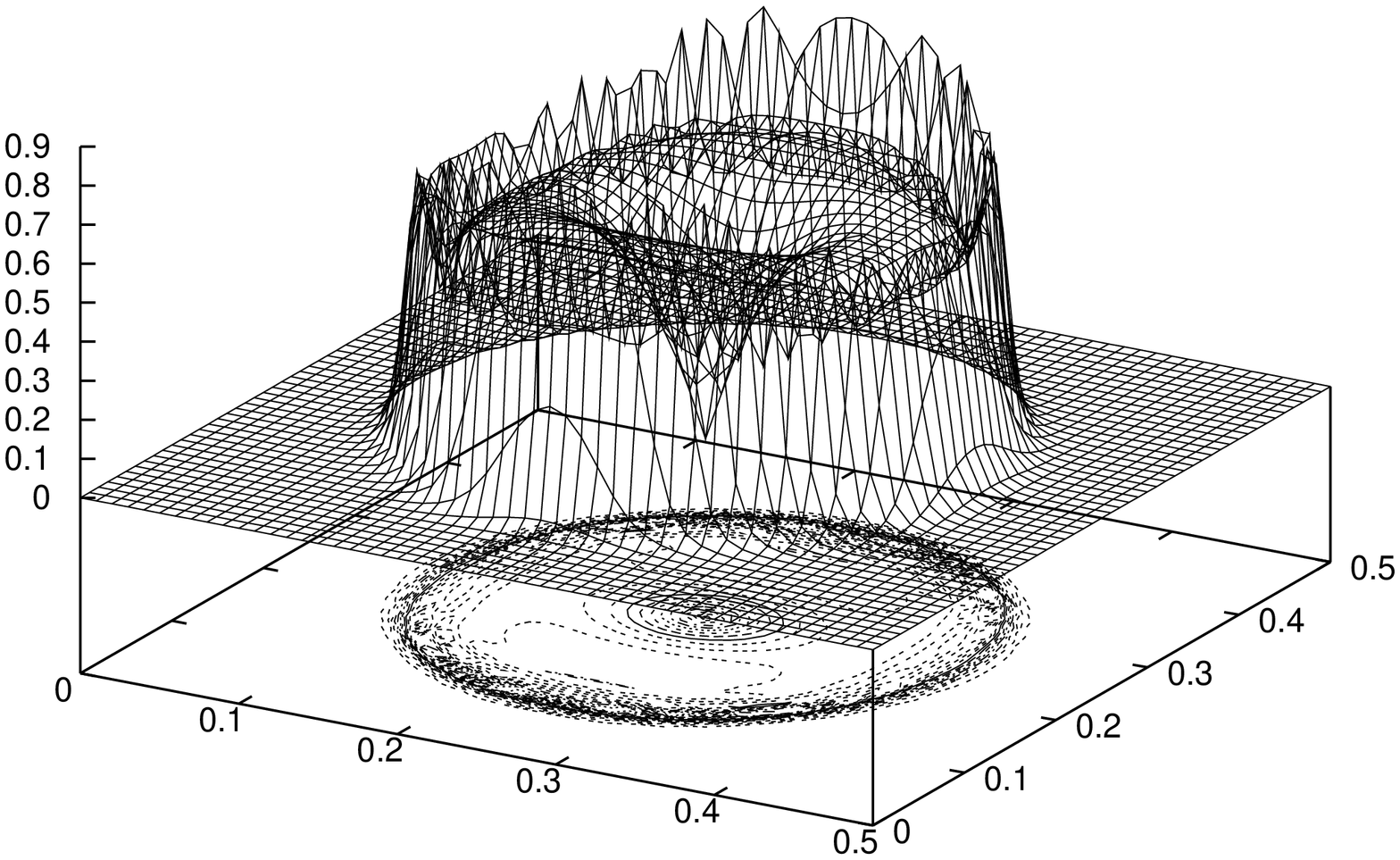}\\
\vspace{-1.5cm}
\includegraphics[width=5cm,angle=-90]{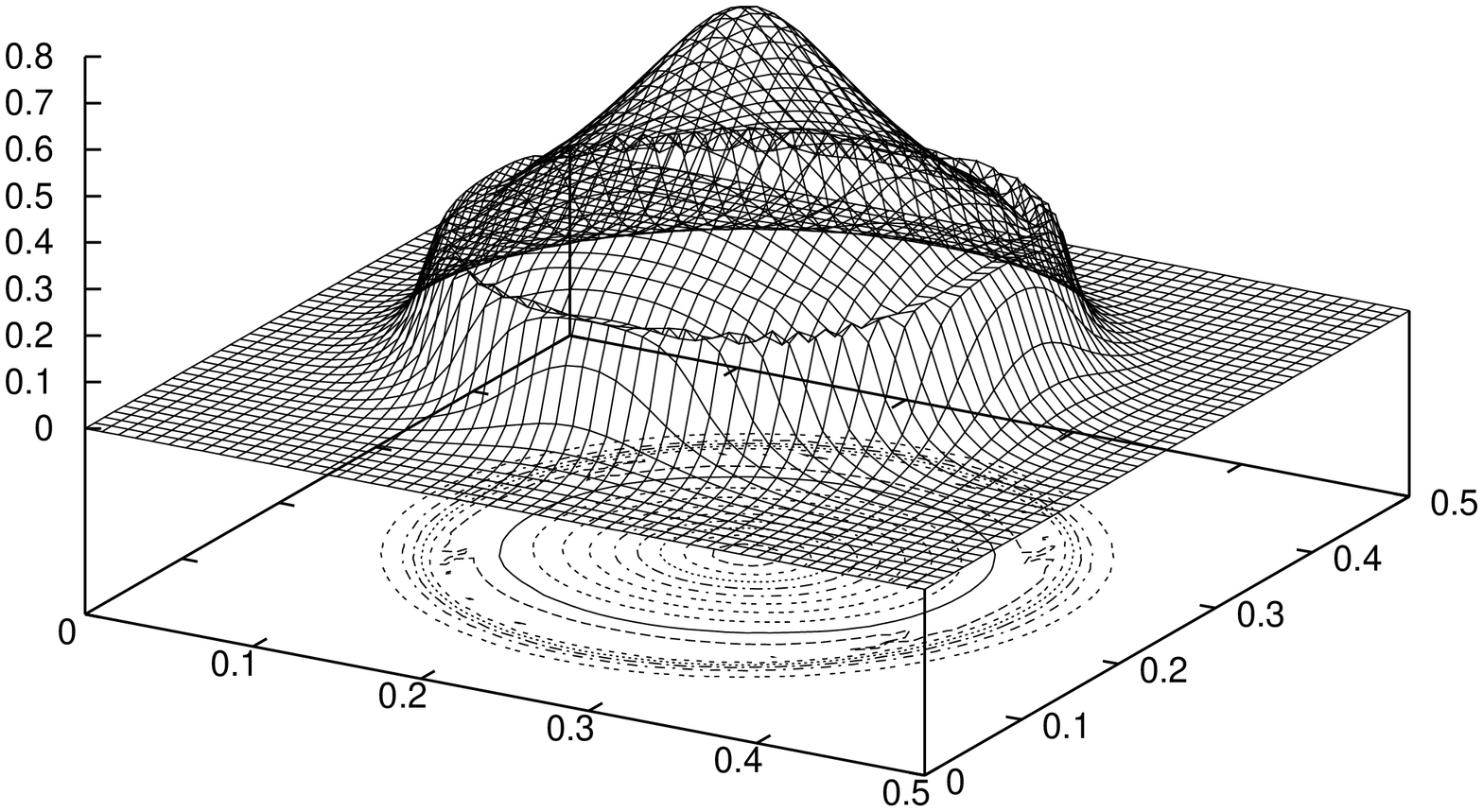}&
\hspace{-1.5cm}
\includegraphics[width=5cm,angle=-90]{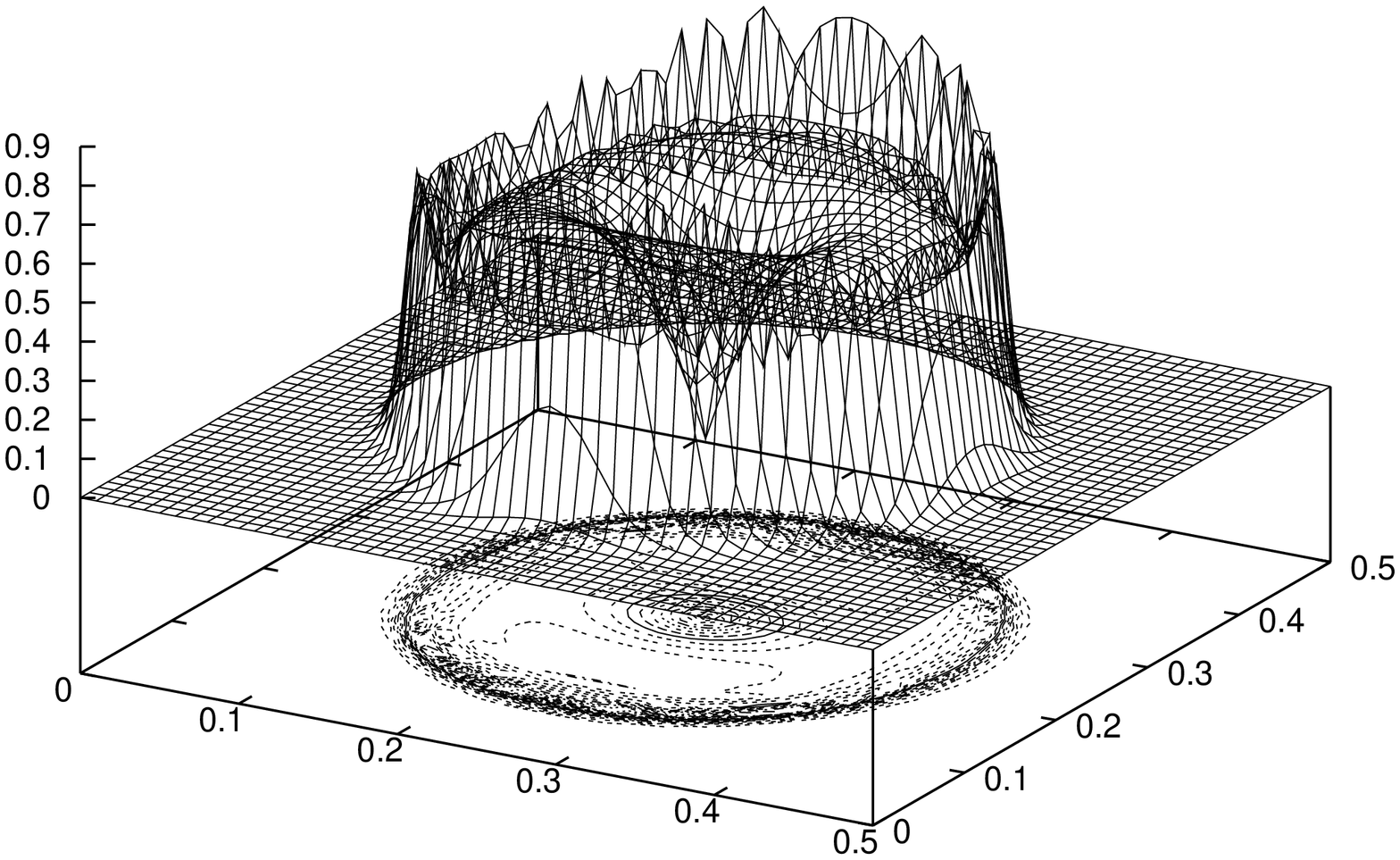}\\
\vspace{-1.5cm}
\includegraphics[width=5cm,angle=-90]{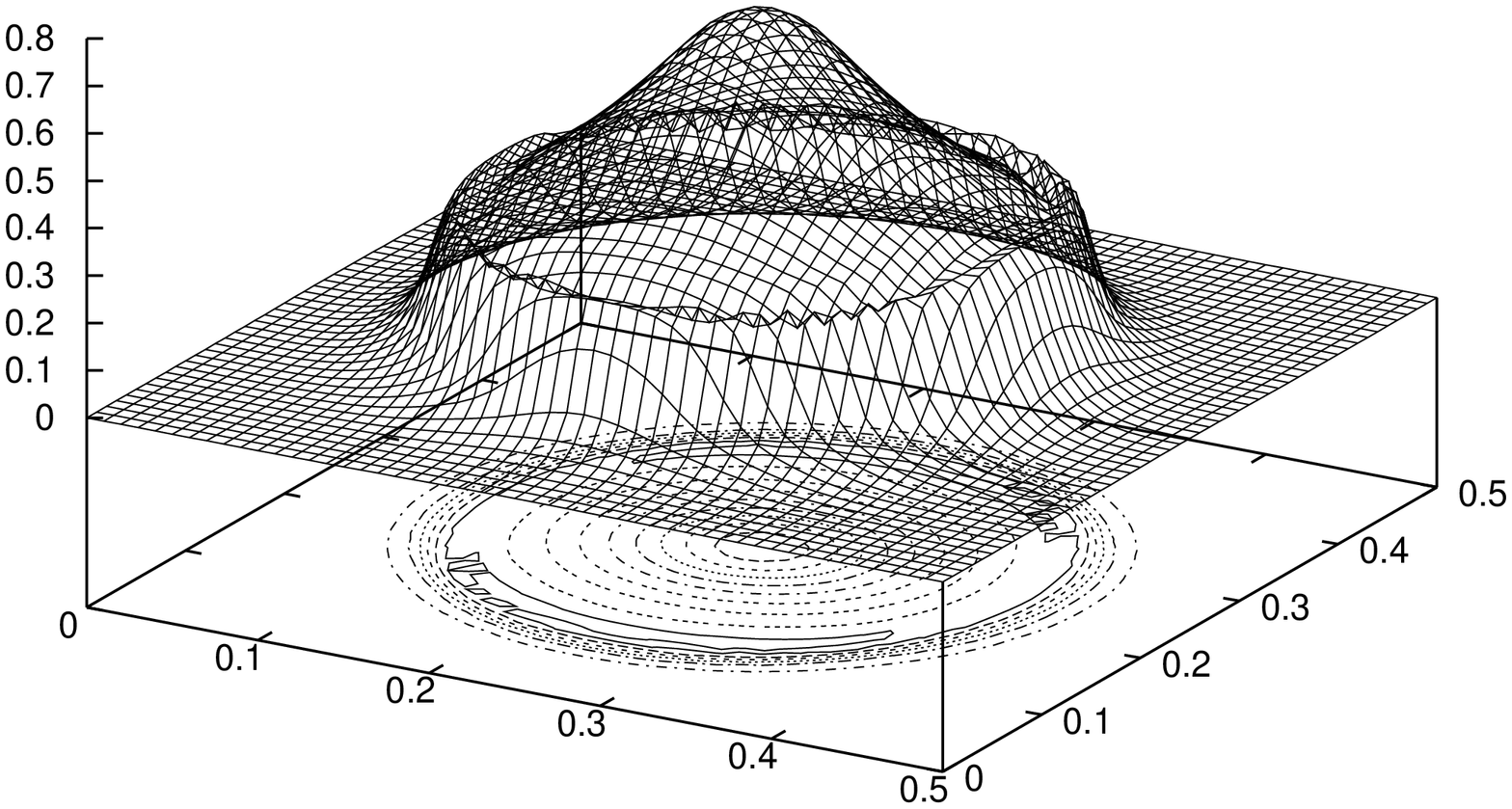}&
\hspace{-1.5cm}
\includegraphics[width=5cm,angle=-90]{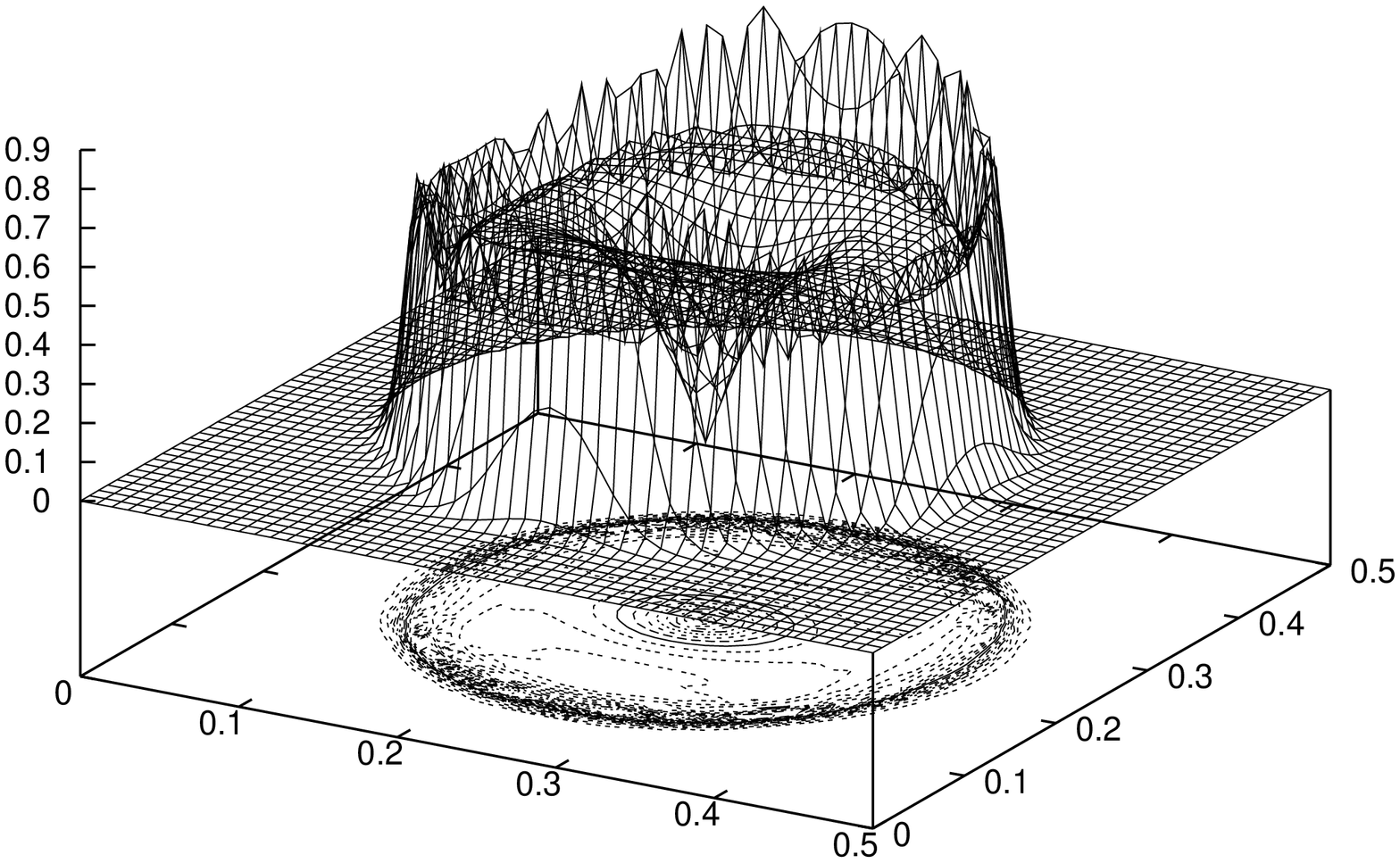}\\
\includegraphics[width=5cm,angle=-90]{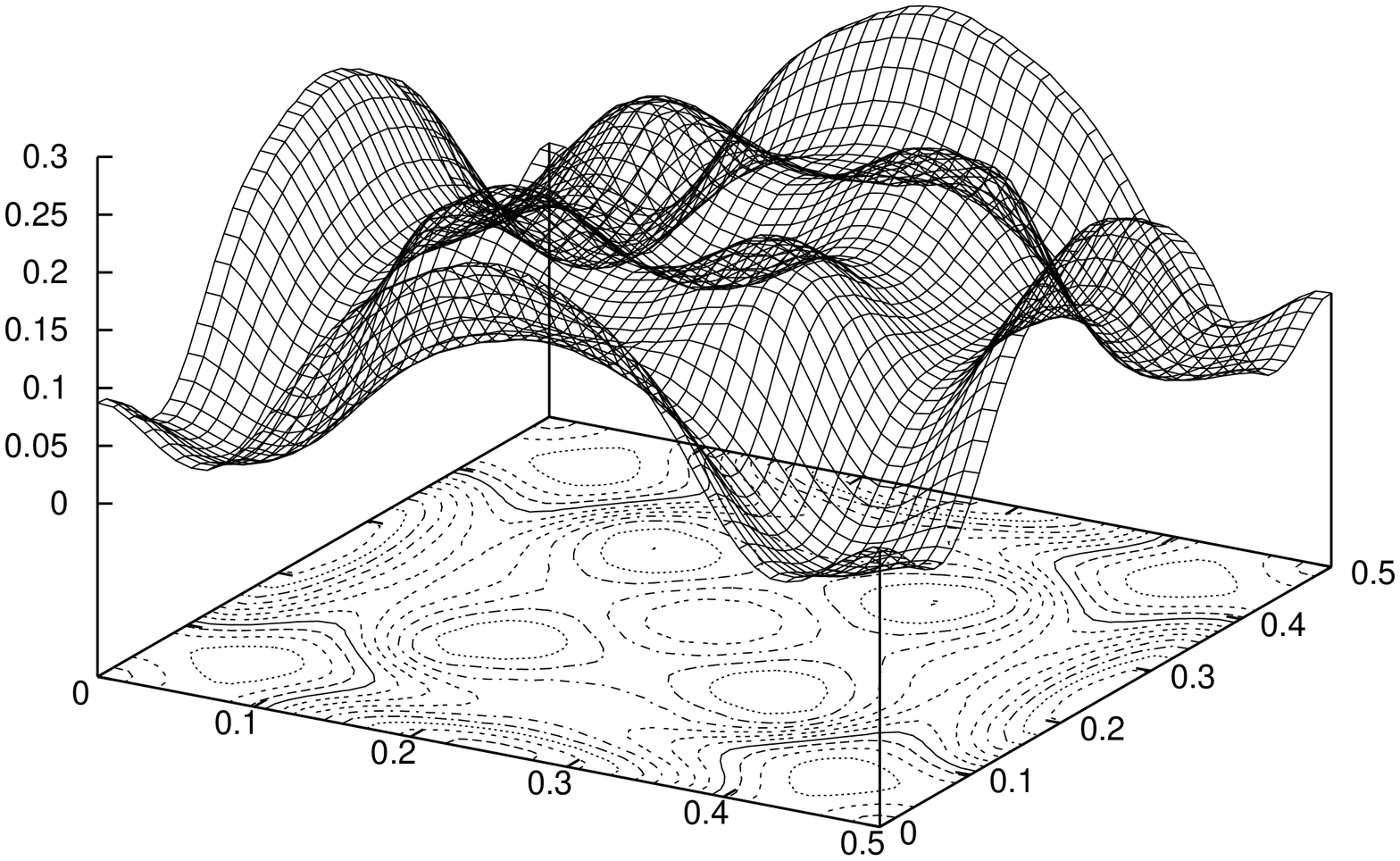}&
\hspace{-1.5cm}
\includegraphics[width=5cm,angle=-90]{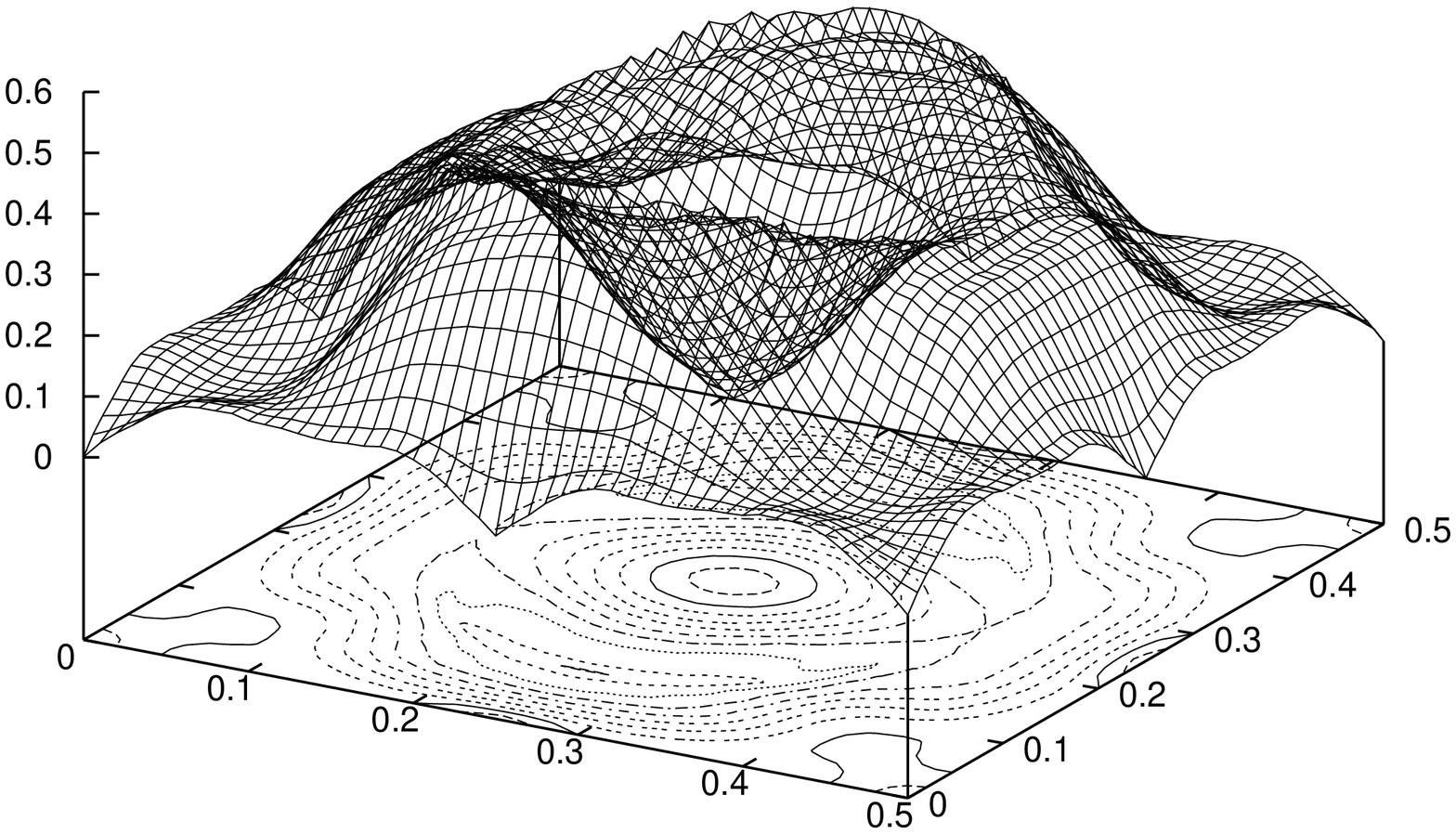}
\end{tabular}
\caption{ \label{iso_rj2} Position (left column) and current density (right column)
at $T=0.1 \sec$ for (resp. from the top) $\eps=0, 0.001, 0.01$ and $0.1$ with $\alpha=0.01$.}
\end{center}
\end{figure}

\begin{figure}[!htp!]
\begin{center}
\includegraphics[width=6cm,angle=-90]{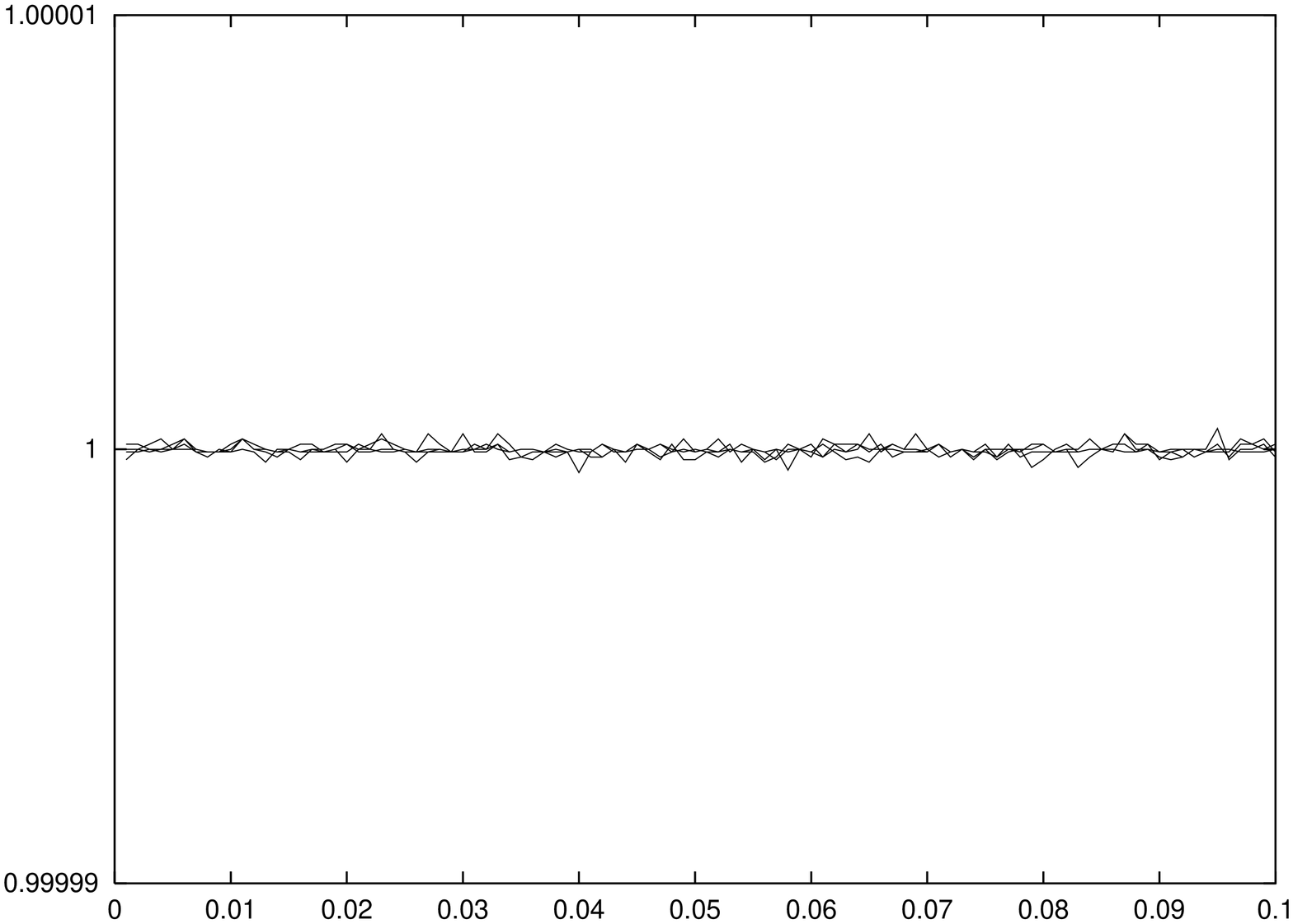}
\includegraphics[width=6cm,angle=-90]{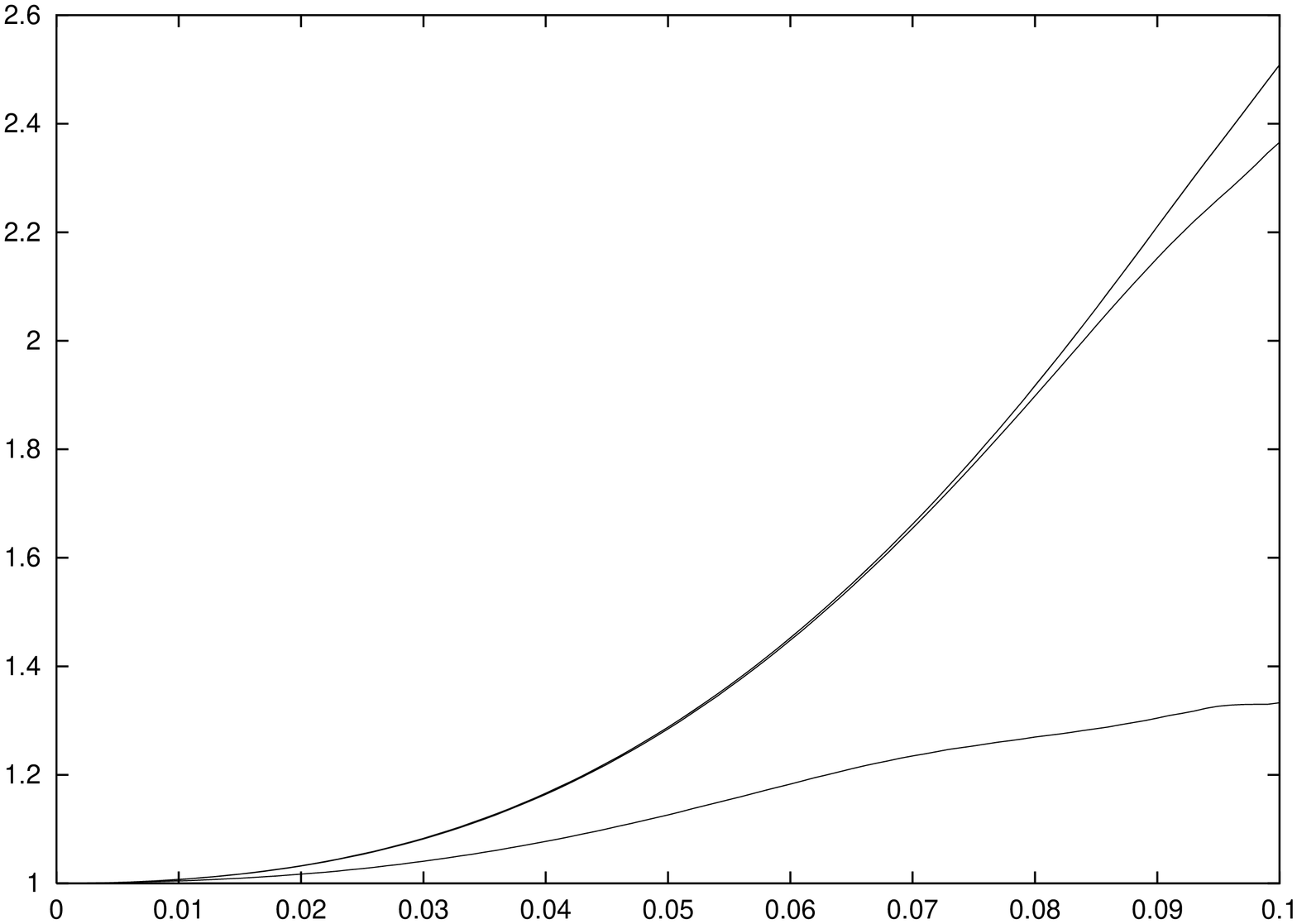}
\includegraphics[width=6cm,angle=-90]{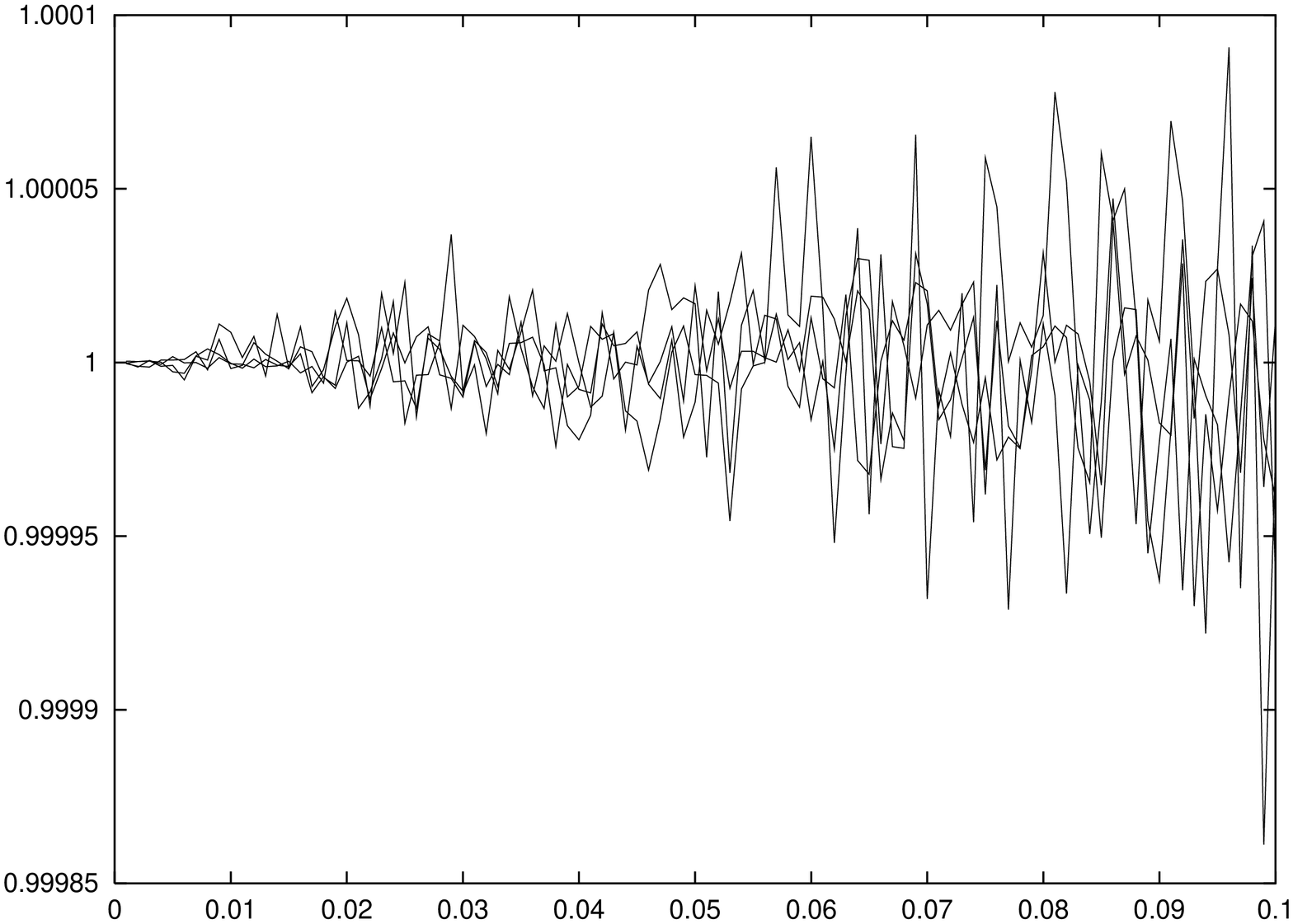}
\end{center}
\caption{ \label{res_roej2} Evolution  in time ($\sec$) of the
  constraints on the
position density, energy and sum of both components of the current
density (resp. from the top) for  $\eps=0, 0.001, 0.01$ and $0.1$
for an initial condition with $\alpha=0.01$.}
\end{figure}

Figures~\ref{indic_L1_2} and \ref{indic_L2_2} show that
indicators \eqref{eqindic_L1} and \eqref{eqindic_L2} are still linear
with respect to $\eps$ but on a shorter range close to zero.


\begin{figure}[!htp!]
\begin{center}
\includegraphics[width=7cm,angle=-90]{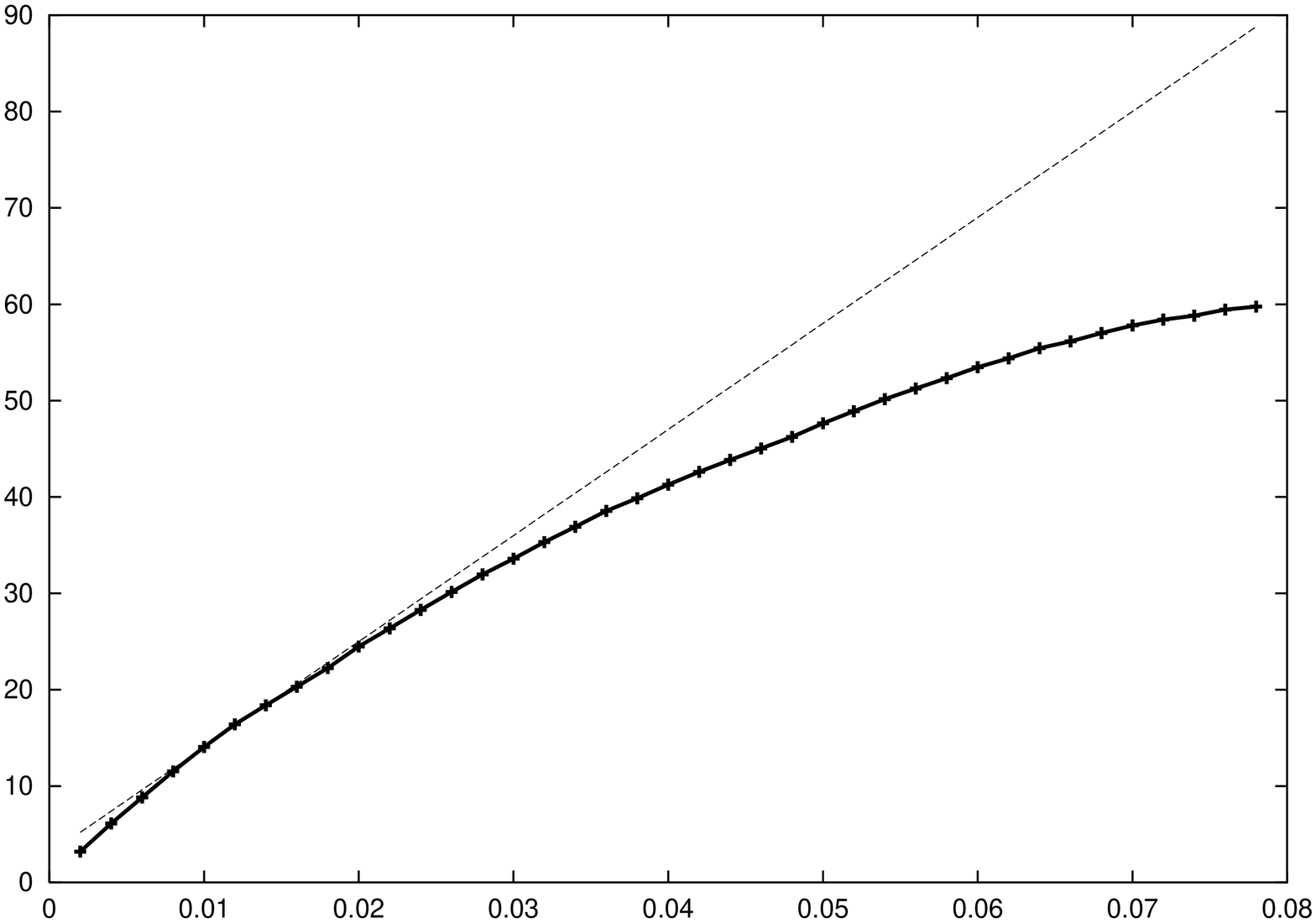}
\end{center}
\caption{ \label{indic_L1_2} Linear dependency of (\ref{eqindic_L1})  at $T=0.1 \sec$ with respect to $\eps$.
}
\end{figure}

\begin{figure}[!htp!]
\begin{center}
\includegraphics[width=7cm,angle=-90]{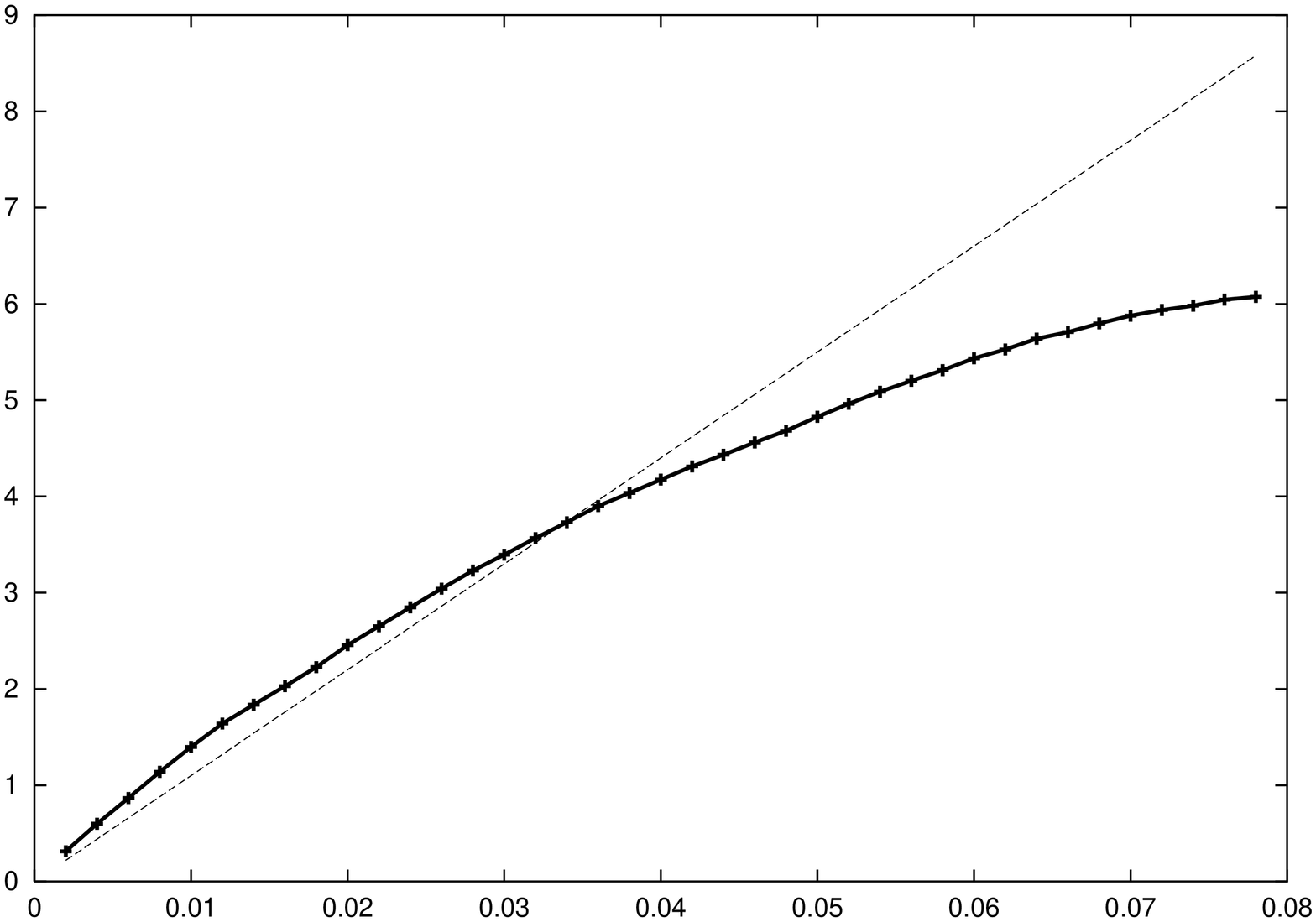}
\end{center}
\caption{ \label{indic_L2_2} Linear dependency of  (\ref{eqindic_L2}) at $T=0.1 \sec$ with respect to $\eps$.}
\end{figure}


\subsection{$a^\eps$ changing sign}\label{sec:changing}

To introduce a changing sign initial data for $a^\eps$, we consider an
initial condition given by
$a_0(x)=(\exp(-320((x_1-L/2)^2+(x_2-L/2)^2))-\exp(-320((x_1-L/2)^2+(x_2-L/2)^2)))
(1+i)$. This initial amplitude changes signs: the set where it is zero
corresponds to the \emph{presence of vacuum} in the hydrodynamical point of
view. 
The initial current is as for the previous case with $f$ and $g$ given
in \eqref{fg}. 
Figure~\ref{iso_rj03} shows the initial position and current
densities. Figure~\ref{iso_rj3} shows the solution at $T=0.05 \sec$
for $\eps=0, 0.001, 0.01$ and $0.1$. Figure~\ref{res_roej3}  shows
the evolution of the position density, energy and current density
  constraints with time for different values of $\eps$ when
  only  mass through $I_1$ and the current density through vector
  $I_3$ have been maintained. 


\begin{figure}[!htp!]
\begin{center}
\begin{tabular}{ll}
\includegraphics[width=5cm,angle=-90]{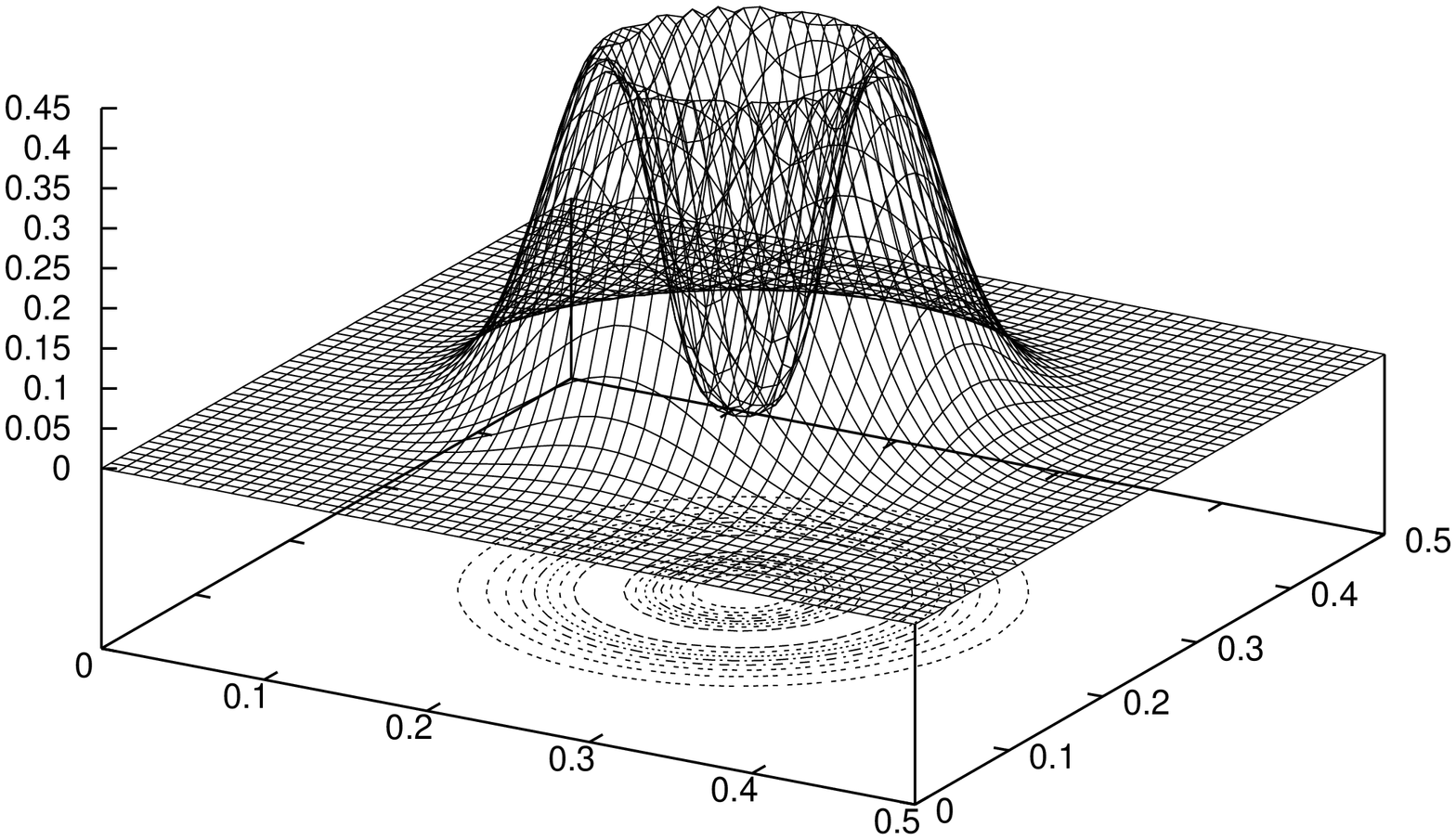}&
\hspace{-1.5cm}
\includegraphics[width=5cm,angle=-90]{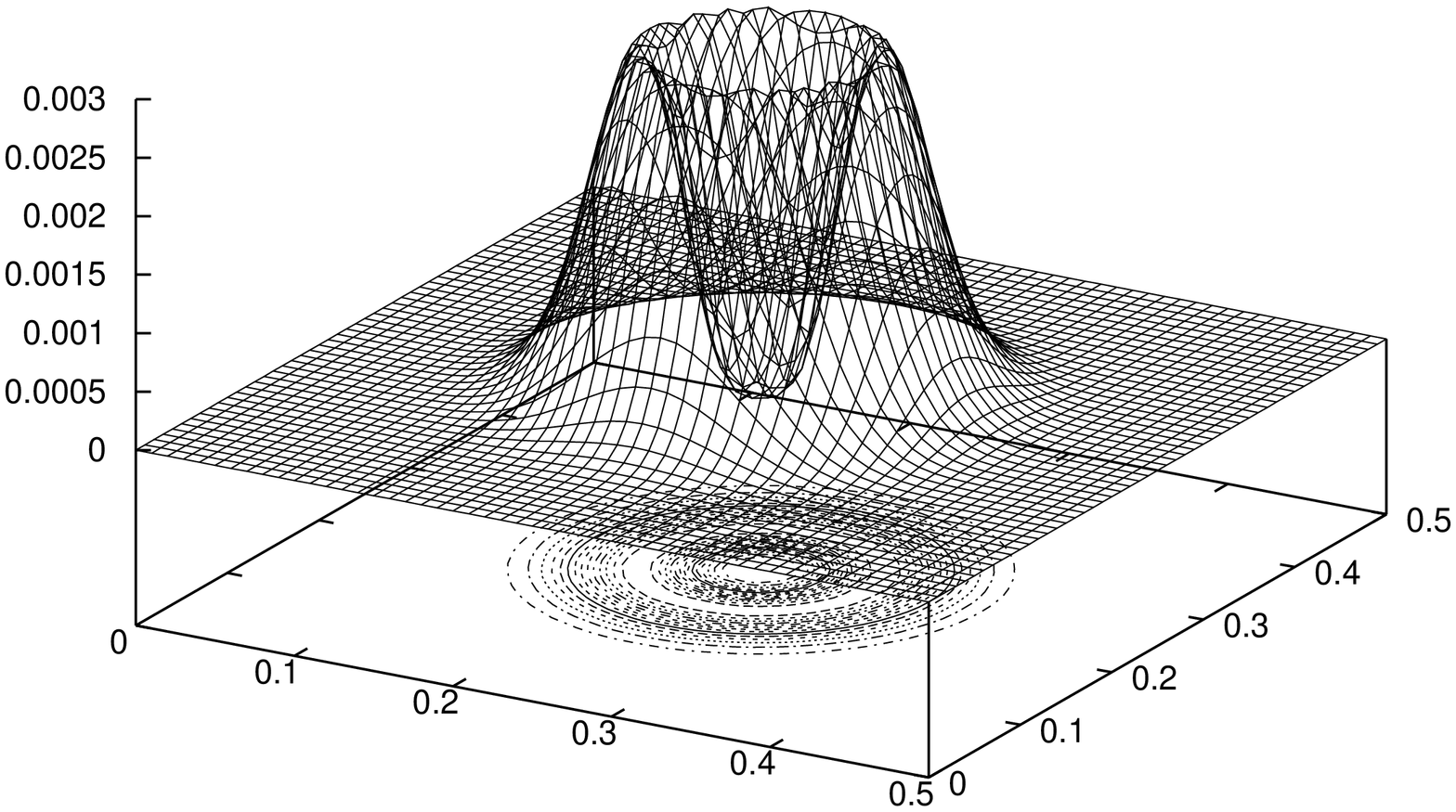}
\end{tabular}
\caption{ \label{iso_rj03} Initial position (left) and norm of the current density
vector (right)  with varying sign initial $a^\eps$.}
\end{center}
\end{figure}

\begin{figure}[!htp!]
\begin{center}
\begin{tabular}{ll}
\vspace{-1.5cm}
\includegraphics[width=5cm,angle=-90]{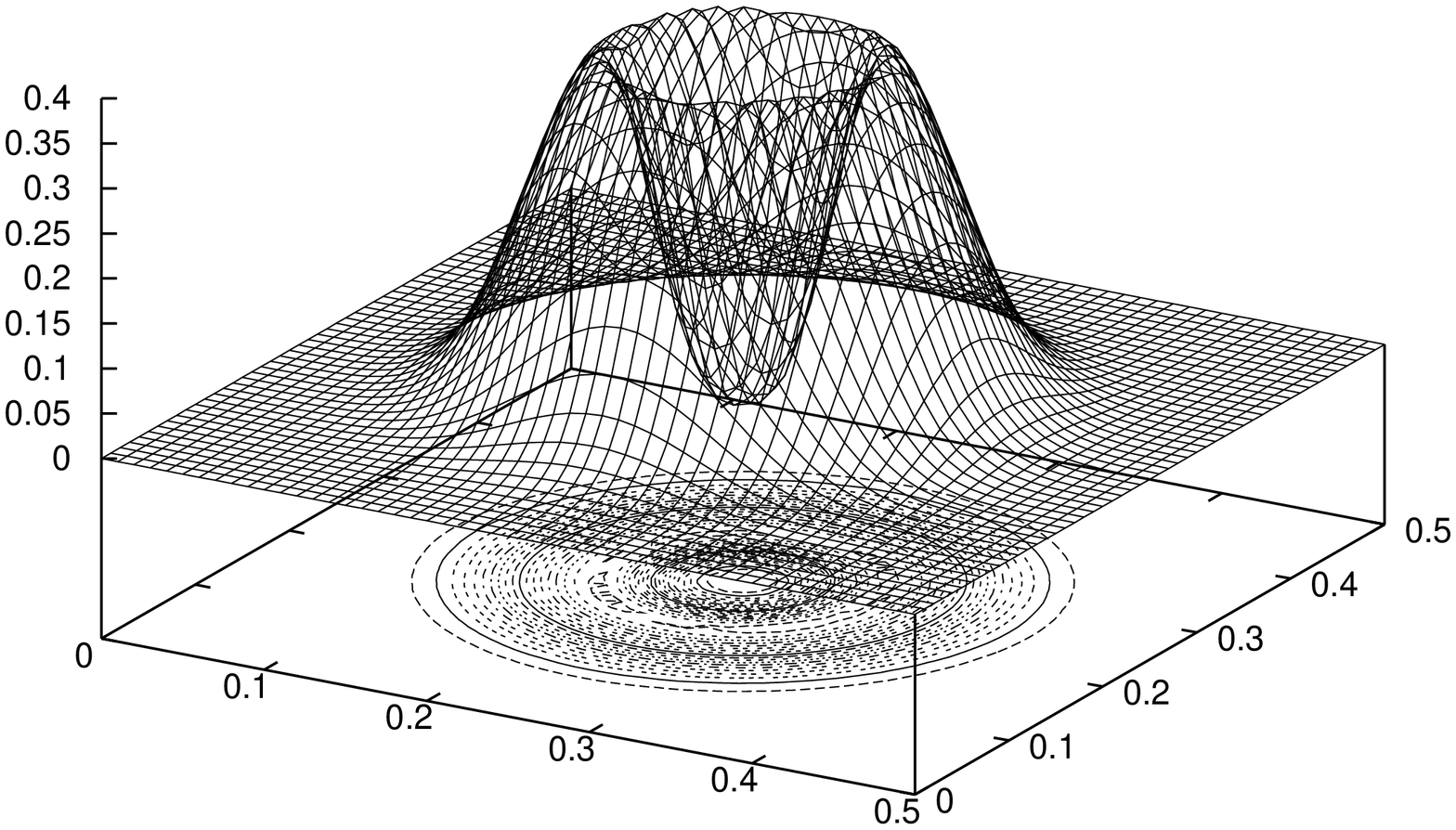}&
\hspace{-1.5cm}
\includegraphics[width=5cm,angle=-90]{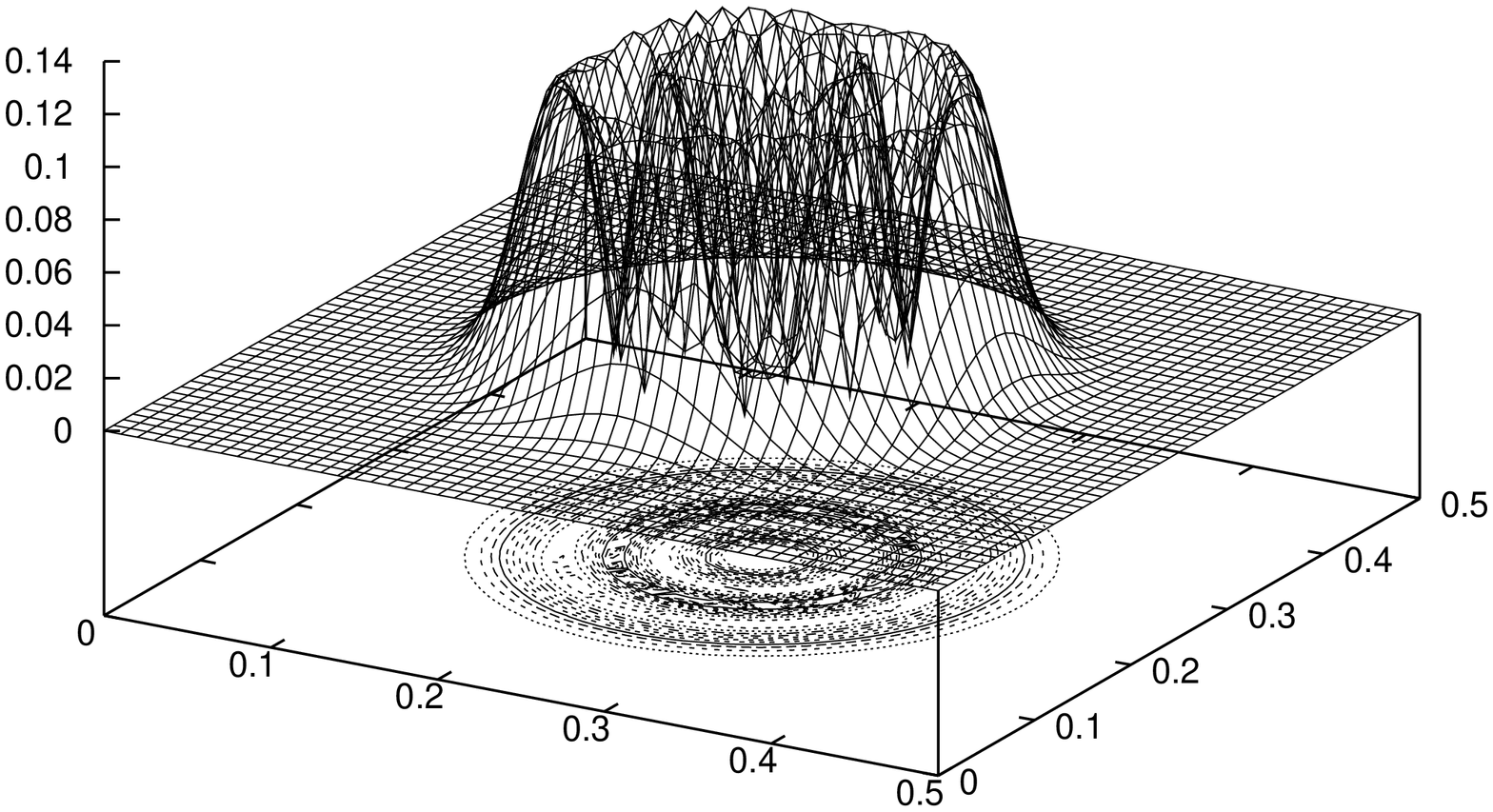}\\
\vspace{-1.5cm}
\includegraphics[width=5cm,angle=-90]{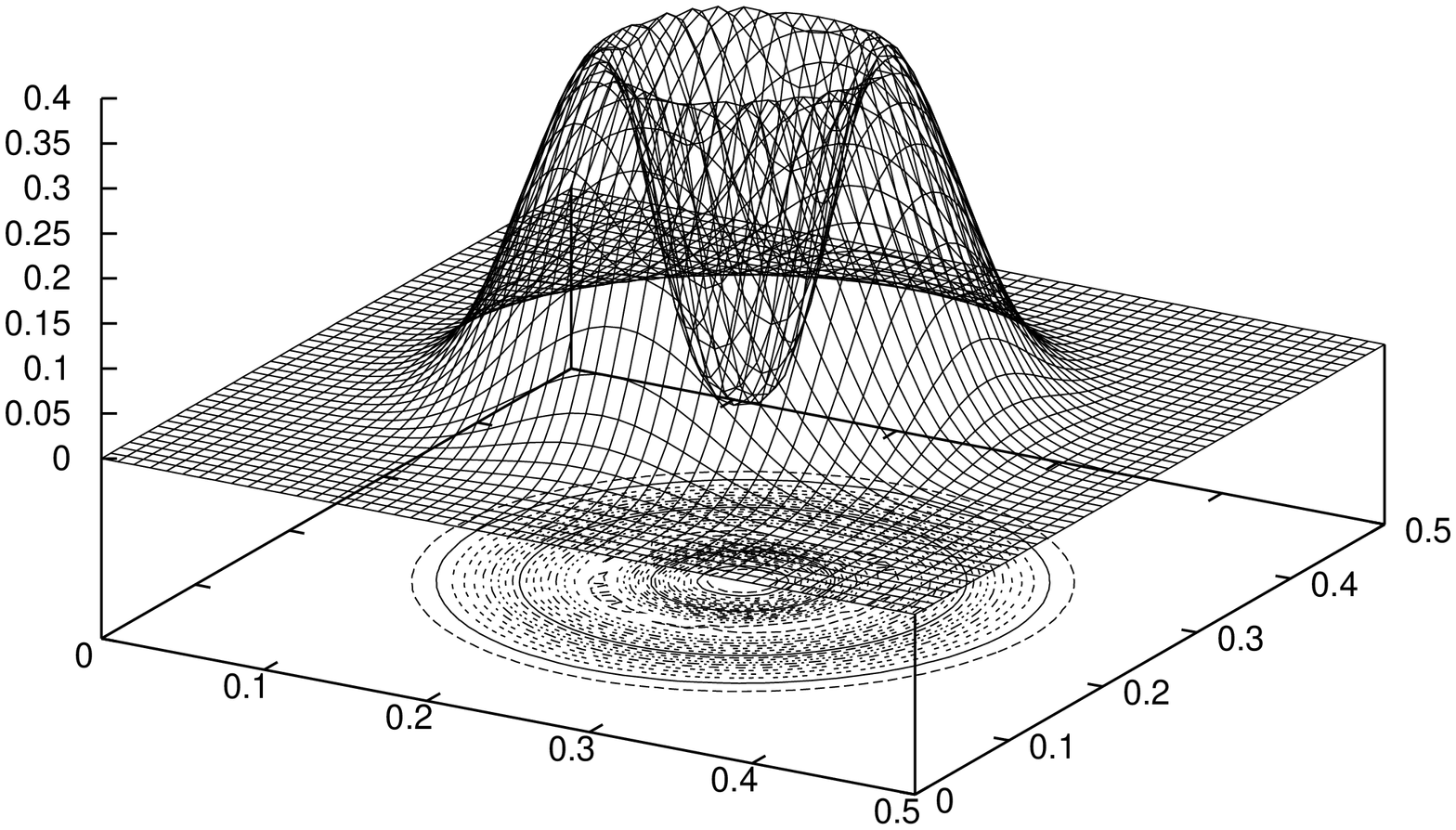}&
\hspace{-1.5cm}
\includegraphics[width=5cm,angle=-90]{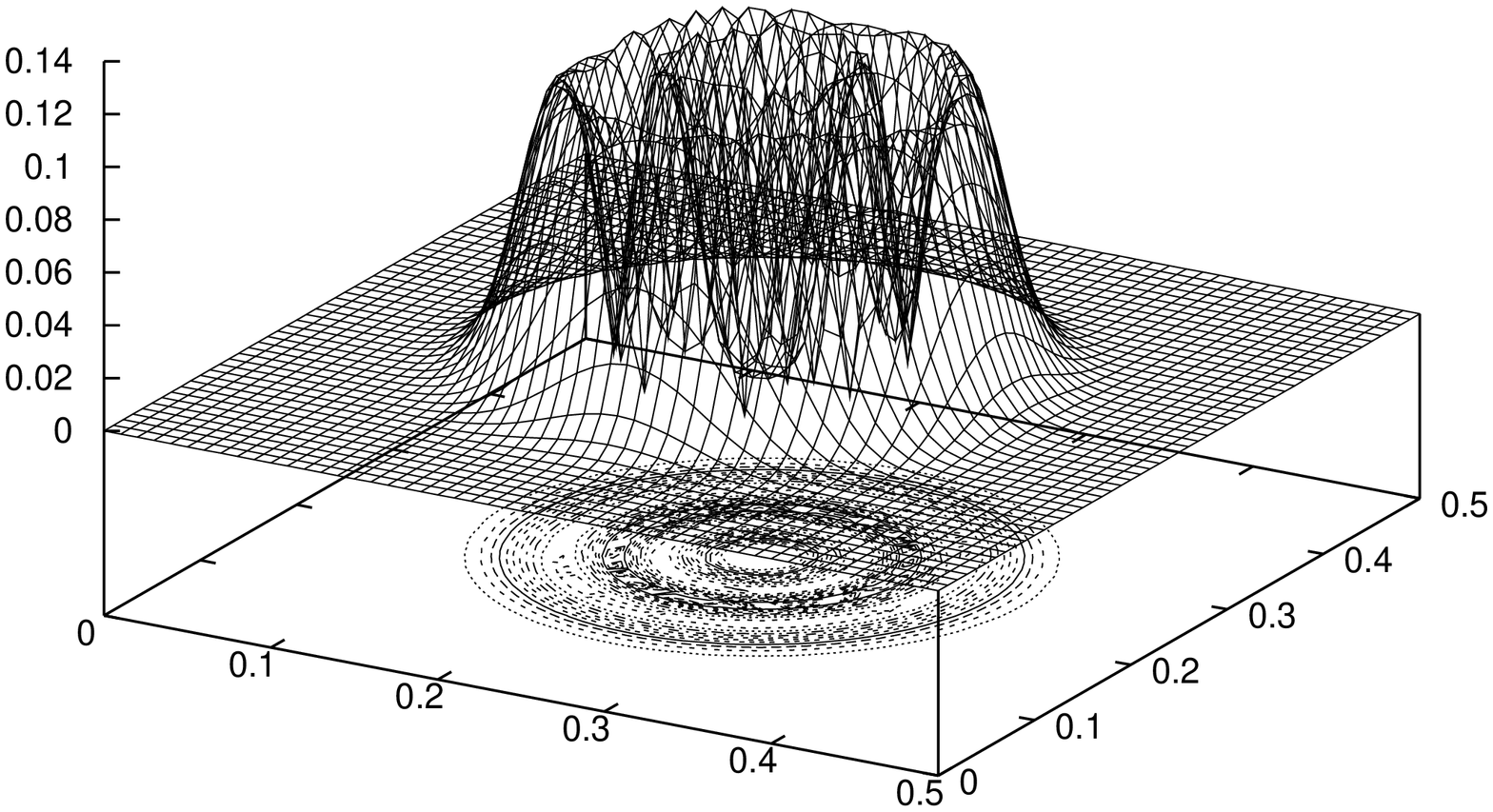}\\
\vspace{-1.5cm}
\includegraphics[width=5cm,angle=-90]{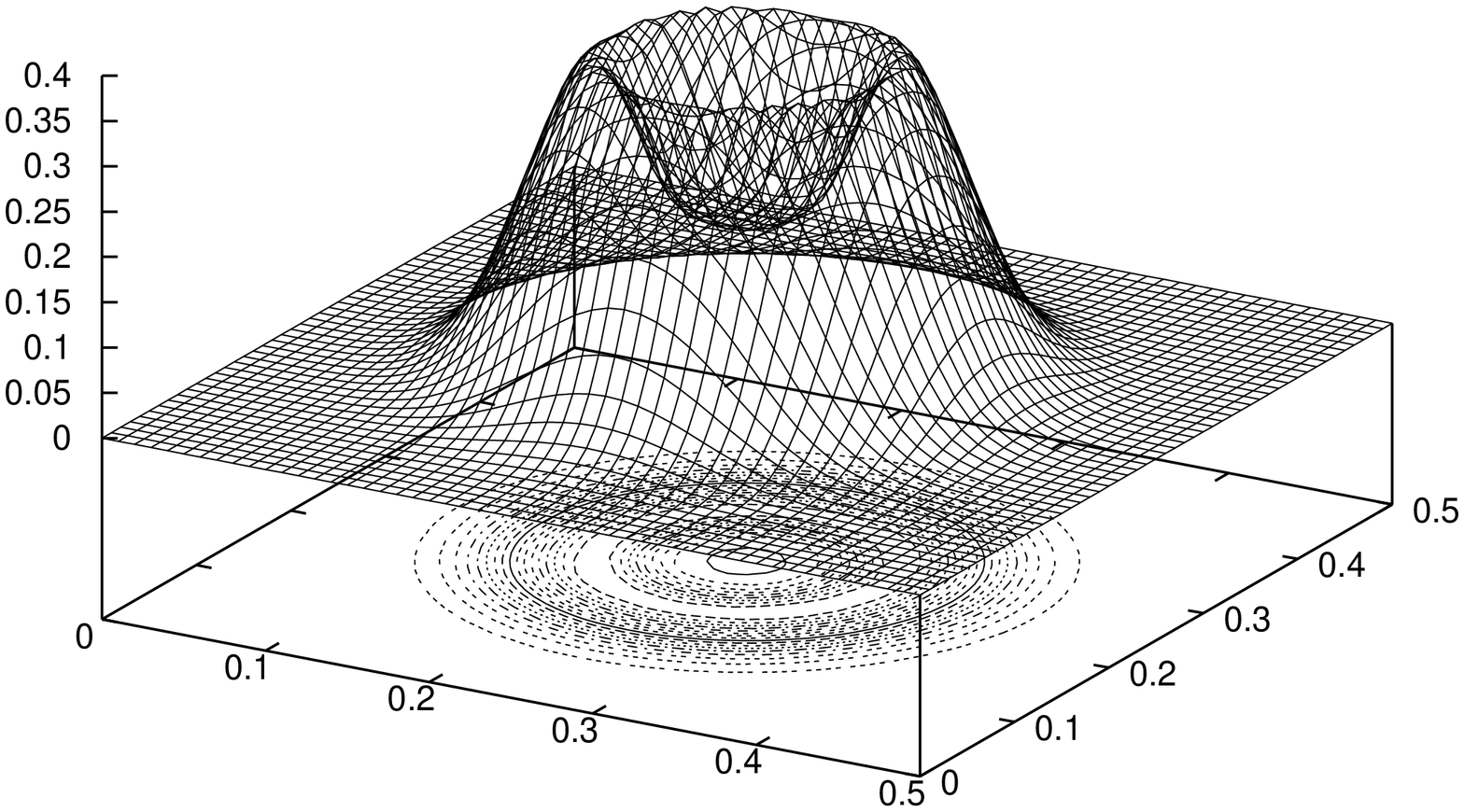}&
\hspace{-1.5cm}
\includegraphics[width=5cm,angle=-90]{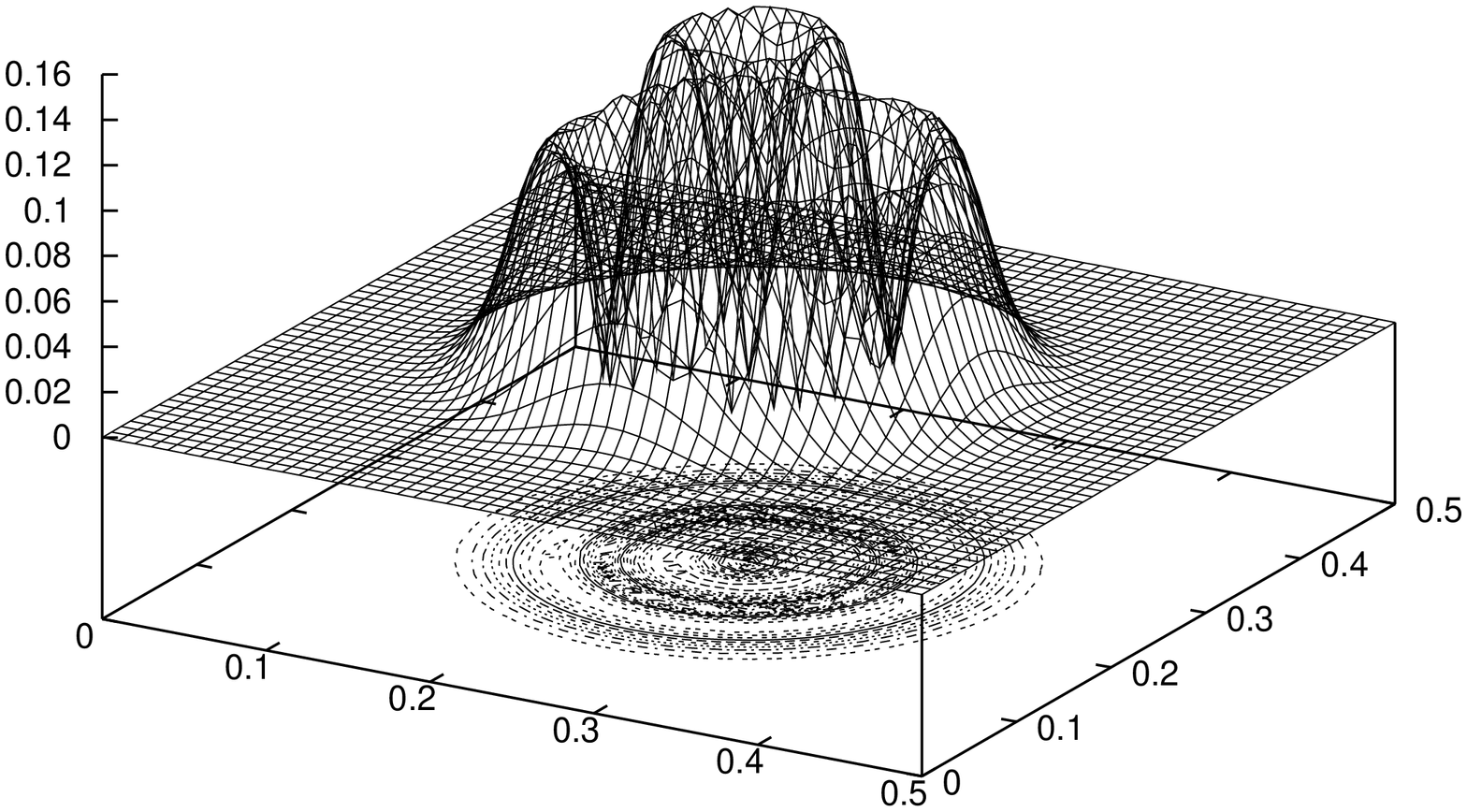}\\
\includegraphics[width=5cm,angle=-90]{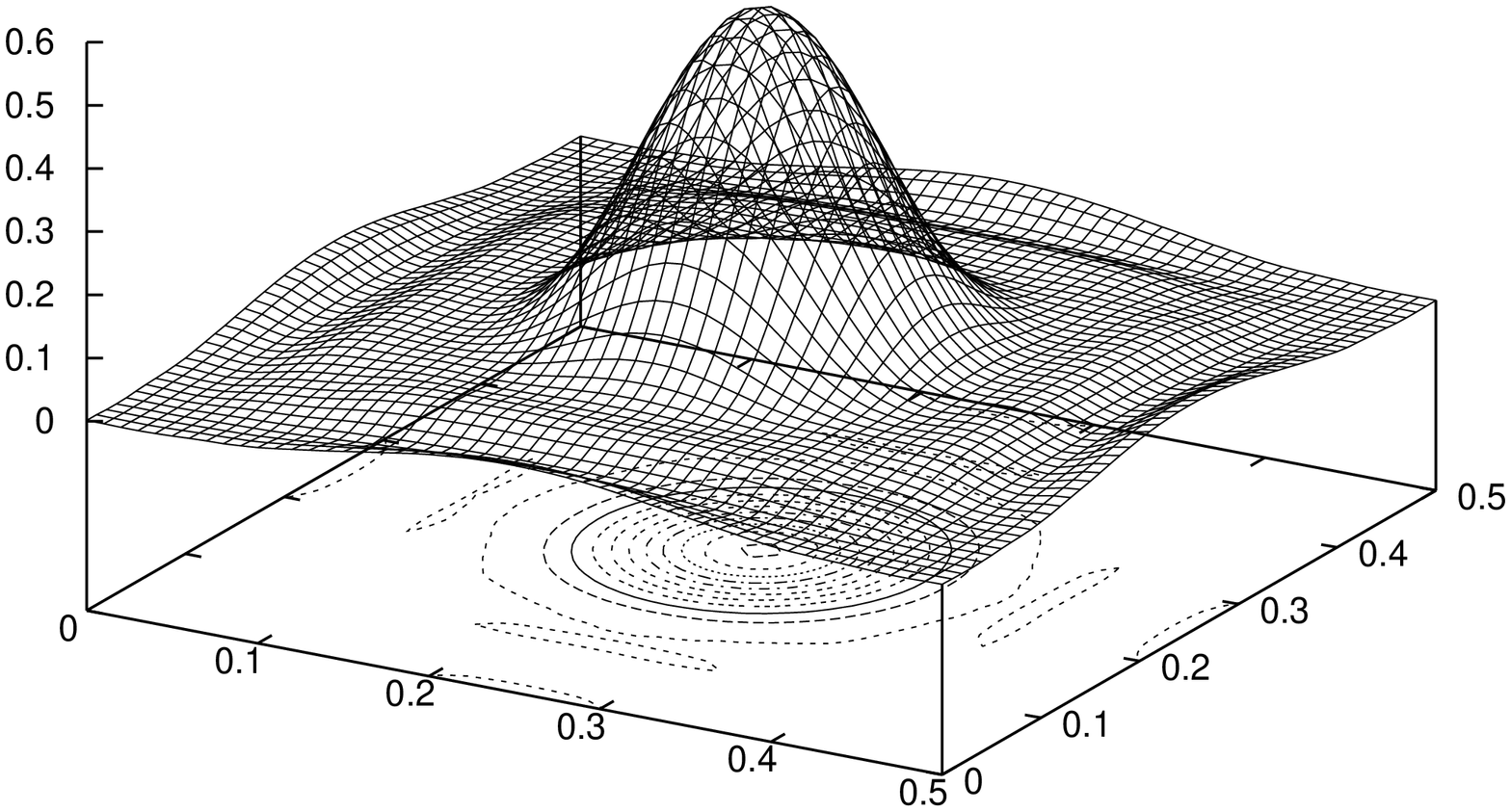}&
\hspace{-1.5cm}
\includegraphics[width=5cm,angle=-90]{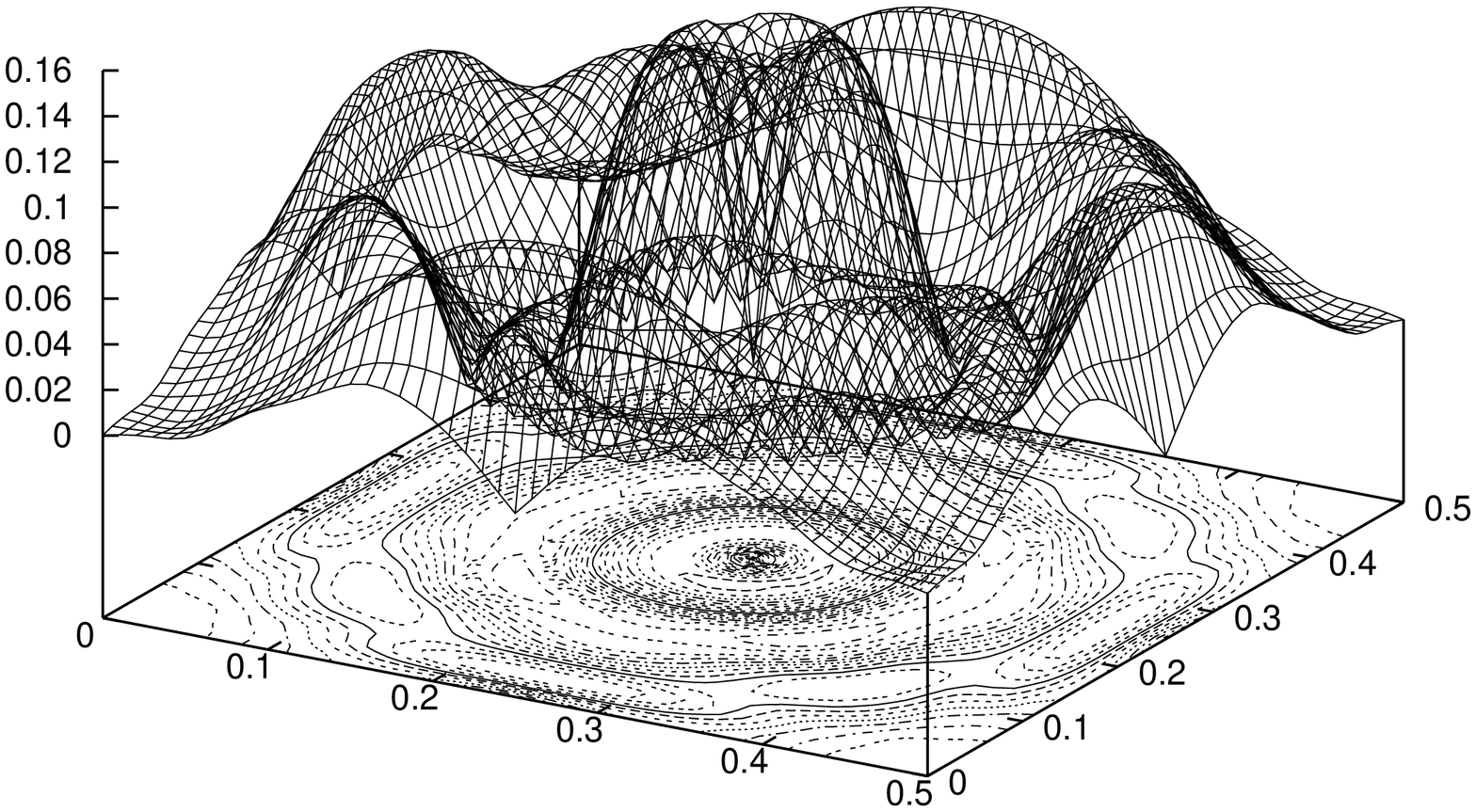}
\end{tabular}
\caption{ \label{iso_rj3} Position (left column) and current density (right column)
at $T=0.05 \sec$ for (resp. from the top) $\eps=0, 0.001, 0.01$ and $0.1$ with $\alpha=0.01$.}
\end{center}
\end{figure}

\begin{figure}[!htp!]
\begin{center}
\includegraphics[width=6cm,angle=-90]{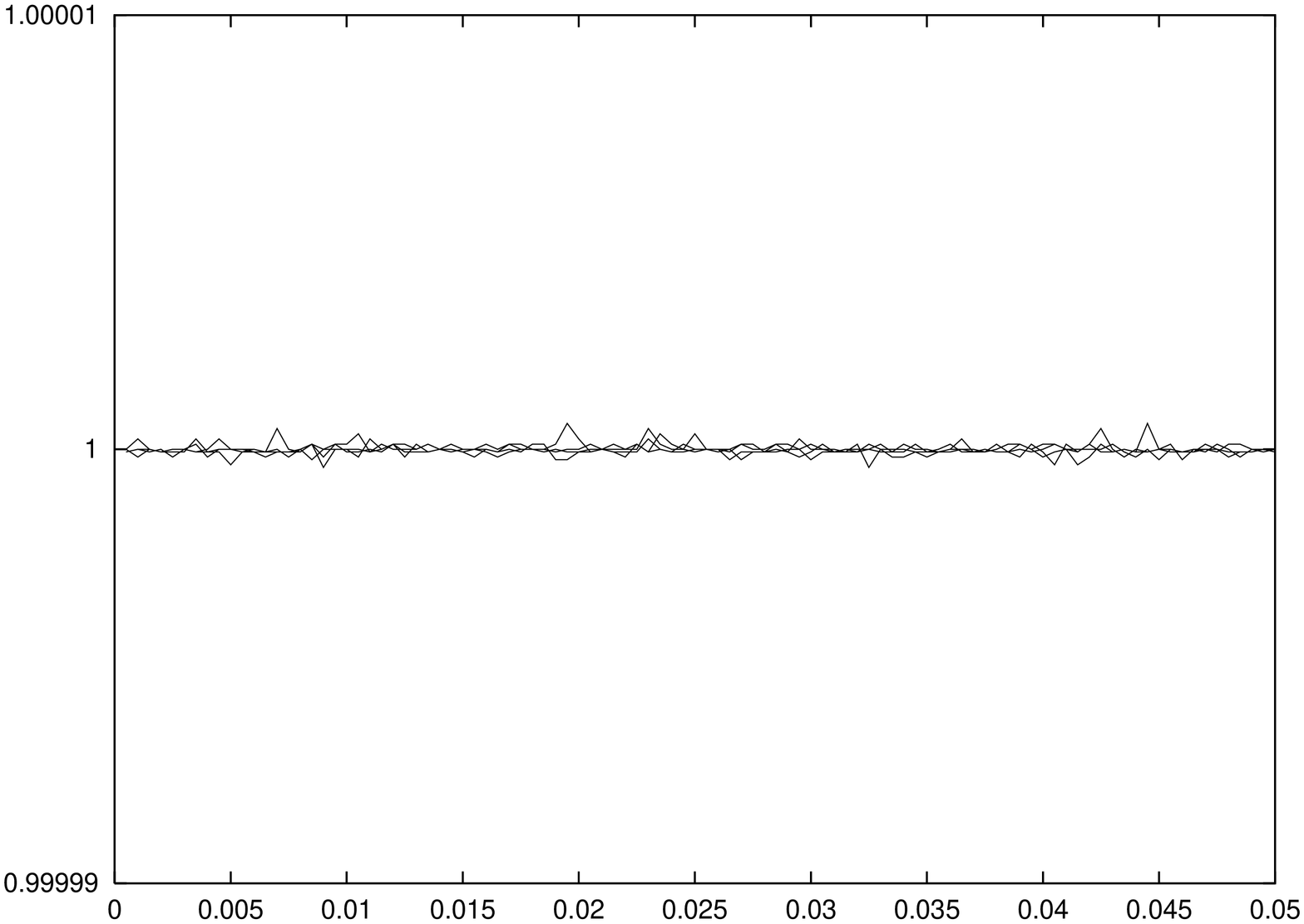}
\includegraphics[width=6cm,angle=-90]{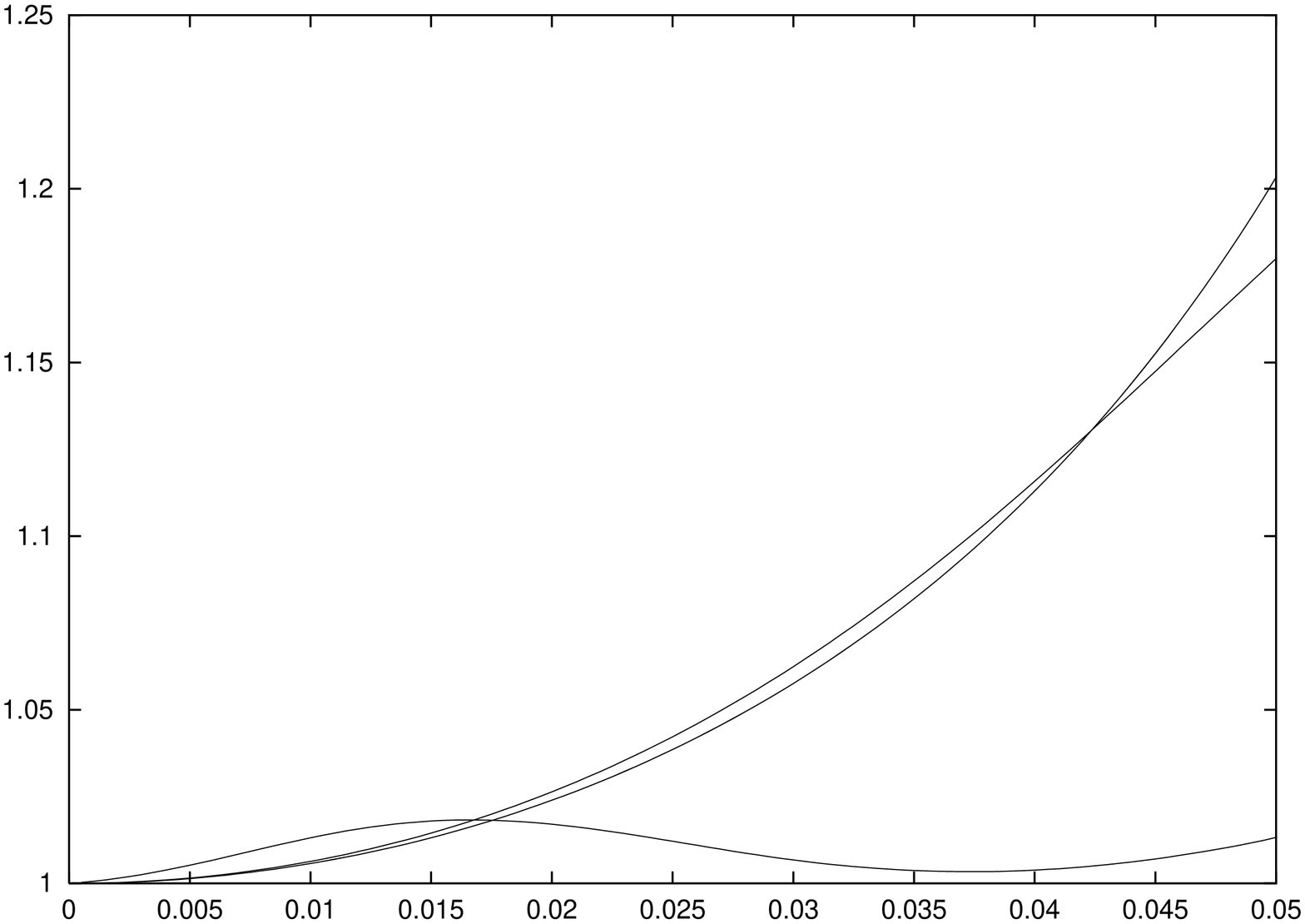}
\includegraphics[width=6cm,angle=-90]{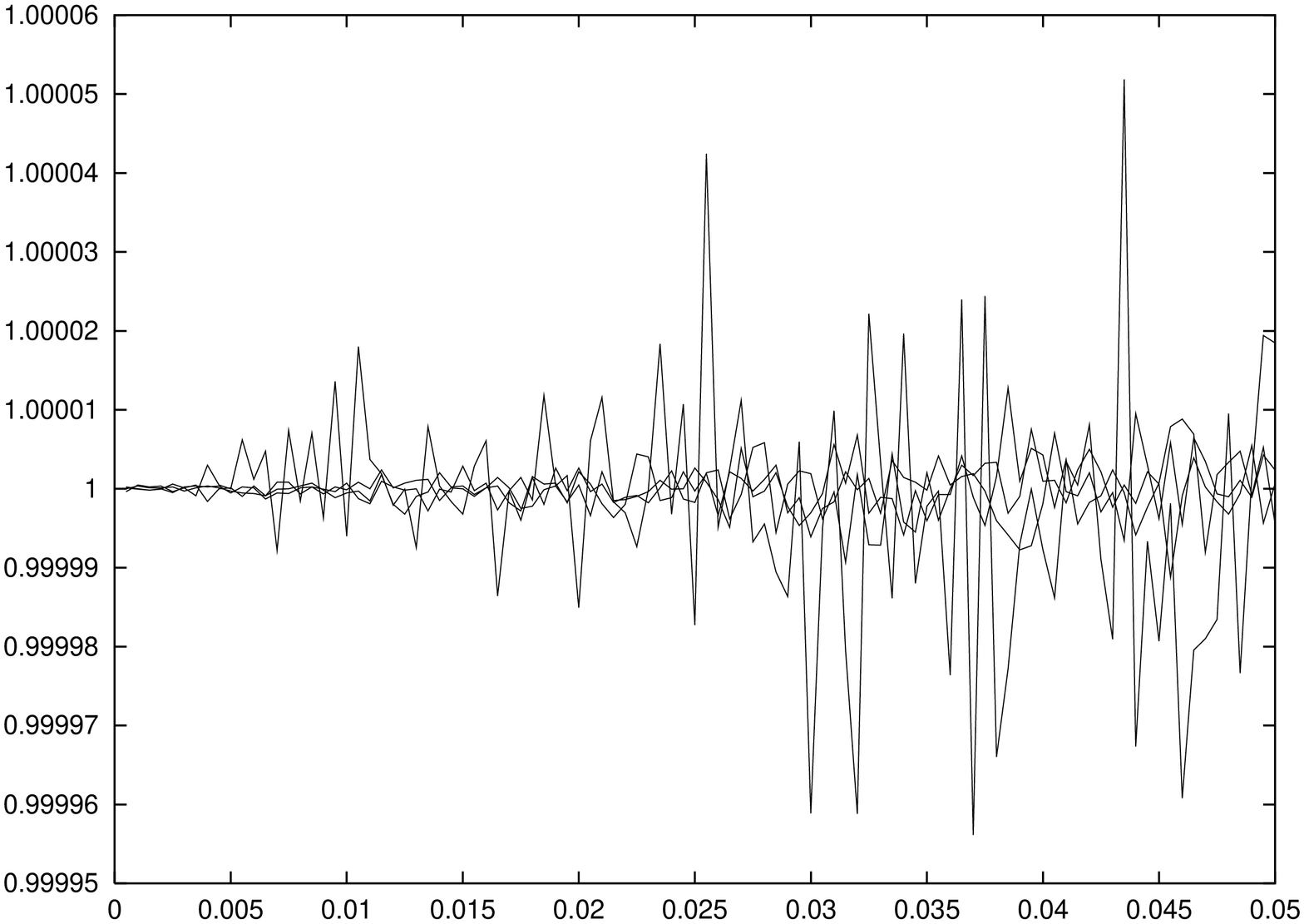}
\end{center}
\caption{ \label{res_roej3} Evolution  in time ($\sec$) of the
  constraints on the
position density, energy and sum of both components of the current
density (resp. from the top) for  $\eps=0, 0.001, 0.01$ and $0.1$
for an initial condition having sign variation in $a^\eps$.}
\end{figure}

Figures~\ref{indic_L1_3} and \ref{indic_L2_3} show
indicators \eqref{eqindic_L1} and \eqref{eqindic_L2}.

\begin{figure}[!htp!]
\begin{center}
\includegraphics[width=7cm,angle=-90]{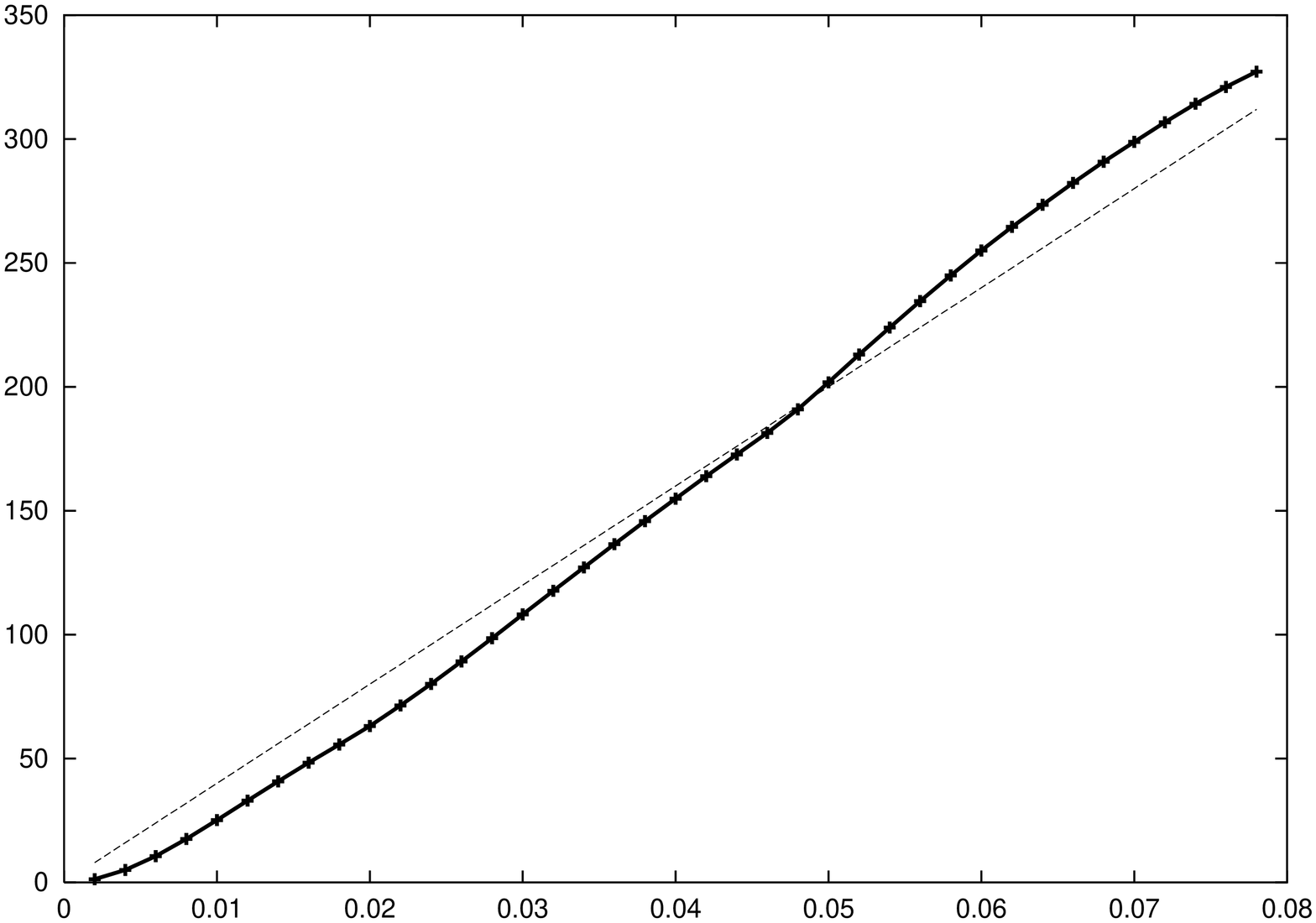}
\end{center}
\caption{ \label{indic_L1_3} Linear dependency of (\ref{eqindic_L1})  at $T=0.05 \sec$ with respect to $\eps$
with an initial condition having sign variation in $a^\eps$.
}
\end{figure}

\begin{figure}[!htp!]
\begin{center}
\includegraphics[width=7cm,angle=-90]{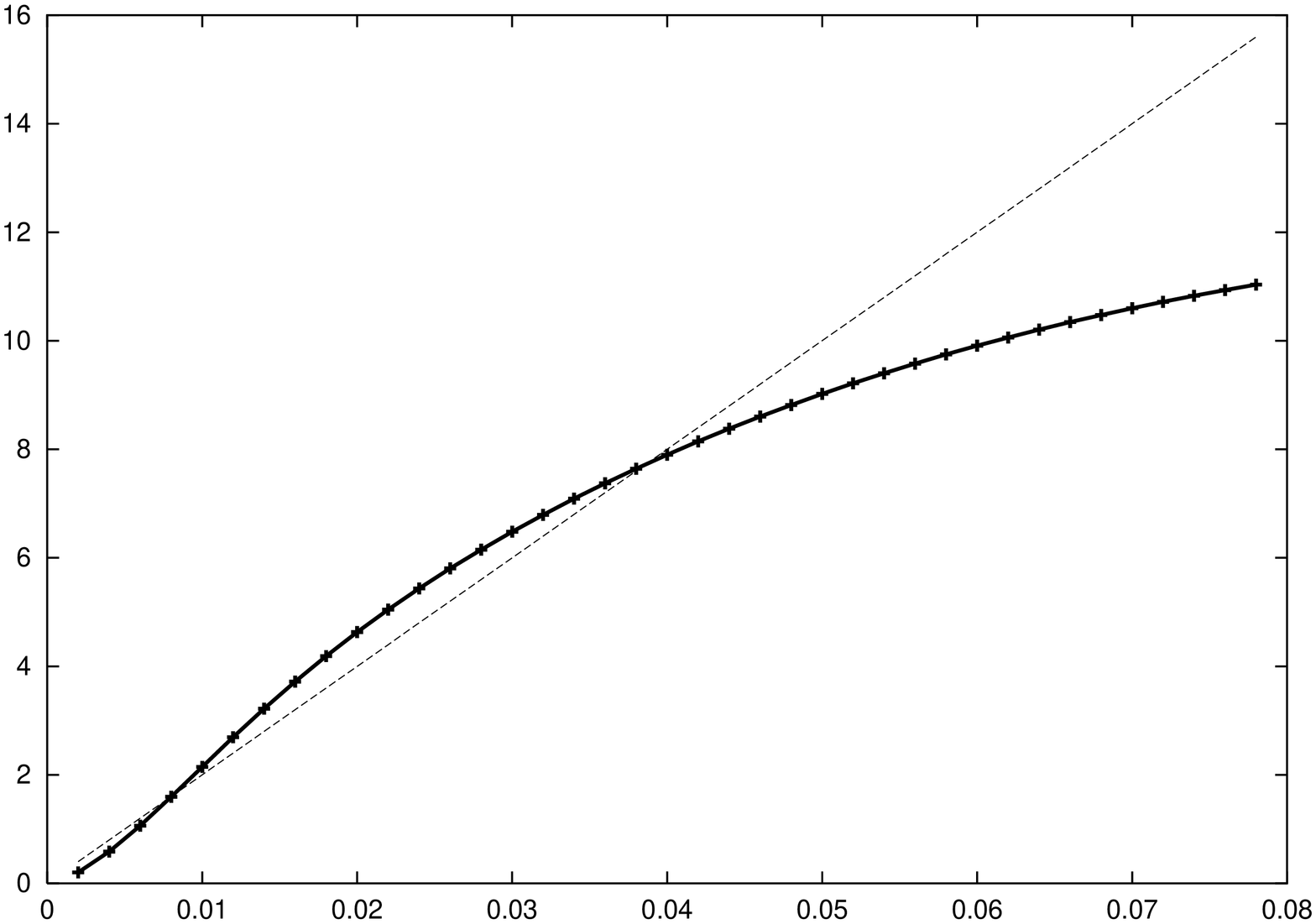}
\end{center}
\caption{ \label{indic_L2_3} Linear dependency of  \eqref{eqindic_L2}
  at $T=0.05 \sec$ with respect to $\eps$ 
with an initial condition having sign variation in $a^\eps$.}
\end{figure}

\smallbreak

To see the behavior of the approach after singularities have
formed in the Euler equation
(for $\eps=0$),  we show in Figure~\ref{iso_rj5} the solution at
$T=0.15 \sec$: the solution for $\eps=0$  has become singular, while the
solution for $\eps>0$ seems to remain smooth. In this case, the
meaning of the figure for 
$\eps=0$ is unclear, since we know that the scheme has dealt with a
singularity. On the other hand,
rapid oscillations  have appeared at least for $\eps=0.1$. For
$\eps=0.001$, the map is not very smooth, as if some oscillations were
not resolved. Recall
however that the time step and the mesh size are independent of
$\eps$: in the presence of rapid oscillations, this strategy has
proven unefficient in \cite{BJM03}, as recalled in
\S\ref{sec:seminum}. This may very well be the case in
Figure~\ref{iso_rj5}.

\begin{figure}[!htp!]
\begin{center}
\begin{tabular}{ll}
\vspace{-1.5cm}
\includegraphics[width=5cm,angle=-90]{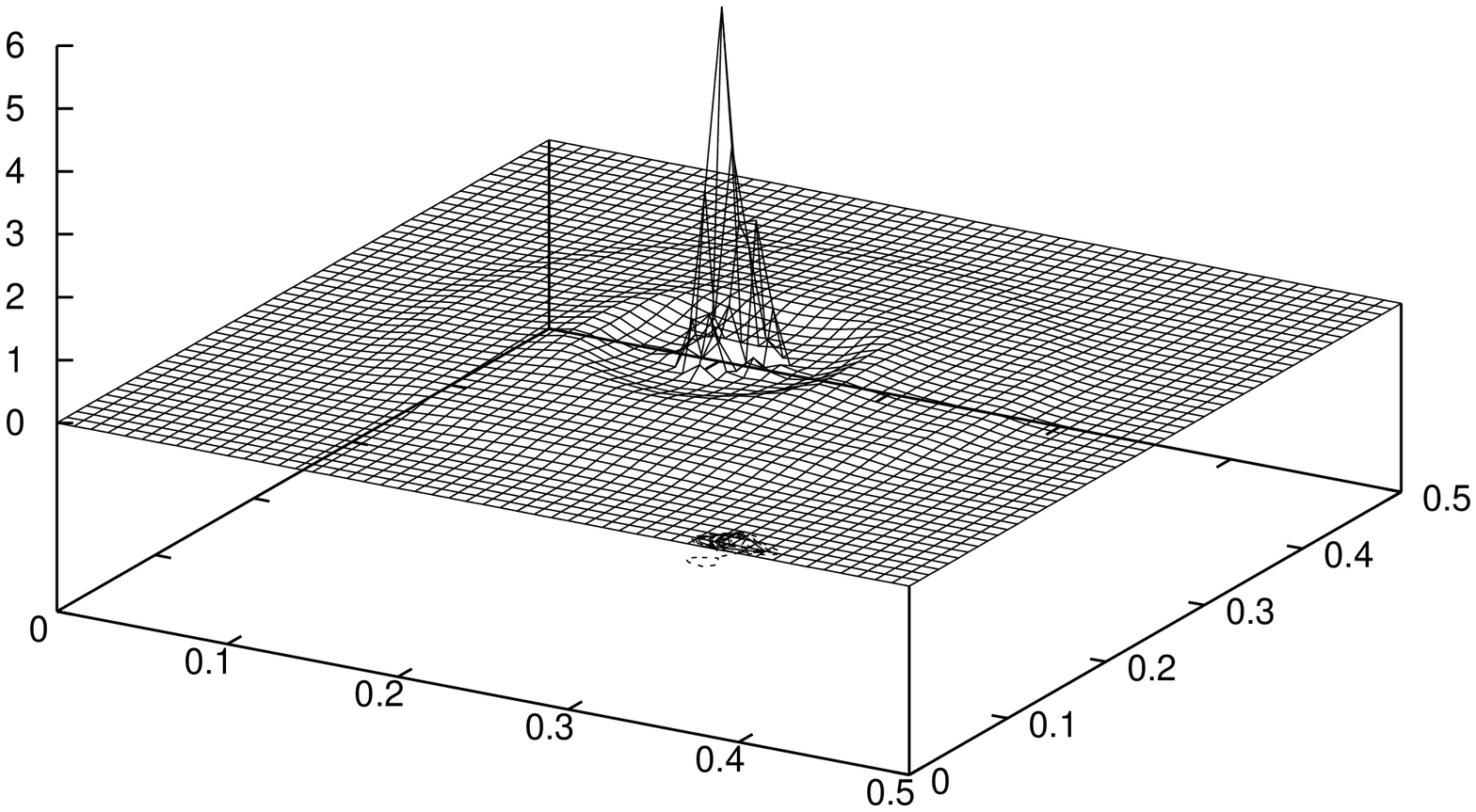}&
\hspace{-1.5cm}
\includegraphics[width=5cm,angle=-90]{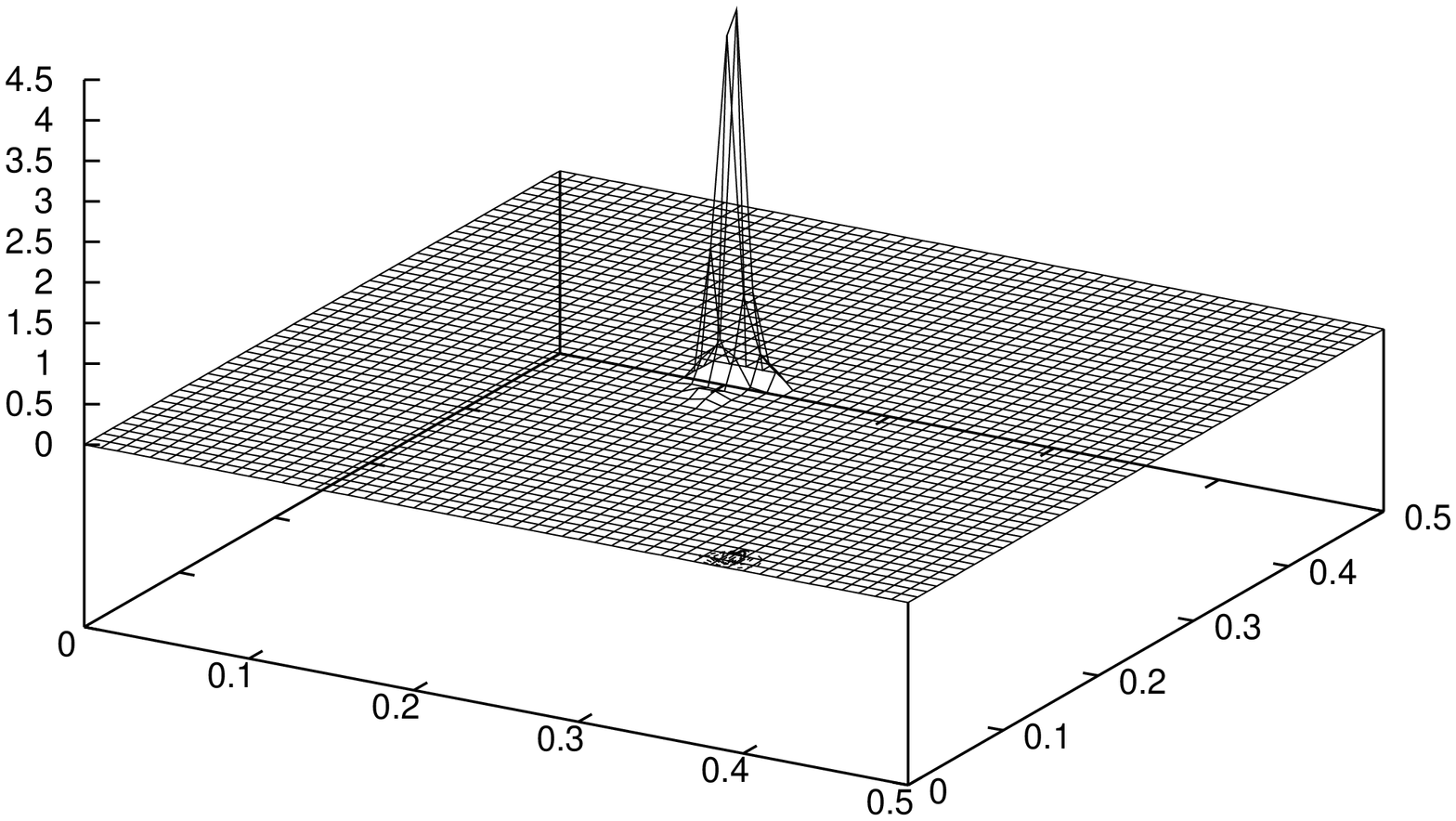}\\
\vspace{-1.5cm}
\includegraphics[width=5cm,angle=-90]{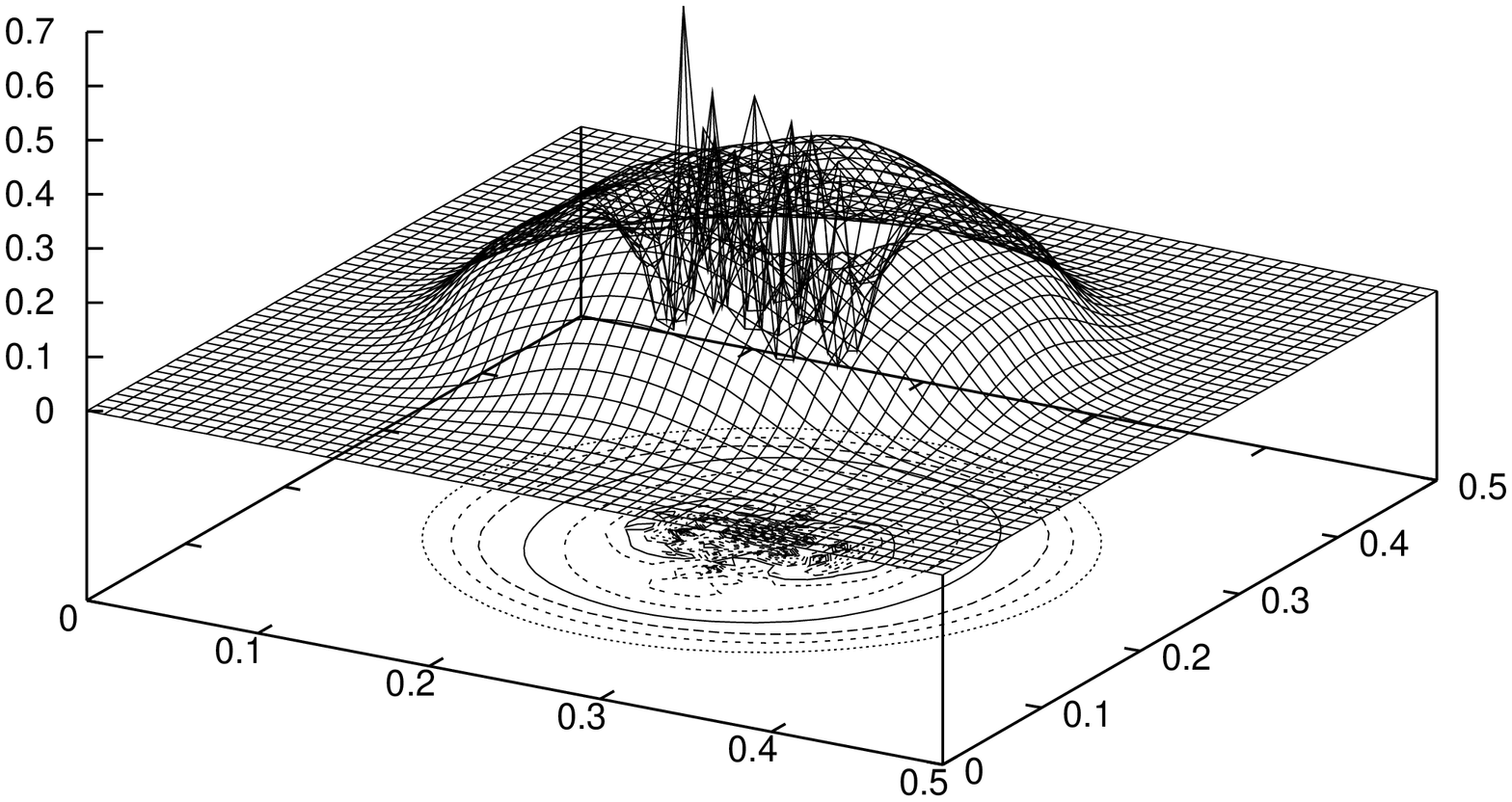}&
\hspace{-1.5cm}
\includegraphics[width=5cm,angle=-90]{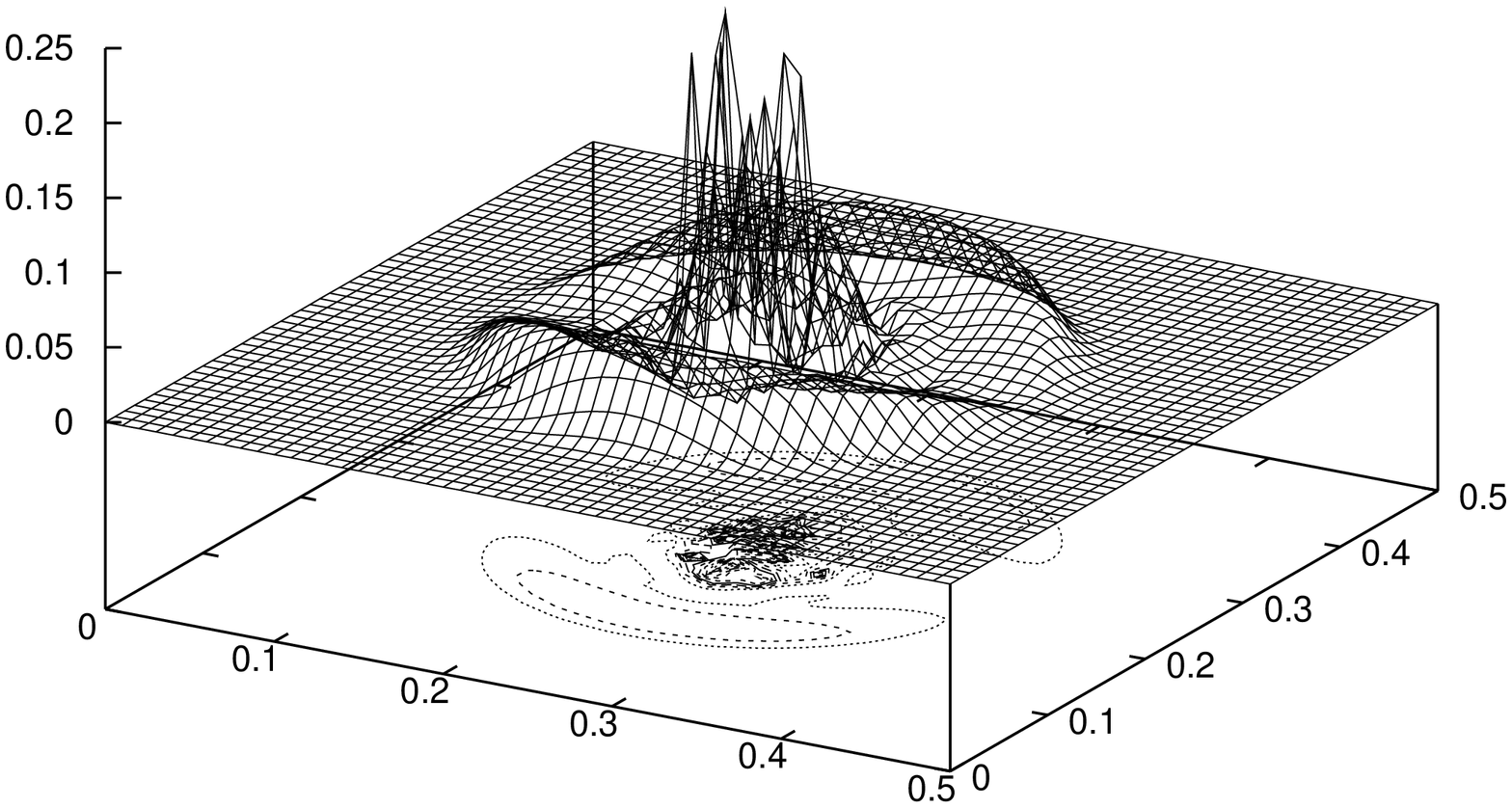}\\
\vspace{-1.5cm}
\includegraphics[width=5cm,angle=-90]{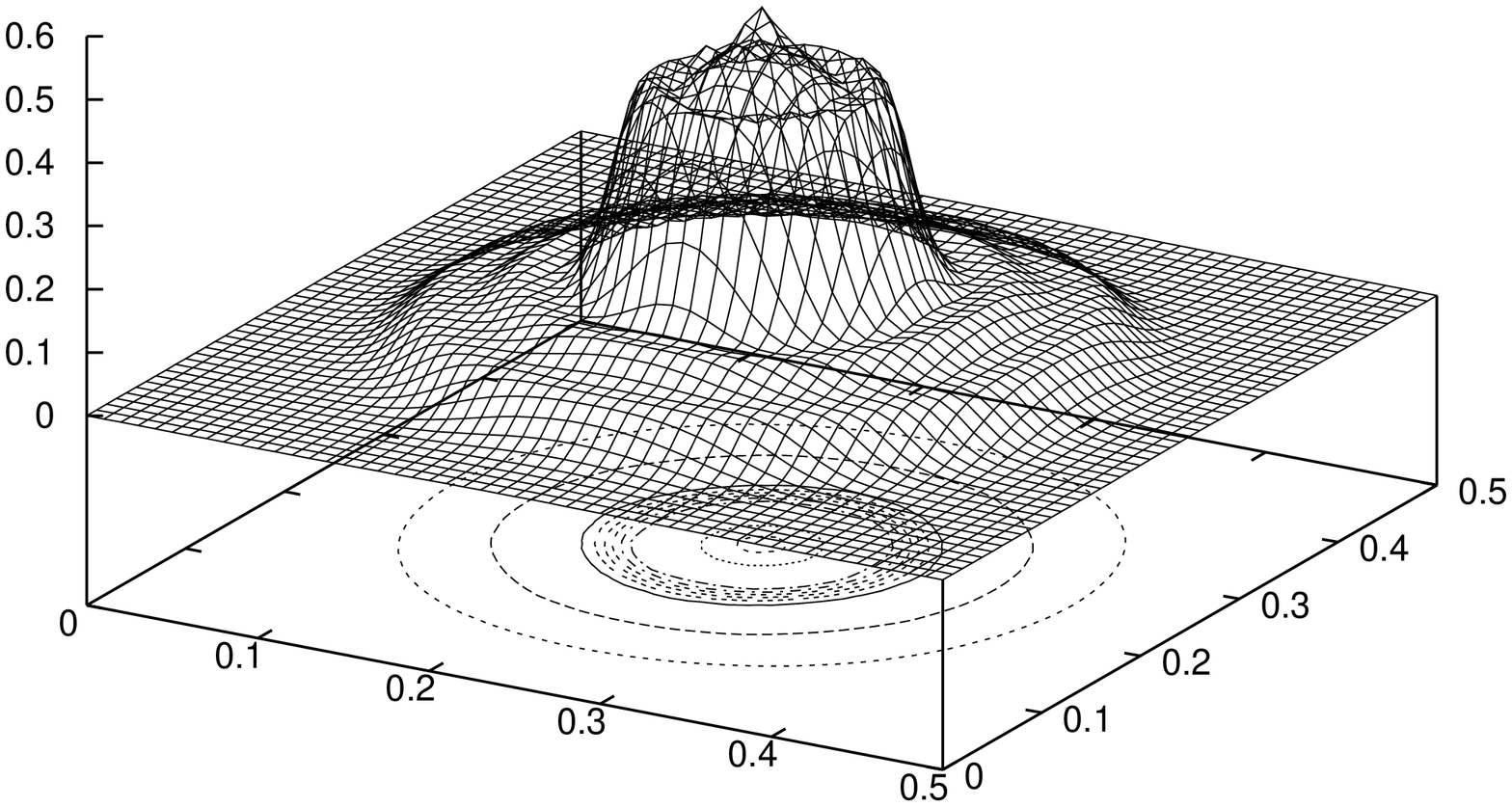}&
\hspace{-1.5cm}
\includegraphics[width=5cm,angle=-90]{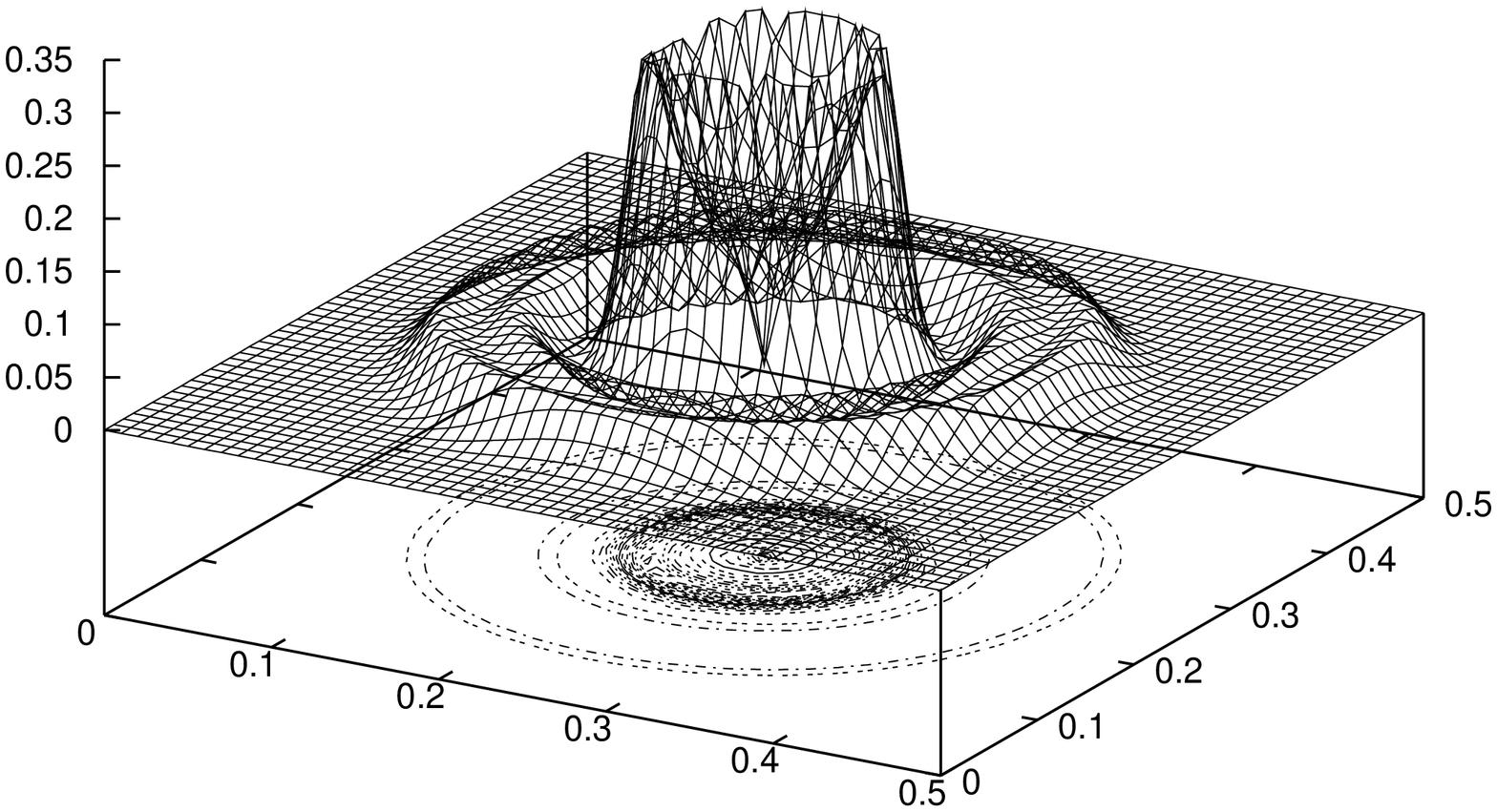}\\
\includegraphics[width=5cm,angle=-90]{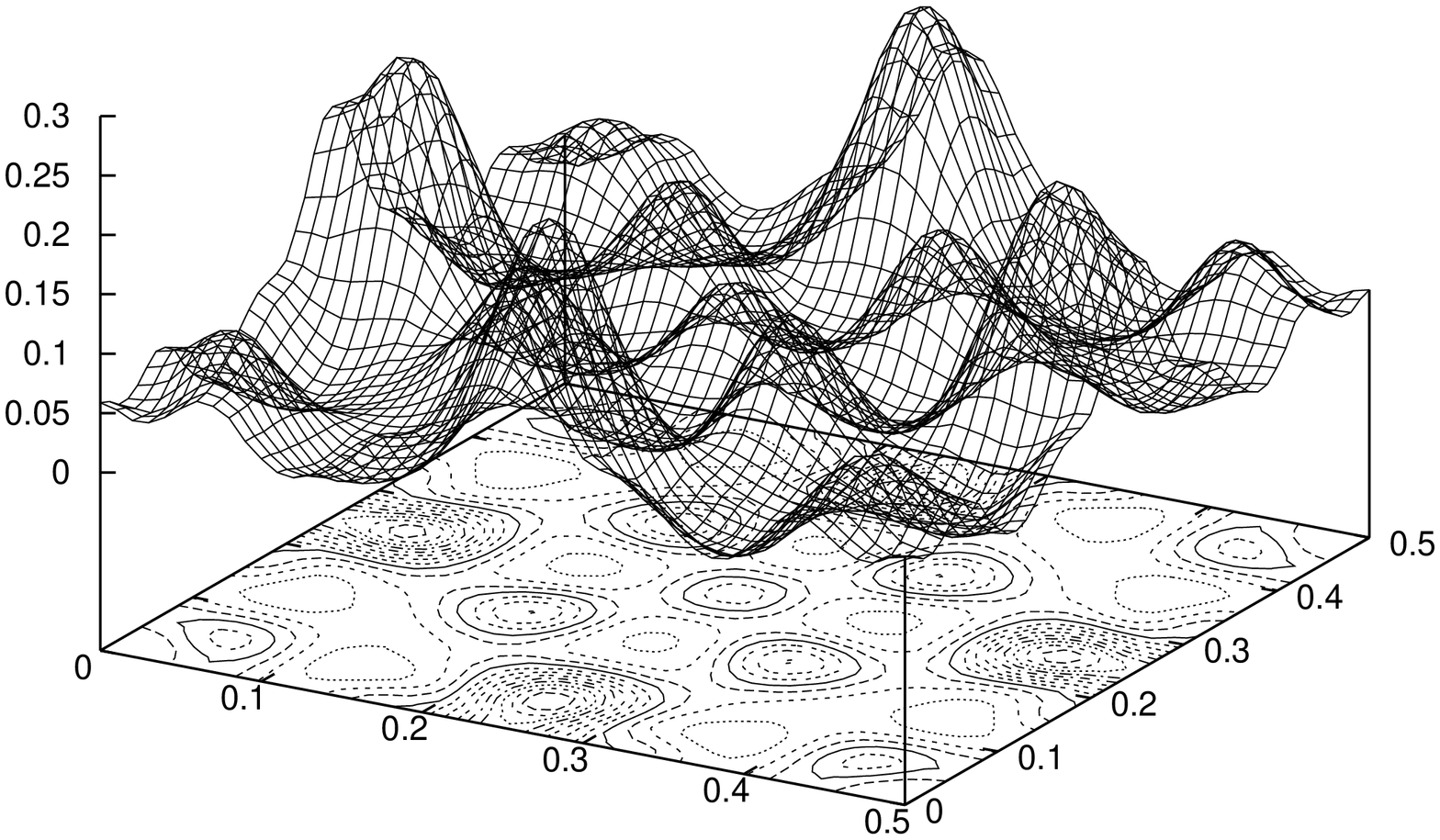}&
\hspace{-1.5cm}
\includegraphics[width=5cm,angle=-90]{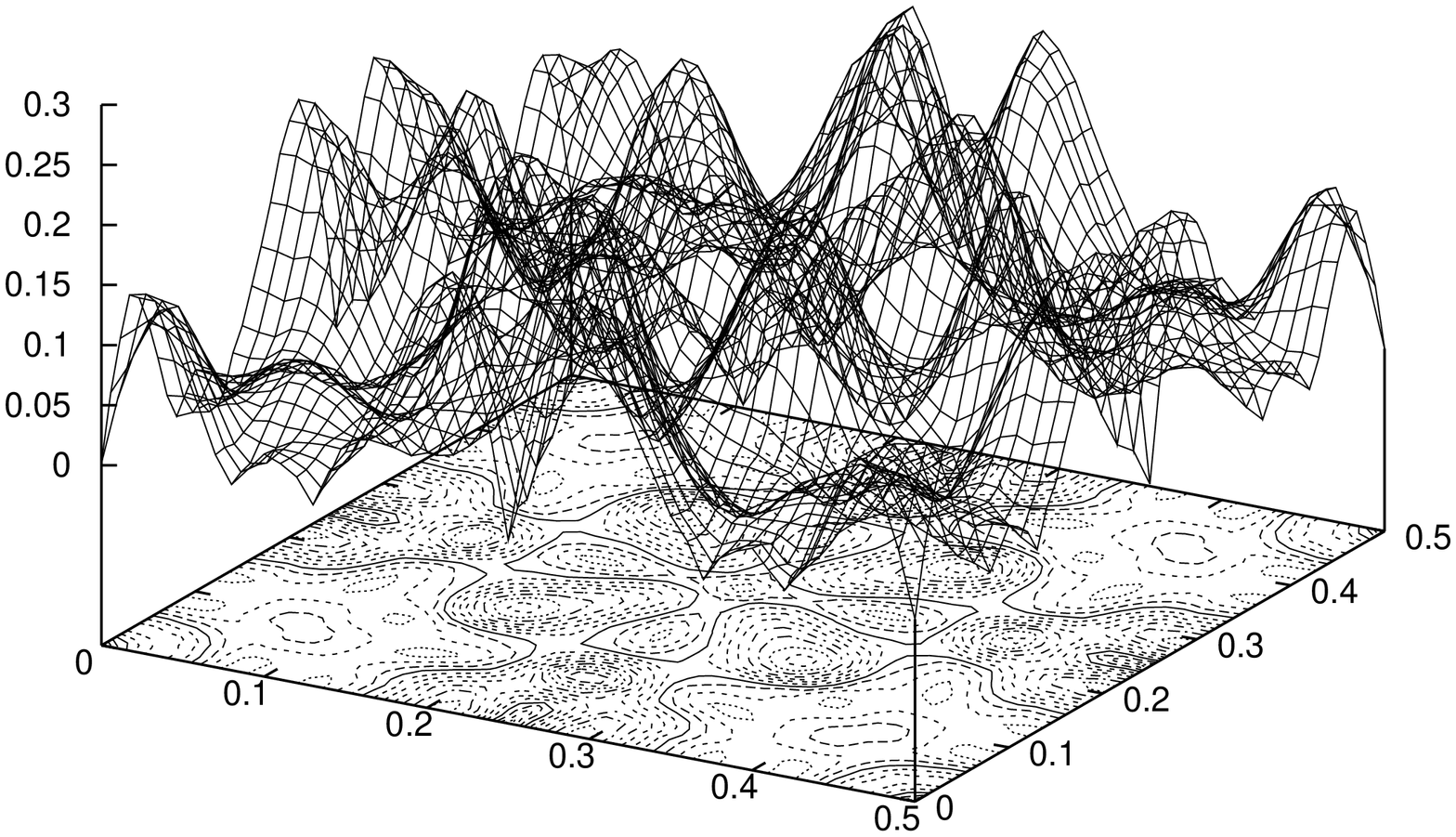}
\end{tabular}
\caption{\label{iso_rj5} Position (left column) and current density
  (right column) 
at $T=0.15 \sec$ for (resp. from the top) $\eps=0, 0.001, 0.01$ and
$0.1$ nearly after the solution blows up for $\eps=0$.} 
\end{center}
\end{figure}

\section{Conclusion}
\label{sec:concl}

We have presented a numerical implementation to compute the solution
of the system \eqref{eq:grenierv}, which is a way to solve
the nonlinear Schr\"odinger equation that is asymptotic
preserving in the semiclassical limit. To reconstruct the wave
function $u^\eps$, the phase $\phi^\eps$ can be computed by a simple
time integration, in view of \eqref{eq:backtophi}.
\smallbreak

The scheme used in this paper is 
explicit, and is therefore rather cheap on the computational level. It
preserves the $L^2$-norm of the solution to the nonlinear
Schr\"odinger equation, and can be adapted in order to conserve the
momentum as well, thanks to simple projections based on rescaling. On
the other hand, the energy is not conserved.  
\smallbreak

With
mesh sizes and time steps which are independent of the Planck constant
$\eps$, we retrieve moreover the main two quadratic observables (position and
current densities) in the semiclassical limit $\eps\to 0$, and before
singularities are formed in the limiting Euler equation, up to an
error of order $\O(\eps)$, as predicted by theoretical
results. The presence of vacuum (zeroes of the position density) is
not a problem in this approach; the case treated in
Section~\ref{sec:changing} is in perfect agreement with this
theoretical result. 
\smallbreak

Finally, these experiments suggest that once the solution to the Euler
equation has developped singularities, the solution to
\eqref{eq:grenierv} may remain smooth, while it becomes rapidly
oscillatory. It is possibly $\eps$-oscillatory in the sense of \cite{GMMP},
but the existence of intermediary scales of oscillation cannot be a
priori ruled out. We do not claim to observe any quantitative result
for post-breakup time, but rather a qualitative phenomenon: a
refinement of time step and mesh size would be needed in view of a
more reliable result after the breakup time. This aspect goes beyond
the scope of the present paper.   
\bibliographystyle{amsplain}
\bibliography{biblio}

\end{document}